\documentclass[11pt]{article}
\usepackage[margin=2.5cm]{geometry}
\usepackage{amsfonts}
\usepackage{amsmath}
\usepackage{amsthm}
\usepackage{amssymb}
\usepackage{amsmath,amscd}
\usepackage{graphicx}
\usepackage{enumerate}
\usepackage{subcaption}
\graphicspath{ {./Figure/} }
\usepackage[small,nohug,heads=vee]{diagrams}
\diagramstyle[labelstyle=\scriptstyle]
\usepackage[pdffitwindow=false,
colorlinks=true,linkcolor=black,urlcolor=black,citecolor=black]{hyperref}
\usepackage{tikz}

\newtheorem{theorem}{Theorem}[section]
\newtheorem{lemma}[theorem]{Lemma}
\theoremstyle{definition}
\newtheorem{definition}[theorem]{Definition}

\newtheorem{proposition}[theorem]{Proposition}

\newtheorem{remark}[theorem]{Remark}
\numberwithin{equation}{section}
\newcommand{\spann}{\mathrm{span}}

\newcommand{\spin}{\mathrm{Spin}}

\newcommand{\trace}{\mathrm{tr}}

\newcommand{\hmo}{\mathrm{Hom}}

\newcommand{\ricci}{\mathrm{Ric}}

\begin{document}
\title {Invariant Ricci-flat Metrics of Cohomogeneity One with Wallach Spaces as Principal Orbits}
\author{Hanci Chi}
\date{\today}
\maketitle

\begin{abstract}
We construct a continuous 1-parameter family of smooth complete Ricci-flat metrics of cohomogeneity one on vector bundles over $\mathbb{CP}^2$, $\mathbb{HP}^2$ and $\mathbb{OP}^2$ with respective principal orbits $G/K$ the Wallach spaces $SU(3)/T^2$, $Sp(3)/(Sp(1)Sp(1)Sp(1))$ and $F_4/\spin(8)$. Almost all the Ricci-flat metrics constructed have generic holonomy. The only exception is the complete $G_2$ metric discovered in \cite{bryant1989construction}\cite{gibbons1990Einstein}. It lies in the interior of the 1-parameter family on $\bigwedge_-^2\mathbb{CP}^2$.
All the Ricci-flat metrics constructed have asymptotically conical limits given by the metric cone over a suitable multiple of the normal Einstein metric on $G/K$.
\end{abstract}

\tableofcontents

\section{Introduction}
\subsection{Background and Main Result}
A Riemannian manifold $(M,g)$ is \emph{Ricci-flat} if its Ricci curvature vanishes:
\begin{equation}
\label{general Ricci-flat}
\ricci(g)=0.
\end{equation}
A Ricci-flat manifold is the Euclidean analogy of a vacuum solution of the Einstein field equations.

In this article, we study complete noncompact Ricci-flat manifolds of cohomogeneity one.
A Riemannian manifold $(M,g)$ is of cohomogeneity one if a Lie Group $G$ acts isometrically on $M$ such that the principal orbit $G/K$ is of codimension one. The Ricci-flat condition \eqref{general Ricci-flat} is then reduced to a system of ODEs.

Many examples of cohomogeneity one Ricci-flat metrics have special holonomy. These include the first example of an inhomogeneous Einstein metric, which is also a K\"ahler metric. It was constructed in \cite{calabi1975construction} on a non-compact open set of $\mathbb{C}^n$. A complete Calabi--Yau metric was constructed on $T^*\mathbb{S}^2$ independently in \cite{calabi1979metriques}\cite{eguchi1979self}. The construction was generalized to $T^*\mathbb{CP}^n$ in \cite{calabi1979metriques} and those Ricci-flat metrics are hyper-K\"ahler. Cohomogeneity one K\"ahler--Einstein metrics were constructed on complex line bundles over a product of compact K\"ahler--Einstein manifolds in  \cite{bergery1982nouvelles}\cite{KEcoho1}. Complete metrics with $G_2$ or $\spin(7)$ holonomy can be found in \cite{bryant1989construction}\cite{gibbons1990Einstein} \cite{cvetivc2002cohomogeneity}\cite{cvetivc2004new}\cite{foscolo_infinitely_2018}. 

Ricci-flat metrics with generic holonomy, for example, were constructed on various vector bundles in \cite{bergery1982nouvelles}\cite{bohm_inhomogeneous_1998}\cite{wang1998Einstein}\cite{chen2011examples}. It is further shown in \cite{buzano2015family}\cite{buzano_non-kahler_2015} that for infinitely many dimensions, there exist examples which are homeomorphic but not diffeomorphic. The case where the isotropy representation of the principal orbit contains exactly two inequivalent irreducible summands was studied in \cite{bohm_inhomogeneous_1998}\cite{wink2017cohomogeneity}. In this article, we consider examples with three inequivalent summands. Specifically, let $(G,H,K)$ be one of 
\begin{equation}
\label{Cases}
\begin{split}
&\text{I. } (SU(3),S(U(2)U(1)),S(U(1)U(1)U(1))),\\ &\text{II. } (Sp(3),Sp(2)Sp(1),Sp(1)Sp(1)Sp(1)),\\
&\text{III. }(F_4,\spin(9),\spin(8)).
\end{split}
\end{equation}
For these triples, we construct Ricci-flat metrics on the corresponding cohomogeneity one vector bundles $M$ with unit sphere bundle 
$H/K\hookrightarrow G/K\rightarrow G/H.$ 
The singular orbits $G/H$'s are respectively 
$\mathbb{CP}^2$, $\mathbb{HP}^2$ and $\mathbb{OP}^2$. 
The principal orbits $G/K$'s are \emph{Wallach spaces}. They appeared explicitly in Wallach classification of even dimensional homogeneous manifolds with positive sectional curvature  \cite{wallach1972compact}.
Throughout this paper, the letters $j,k,l$ will denote three distinct numbers in $\{1,2,3\}$ whenever more than one of them appear in a formula together. Let $d=\dim(H/K)$ and $n=\dim(G/K)$. 
As will be shown in Section \ref{Smoothness}, each $M$ is in fact an irreducible (sub)bundle of $\bigwedge^{d}_-T^*(G/H)$.

In all three cases, we can rescale the normal metric on $G/K$ to a metric $Q$, whose restriction on $H/K$ is the standard metric with constant sectional curvature 1. Take $Q$ as the background metric for $G/K$. As will be shown in Section \ref{Cohomogeneity One Ricci-flat Equation}, the isotropy representation $\mathfrak{g}/\mathfrak{k}$ has $\mathbb{Z}_3$-symmetry among its three inequivalent irreducible summands. By Schur's lemma, any $G$-invariant metric on $G/K$ has the form
\begin{equation}
\label{invariant metric on g/k}
g_{G/K}=f_1^2\left.Q\right|_{\mathfrak{p}_{1}}\oplus f_2^2\left.Q\right|_{\mathfrak{p}_{2}}\oplus f_3^2\left.Q\right|_{\mathfrak{p}_{3}}
\end{equation}
for some $f_j>0$. Correspondingly, the Ricci endomorphism $r$ of $G/K$, defined by $g_{G/K}(r(\cdot),\cdot)=\ricci(\cdot,\cdot)$, has the form 
\begin{equation}
\label{ricci on principaltotal}r=r_1\left.Q\right|_{\mathfrak{p}_{1}}\oplus r_2\left.Q\right|_{\mathfrak{p}_{2}}\oplus r_3\left.Q\right|_{\mathfrak{p}_{3}},
\end{equation}
where
\begin{equation}
\label{ricci on principal}
\begin{split}
&r_j=\dfrac{a}{f_j^2}+b\left(\dfrac{f_j^2}{f_k^2f_l^2}-\dfrac{f_k^2}{f_j^2f_l^2}-\dfrac{f_l^2}{f_j^2f_k^2}\right)
\end{split}
\end{equation}
for some constants $a$ and $b$. Their values were computed in \cite{nikonorov2014classification}. We have
\begin{table}
\centering
\caption{}
\label{constant}
\begin{tabular}{ l|llll }
\hline
Case & $d$ & $n$ & $a$ & $b$\\ 
\hline
I & 2& 6& $\frac{3}{2}$ & $\frac{1}{4}$\\ 
\hline
II & 4& 12& $4$ & $\frac{1}{2}$\\ 
\hline
III & 8& 24&$9$ & $1$ \\ 
\hline
\end{tabular}.
\end{table}

\begin{remark}
\label{scalr of spher basic ineq}
A basic observation on $a$ and $b$ is $a-2b=d-1.$ This is not surprising since $Q$ is the sectional curvature 1 metric on $\mathbb{S}^{d}$. Another observation is $
a-6b\geq 0,$
where the equality is achieved in Case I. These observations are frequently used in this article, especially in Section \ref{Completensese} and Section \ref{Entrance Zone}.
\end{remark}

Note that all three possible $\frac{f_j^2}{f_k^2f_l^2}$'s appear in \eqref{ricci on principal}. An important motivation for our choices of principal orbits to consider is to study the complications that arise from the simultaneous presence of the terms $\frac{f_1^2}{f_2^2f_3^2}$,$\frac{f_2^2}{f_1^2f_3^2}$ and $\frac{f_3^2}{f_1^2f_2^2}$. If two of $f_j$'s are identical, say $f_2\equiv f_3$, the Ricci endomorphism takes a simpler form, with $r_1=\frac{a-2b}{f_1^2}+b\frac{f_1^2}{f_2^4}$ and $r_2\equiv r_3=\frac{a}{f_2^2}-b\frac{f_1^2}{f_2^4}$. The Ricci-flat ODE system for this special case then reduces to the one for $\mathfrak{g}/\mathfrak{k}$ with two inequivalent irreducible summands considered in \cite{bohm_inhomogeneous_1998}\cite{wink2017cohomogeneity}. It is noteworthy that the functional $\widehat{\mathcal{G}}$ introduced in \cite{bohm_inhomogeneous_1998} does not have any positive real root for Case I. Nevertheless, the two summands case can be viewed as the subsystem of the ODE system studied in this article. The invariant compact set constructed in Section \ref{Completensese} can be used to prove the existence of complete Ricci-flat metric for this special case. With the condition $f_2\equiv f_3$ relaxed, we prove the following theorem.
\begin{theorem}
\label{main 1}
There exists a continuous 1-parameter family of non-homothetic complete smooth invariant Ricci-flat metrics on each $M$.
\end{theorem}

\begin{remark}
Ricci-flat metrics constructed in Case II and Case III all have generic holonomy. In Case I, the 1-parameter family of smooth Ricci-flat metrics contains in its interior the complete smooth $G_2$ metric that was first constructed in \cite{bryant1989construction}\cite{gibbons1990Einstein}. The other metrics in that family all have generic holonomy. Therefore, for $M$ in Case I, the moduli space $\mathcal{M}_{G_2}$ of $G_2$ metric is \emph{not} isolated in $\mathcal{M}_0$ the moduli space of Ricci-flat metric in the $C^0$ sense. Such a phenomenon cannot occur on a simply connected spin closed manifold, for example, by Theorem 3.1 in \cite{wang1991preserving}.
\end{remark}

\begin{definition}
\label{AC}
Let $(N,g_N)$ and $(M,g_M)$ be Riemannian manifolds of respective dimension $n$ and $n+1$. Let $t$ be the geodesic distance from some point on $M$. Then $M$ has one \emph{asymptotically conical (AC) end} if there exists a compact subset $\check{M}\subset M$ such that $M\backslash \check{M}$ is diffeomorphic to $(1,\infty)\times N$ with $g_M=dt^2+t^2g_N+o(1)$ as $t\to \infty.$
\end{definition}

With further analysis on the asymptotic behavior of Ricci-flat metrics in Theorem \ref{main 1}, we are able to prove the following:
\begin{theorem}
\label{main 2}
Each Ricci-flat metric in Theorem \ref{main 1} has an AC end with limit the metric cone over a suitable multiple of the normal Einstein metric on $G/K$.
\end{theorem}

\begin{remark}
In Case I, the normal Einstein metric on the principal orbit $SU(3)/T^2$ admits a (strict) nearly K\"ahler structure. Hence the metric cone over $G/K$ is the singular $G_2$ metric which was first constructed in \cite{bryant1987metrics}. The other two principal orbits, however, do not admit (strict) nearly K\"ahler structure  \cite{davila_homogeneous_2012}.
\end{remark}

%These techniques can potentially be modified to approach the same question on a larger class of manifolds. We can look into vector bundles with principal orbits as generalized Wallach spaces since these spaces have similar algebraic structure in their isotropy representation \cite{nikonorov2014classification}. Note that vector bundles of this type include $T\mathbb{S}^{n+1}$, whose principal orbit is $SO(n+2)/SO(n)$, a generalized Wallch space. This class of manifold is shown to carry various special holonomies. In particular, it admits hyperk\"ahler structure\cite{calabi1979metriques}\cite{eguchi1979self}. It is also known in \cite{stenzel1993ricci} that the Calabi--Yau metric exists on $T\mathbb{S}^n$. By \cite{dancer2002Ricci-flat}, complete negative K\"ahler Ricci-flat metrics also exists on $T\mathbb{S}^n$. Our method may help to answer the following two questions: Can we eliminate the K\"ahler condition for a Ricci-flat metric on $T\mathbb{S}^n$? How do moduli spaces of Calabi--Yau metric and hyperk\"ahler metric lie in $\mathcal{M}_0$?

\subsection{Organization}
This paper is structured as followings. In Section \ref{Cohomogeneity One Ricci-flat Equation} and \ref{Smoothness}, we discuss some details of the geometry of the cohomogeneity one manifolds $M$. Based on the work in \cite{eschenburg2000initial}, we reduce \eqref{general Ricci-flat} to a system of ODEs \eqref{Ricci-flat1} with a conservation law \eqref{old conservation law}. A $G$-invariant Ricci-flat metric around $G/H$ is hence represented by an integral curve defined on $[0,\epsilon)$. We derive the condition for smooth extension to $G/H$ using Lemma 1.1 in \cite{eschenburg2000initial}. If in addition, the integral curve is defined on $[0,\infty)$, the corresponding Ricci-flat metric is complete. 

In Section \ref{Local Existence}, we apply the coordinate change introduced in \cite{dancer2008non}\cite{dancer2008some}. The ODE system is transformed to a polynomial one. Invariant Einstein metrics on $G/H$ and $G/K$ are transformed to critical points of the new system. We carry out linearizations at these critical points and prove the local existence of invariant Ricci-flat metrics around $G/H$. An integral curve defined on $[0,\epsilon)$ is transformed to a new one that is defined on $(-\infty,\epsilon')$ for some $\epsilon'\in\mathbb{R}$. Each integral curve represents a Ricci-flat metric on $M$ up to homothety. It is determined by a parameter $s_1$ that controls the principal curvature of $G/H$ at $t=0$. To show the completeness of the metric is equivalent to proving that the new integral curve is defined on $\mathbb{R}$.  

The proof of completeness of the metric is divided into two sections. In Section \ref{Completensese}, we construct a compact invariant set whose boundary contains critical points that represent the invariant metric on $G/H$ and the normal Einstein metric on $G/K$. The construction is almost the same for all three cases with a little difference in Case I. Section \ref{Entrance Zone} proves that as long as $s_1$ is close enough to zero, integral curves of Ricci-flat metrics enter the compact invariant set constructed in Section \ref{Completensese} in finite time, hence proving the completeness.

In Section \ref{aymptoticlimit}, we analyze the asymptotic behavior of all the Ricci-flat metrics constructed in Section \ref{Entrance Zone}. There also exist solutions to the polynomial system that represent singular Ricci-flat metrics. They are discussed in Section \ref{a digression}. Results in this article are summarized by a plot at the end.

With similar techniques introduced in Section \ref{Completensesetotal}, we can also show that there exists a 2-parameter family of Poincar\'e--Einstein metrics on each $M$. More details will appear in another upcoming article.

\textbf{Acknowledgement.} The author is grateful to his PhD supervisor, Prof. McKenzie Wang for his guidance and encouragement.

\section{Local Solution Near Singular Orbit}
\label{Local EXISTENCEEE}
\subsection{Cohomogeneity One Ricci-flat Equation}
\label{Cohomogeneity One Ricci-flat Equation}
In this section, we derive the system of ODEs whose solutions give Ricci-flat metrics of cohomogeneity one on $M$. 

Since $M$ is of cohomogeneity one, there is a $G$-diffeomorphism between $M\backslash({G/H})$ and $(0,\infty)\times G/K$. We construct a Ricci-flat metric $g$ on $M$ by setting $(0,\infty)$ as a geodesic and assigning a $G$-invariant metric $g_{G/K}$ to each hypersurface $\{t\}\times G/K$, i.e., define
\begin{equation}
\label{Ricci-flat metric on M}
g=dt^2+g_{G/K}(t)
\end{equation}
on $M$. 
 By \cite{eschenburg2000initial}, if $g_{G/K}(t)$ satisfies
\begin{equation}
\label{old Ricci-flat1}
\dot{g}_{G/K}=2g_{G/K}(L\cdot,\cdot),
\end{equation}
\begin{equation}
\label{old Ricci-flat2}
\dot{L}=-\trace(L)L+r,
\end{equation}
\begin{equation}
\label{old Ricci-flat3}
\trace(\dot{L})=-\trace(L^2),
\end{equation}
\begin{equation}
\label{old Ricci-flat4}
d(\trace (L))+\delta^{\nabla} L=0,
\end{equation}
on $(0,\epsilon)$, where $\delta^{\nabla}\colon \Omega^1(G/K,T(G/K))\rightarrow T^*(G/K)$ is the divergence operator composed with the musical isomorphism, then $g$ is a Ricci-flat metric on $(0,\epsilon)\times G/K$. 

Note that \eqref{old Ricci-flat1} provides a formula for computing $L(t)$ the shape operator of hypersurface $\{t\}\times G/K$ for each $t\in (0,\epsilon)$.
By \cite{back1986local} and \cite{eschenburg2000initial}, Equation \eqref{old Ricci-flat4} automatically holds for a $C^3$ metric satisfying \eqref{old Ricci-flat1} and \eqref{old Ricci-flat2} if there exists a singular orbit of dimension smaller than $\dim(G/K)$.
Canceling the term $\trace{(\dot{L})}$ using \eqref{old Ricci-flat2} and \eqref{old Ricci-flat3} yields the conservation law
\begin{equation}
\label{old conservation}
R-(\trace{(L)})^2+\trace{(L^2)}=0.
\end{equation}

We shall focus on deriving specific formulas for \eqref{old Ricci-flat1},\eqref{old Ricci-flat2} and \eqref{old conservation} on $M$. It requires a closer look at isotropy representations of $G/K$ and $G/H$.
We fix notations first. Each irreducible complex representation is characterized by inner products between the dominant weight and simple roots on nodes of the corresponding Dynkin diagram. We use $[a]$ for class $A_1=B_1=C_1$; $[a,b]$ for $C_2=B_2$ with the shorter root on the right end; $[a,b,c,d]$ for $B_4$ with the shorter root on the right end. Furthermore, let $\mathfrak{t}$ be the Lie algebra of $S(U(1)U(1)U(1))$. Choose $Q$-orthogonal decomposition $\mathfrak{t}=\mathfrak{t}_1\oplus\mathfrak{t}_2$, where
$$
\mathfrak{t}_1=\spann_\mathbb{R}\left\{
\begin{bmatrix}
i&&\\
&-i&\\
&&0
\end{bmatrix}
\right\},\quad \mathfrak{t}_2=\spann_\mathbb{R}\left\{
\begin{bmatrix}
i&&\\
&i&\\
&&-2i
\end{bmatrix}
\right\}.
$$
Let $\theta_j^a$ denote the complexified irreducible representation of circle generated by $\mathfrak{t}_j$ with weight $a$. We use $\Lambda_8$ and $\Delta_8^\pm$ to respectively denote the complexified standard representation and spin representations of $\spin(8)$. We use $\mathbb{I}$ to denote the trivial representation.

\begin{proposition}
The formula of $g_{G/K}$ is given by \eqref{invariant metric on g/k}. 
\end{proposition}
\begin{proof}
With $(G,H,K)$ listed in \eqref{Cases}, 
we have the following $Q$-orthogonal decomposition for $\mathfrak{g}$:
\begin{equation}
\label{decomposition}
\begin{split}
\mathfrak{g}&=\mathfrak{h}\oplus \mathfrak{q} \quad \text{as a representation of $\left.Ad(G)\right|_{H}$}\\
&=(\mathfrak{k}\oplus \mathfrak{p}_{1})\oplus (\mathfrak{p}_{2}\oplus\mathfrak{p}_{3})\quad \text{as a representation of $\left.Ad(G)\right|_{K}$}.
\end{split}
\end{equation}
Irreducible $K$-modules $\mathfrak{p}_{j}$'s are all of dimension $d$ but they are inequivalent to each other. Specifically, we have Table \ref{Kmodule}.
\begin{table}[h!]
\centering
\caption{}
\label{Kmodule}
\begin{tabular}{ l |l l l }
\hline
Case %&$\mathfrak{q}\otimes\mathbb{C}$ 
& $\mathfrak{p}_{1}\otimes \mathbb{C}$ &$\mathfrak{p}_{2}\otimes \mathbb{C}$&$\mathfrak{p}_{3}\otimes \mathbb{C}$\\
\hline
I %& $[1]\otimes\overset{3}{\mathfrak{t}}$
&$\theta_1^2\otimes\mathbb{I}$ &$\theta_1^1\otimes\theta_2^3$&$\theta_1^{-1}\otimes\theta_2^3$\\
\hline
II %& $[10]\otimes[1]$
&$[1]\otimes[1]\otimes\mathbb{I}$&$[1]\otimes\mathbb{I}\otimes[1]$&$\mathbb{I}\otimes[1]\otimes[1]$\\
\hline
III %& $[0001]$
&$\Lambda_8$&$\Delta_8^+$&$\Delta_8^-$\\
\hline
\end{tabular}
\end{table}
By Schur's lemma, a $G$-invariant metric on $G/K$ has the form of \eqref{invariant metric on g/k}.
\end{proof}

\begin{proposition}
\label{metric}
The formula of Ricci endomorphism on $(G/K,g_{G/K})$ is given by \eqref{ricci on principaltotal} and \eqref{ricci on principal} with constants $a$ and $b$ listed in Table \ref{constant}.
\end{proposition}
\begin{proof}
Since the Ricci endomorphism is also $G$-invariant, it has the form of \eqref{ricci on principaltotal}. To compute its formula, use (7.39) in \cite{besse2007einstein} to derive the scalar curvature on $G/K$ and then apply variation. For each case, since $[\mathfrak{p}_{j},\mathfrak{p}_{j}]\subset \mathfrak{k}$ and $[\mathfrak{p}_j,\mathfrak{p}_k]\subset \mathfrak{p}_l$, each $r_j$ in \eqref{ricci on principaltotal} has the form of \eqref{ricci on principal}.
\end{proof}

Take $M$ as an associated vector bundle to principal $H$-bundle $G\rightarrow G/H$ of cohomogeneity one. As the orbit space is of dimension one, the action of $H$ on the unit sphere of $\mathbb{R}^{d+1}$ must be transitive. Then the group $K$ is taken as an isotropy group of a fixed nonzero element in $\mathbb{R}^{d+1}$, say $v_0=(1,0,\dots,0)$. It is clear that $H/K=\mathbb{S}^d$. Hence $G/K$ is indeed a unit sphere bundle over $G/H$. In this setting, $g_{G/K}(t)$ is an $S^2(\mathfrak{p}_{1}\oplus\mathfrak{p}_{2}\oplus\mathfrak{p}_{3})^{K}$-valued function with
each $f_j$ in \eqref{invariant metric on g/k} as a positive function. Correspondingly, the Ricci endomorphism $r$ in \eqref{ricci on principaltotal} is also an $S^2(\mathfrak{p}_{1}\oplus\mathfrak{p}_{2}\oplus\mathfrak{p}_{3})^{K}$-valued function. 

\begin{proposition}
\label{ricci}
For $(G,H,K)$ listed in \eqref{Cases}, Ricci-flat conditions \eqref{old Ricci-flat1} \eqref{old Ricci-flat2} and \eqref{old conservation} respectively become
\begin{equation}
\label{Ricci-flat0}
L=\frac{\dot{f}_1}{f_1}\left.Q\right|_{\mathfrak{p}_{1}}\oplus \frac{\dot{f}_2}{f_2}\left.Q\right|_{\mathfrak{p}_{2}}\oplus \frac{\dot{f}_3}{f_3}\left.Q\right|_{\mathfrak{p}_{3}},
\end{equation}
\begin{equation}
\label{Ricci-flat1}
\dfrac{\ddot{f_j}}{f_j}-\left(\dfrac{\dot{f_j}}{f_j}\right)^2=-\left(d\frac{\dot{f_1}}{f_1}+d\frac{\dot{f_2}}{f_2}+d\frac{\dot{f_3}}{f_3}\right)\dfrac{\dot{f_j}}{f_j}+r_j,\quad j=1,2,3
\end{equation}
and
\begin{equation}
\label{old conservation law}
\begin{split}
&-d\sum_{j=1}^3\left(\dfrac{\dot{f_j}}{f_j}\right)^2=-\left(\sum_{j=1}^3d\dfrac{\dot{f_j}}{f_j}\right)^2+R.
\end{split}
\end{equation}
\end{proposition}
\begin{proof}
The proof is complete by computation results in Proposition \ref{metric} and Proposition \ref{ricci}.
\end{proof}

In summary, constructing a smooth complete cohomogeneity one Ricci-flat metric on $M$ is essentially equivalent to solving $g_{G/K}(t)$ that satisfies \eqref{Ricci-flat0}, \eqref{Ricci-flat1} and \eqref{old conservation law}. The fundamental theorem of ODE guarantees the existence of solution on neighborhood around $\{t_0\}\times G/K$ for any $t_0\in (0,\infty)$. In order to have a smooth complete Ricci-flat metric on $M$, we need to show that 

1. (Smooth extension) the solution exists on a tubular neighborhood around $G/H$ and extends smoothly to the singular orbit;

2. (Completeness) the solution exists on $(0,\infty)\times G/K$.

We discuss the smooth extension in Section \ref{Smoothness} and Section \ref{Local Existence}. The proof for completeness is in Section \ref{Completensesetotal}.

\subsection{Smoothness Extension}
\label{Smoothness}
It is not difficult to guarantee the smoothness of $g_{G/K}(t)$ at $t=0$ as a $S^2(\mathfrak{p}_{1}\oplus\mathfrak{p}_{2}\oplus\mathfrak{p}_{3})$-valued function. However, the smooth function does not guarantee the smooth extension of $g=dt^2+g_{G/K}(t)$ as a metric on $G/H$ as $t\to 0$. By Lemma 1.1 in \cite{eschenburg2000initial}, the question boils down to studying the slice representation $\chi=\mathbb{R}^{d+1}$ of M and the isotropy representation $\mathfrak{q}$ of $G/H$. We rephrase the lemma below.

\begin{lemma}\cite{eschenburg2000initial} 
\label{EW}
Let $g(t): [0,\infty)\rightarrow S^2(\chi\oplus \mathfrak{q})^K$ be a smooth curve with Taylor expansion at $t=0$ as $\sum_{l=0}^\infty g_lt^l$. Let $W_l=\hmo(S^l(\chi),S^2(\chi\oplus \mathfrak{q}))^H$ be the space of $H$-equivariant homogeneous polynomials of degree $l$. Let $\iota\colon W_l\rightarrow S^2(\chi\oplus \mathfrak{q})$ denote the evaluation map at $v_0=(1,0,\dots,0)$. Then the map $g(t)$ has a smooth extension to $G/H$ as a symmetric tensor if and only if $g_l\in\iota(W_l)$ for all $l$.
\end{lemma}

To compute $W_l$, we need to identify $\chi$ and $\mathfrak{q}$ first. Since $H$ acts transitively on $H/K$, the slice representation $\chi=\mathbb{R}^{d+1}$ of $M$ is irreducible and hence can be identified. Recall that $\mathfrak{q}$ is an irreducible $H$-module in decomposition \eqref{decomposition}. Hence we have Table \ref{Slicerepn}.
\begin{table}[h!]
\centering
\caption{}
\label{Slicerepn}
\begin{tabular}{ l| l l|l} 
\hline
Case & $\chi\otimes\mathbb{C}$ as an $H$-module & $\chi \otimes\mathbb{C}$  as a $K$-module &$\mathfrak{q}\otimes\mathbb{C}$ as an $H$-module\\
\hline
I  & $[2]\otimes \mathbb{I}$&$\mathbb{R}\oplus(\theta_1^2\otimes\mathbb{I})$ & $([1]\otimes \theta_2^3)\oplus ([1]\otimes \theta_2^{-3})$\\
\hline
II & $[1,0]\otimes \mathbb{I}$&$\mathbb{R}\oplus([1]\otimes[1]\otimes\mathbb{I})$& $[0,1]\otimes[1]$\\
\hline
III& $[1,0,0,0]$&$\mathbb{R}\oplus \Lambda_8$&$[0,0,0,1]$\\
\hline
\end{tabular}.
\end{table}
\begin{remark}
\label{spin(d+1)}
Recall the background metric $Q$ on $G/K$ is chosen that $\left.Q\right|_{\mathfrak{p}_1}$ is the standard metric on $\mathbb{S}^d$. Therefore, the Euclidean inner product $\langle\cdot,\cdot\rangle$ on $\chi$ can be written in ``polar coordinate'' as $dt^2+t^2\left.Q\right|_{\mathfrak{p}_1}$.
As shown in the first column of Table \ref{Slicerepn}, the action of $H$ is essentially the standard representation of $\spin(d+1)$ on $\chi$ and it  preserves $\langle\cdot,\cdot\rangle$. In the following discussion, we take $\langle\cdot,\cdot\rangle\oplus\left.Q\right|_{\mathfrak{p}_2}\oplus\left.Q\right|_{\mathfrak{p}_3}$ as the background metric of $T_pM=\chi\oplus T_p(G/H)$ for $p=[H]\in G/H$.
\end{remark}

Compare the second column of Table \ref{Slicerepn} to the first column of Table \ref{Kmodule}. It is clear that $\chi=\mathbb{R}\oplus \mathfrak{p}_{1}$ as a $K$-module. Since $\chi$ and $\mathfrak{q}$ are inequivalent $H$-modules, we have 
\begin{equation}
\label{decomposition2}
S^2(\chi\oplus \mathfrak{q})^{K}=S^2(\chi)^{K}\oplus S^2(\mathfrak{q})^{K}.
\end{equation}
Hence we have decomposition $W_l=W_l^+\oplus W_l^-$ where $W_l^+$ and $W_l^-$ are respectively valued in $S^2(\chi)$ and $S^2(\mathfrak{q})$. We are ready to compute each $W_l^\pm$.

\begin{proposition}
\label{EACHW}
For each $M$, we have 
$$
W_l^+\cong\left\{\begin{array}{ll}
\mathbb{R}& l=0\\
0& l\equiv1\mod 2\\
\mathbb{R}^2& l\equiv 0\mod 2,\quad l\geq 2\\
\end{array}\right.,\quad 
W_l^-\cong\mathbb{R}
$$
\end{proposition}
\begin{proof}
From Table \ref{Slicerepn}, we can derive the decomposition of complexified symmetric products $S^l(\chi)\otimes \mathbb{C}$ and $S^l(\mathfrak{q})\otimes \mathbb{C}$ as $H$-modules, as shown in Table \ref{Taylor} below. The proof is complete.
\begin{table}[h!]
\centering
\caption{}
\label{Taylor}
\begin{tabular}{ l| l l l} 
\hline
Case & $S^{2m-1}(\chi)\otimes\mathbb{C}$& $S^{2m}(\chi)\otimes\mathbb{C}$&$S^2(\mathfrak{q})\otimes\mathbb{C}$\\
\hline
I  &$\bigoplus^m_{i=1}\left([4i-2]\otimes \mathbb{I}\right)$ & $\bigoplus^m_{i=0}\left([4i]\otimes \mathbb{I}\right)$& $([2]\otimes \theta_2^{6} )\oplus([2]\otimes \theta_2^{-6} )\oplus([2]\otimes \mathbb{I})\oplus\mathbb{I}$\\
\hline
II &$\bigoplus^m_{i=1}\left([2i-1,0]\otimes \mathbb{I}\right)$ & $\bigoplus^m_{i=0}\left([2i,0]\otimes \mathbb{I}\right)$& $([0,2]\otimes[2])\oplus([1,0]\otimes\mathbb{I})\oplus\mathbb{I}$\\
\hline
III& $\bigoplus^m_{i=1}[2i-1,0,0,0]$ &$\bigoplus^m_{i=0}[2i,0,0,0]$&$[0,0,0,2]\oplus[1,0,0,0]\oplus\mathbb{I}$\\
\hline
\end{tabular}
\end{table}
\end{proof}
In order to apply Lemma \ref{EW}, we need to find generators of each $W_l^\pm$ in Proposition \ref{EACHW}. Note that $W_l^{\pm}$ can be viewed as subspaces of $W_{l+2}^{\pm}$ by multiplying each element with $\sum_{i=0}^dx_i^2$. Hence we only need to find generators of $W_0^\pm$
, $W_2^+$ and $W_1^-$.
It is clear that $W_0^+$ is spanned by $I_{d+1}\in S^2(\chi)$ and $W_0^-$ is spanned by $I_{2d}\in S^2(\mathfrak{q})$.
It is also clear that $W_2^+$ is generated by the identity map
and $(\sum_{i=0}^d x_i^2)I_{d+1}$. Note that the identity map in the form of a homogeneous polynomial is a symmetric matrix $\Pi$ with $\Pi_{ij}=x_ix_j$ for $i,j\in\{0,1,\dots,d\}$.  

The computation for $W_1^-$ is a bit more complicated. We follow Chapter 14 in \cite{harvey_spinors_1990} and consider $\chi=\mathbb{R}\oplus \mathbb{F}$ with $\mathbb{F}$ as one of $\mathbb{C},\mathbb{H}$ and $\mathbb{O}$ for Case I, II and III, respectively.

\begin{proposition}
\label{w1-}
$W_1^-$ is generated by the $\mathbb{R}$-linear map
\begin{align*}
\Phi\colon \chi&\rightarrow S^2(\mathfrak{q})\\
(x_0,\mathbf{x})&\mapsto
\begin{bmatrix}
x_0 I_{d}& \mathsf{L}_\mathbf{x}\\
\mathsf{L}_{\bar{\mathbf{x}}} &-x_0 I_{d},
\end{bmatrix}
\end{align*}
where $\mathsf{L}_\mathbf{x}$ is the real matrix representation of left multiplication of $\mathbf{x}\in\mathbb{F}$, as shown in Table \ref{Only for j1} below.
\end{proposition}
\begin{table}[h!]
\centering
\caption{}
\label{Only for j1}
\begin{tabular}{ l| l l l} 
\hline
Case & I & II&III\\
\hline
$\mathsf{L}_\mathbf{x}$ & $\begin{bmatrix}
x_1&-x_2\\
x_2&x_1
\end{bmatrix}
$& $\begin{bmatrix}
x_1&-x_2&-x_3&-x_4\\
x_2&x_1&-x_4&x_3\\
x_3&x_4&x_1&-x_2\\
x_4&-x_3&x_2&x_1
\end{bmatrix}$
&
$\begin{bmatrix}
x_1&-x_2&-x_3&-x_4&-x_5&-x_6&-x_7&-x_8\\
x_2&x_1 &-x_4&x_3 &-x_6&x_5 &x_8 &-x_7\\
x_3&x_4 &x_1 &-x_2&-x_7&-x_8&x_5 &x_6\\
x_4&-x_3&x_2 &x_1 &-x_8&x_7 &-x_6&x_5\\
x_5&x_6 &x_7 &x_8 &x_1 &-x_2&-x_3&-x_4\\
x_6&-x_5&x_8 &-x_7&x_2 &x_1 &x_4 &-x_3\\
x_7&-x_8&-x_5&x_6 &x_3 &-x_4&x_1 &x_2\\
x_8&x_7 &-x_6&-x_5&x_4 &x_3 &-x_2&x_1
\end{bmatrix}$
\\
\hline
\end{tabular}
\end{table}
\begin{proof}
Consider $i\Phi(\chi)$ a subspace of $\mathbb{C}\otimes_\mathbb{R} S^2(\mathfrak{q})$. Since $(i\Phi(x_0,\mathbf{x}))^2=-(x_0^2+\|\mathbf{x}\|^2)I_{d+1}$, it is clear that the matrix multiplication of $i\Phi(\chi)$ generates a Clifford algebra and hence $\spin(d+1)$. Specifically, the group is generated by elements $\Xi(y_0,\mathbf{y}):=\Phi(-1,\mathbf{0})\Phi(y_0,\mathbf{y})$ with $y_0^2+\|\mathbf{y}\|^2=1$. Since each $\mathbb{F}$ is an alternative algebra that satisfies Moufang identity, computations show
\begin{equation}
Ad(\Xi(y_0,\mathbf{y}))(\Phi(x_0,\mathbf{x}))=\Xi(y_0,\mathbf{y}) (\Phi(x_0,\mathbf{x})) \Xi(y_0,\mathbf{y})^{-1}=\Phi(z_0,\mathbf{z}),
\end{equation}
where $z_0=(y_0^2-\|\mathbf{y}\|)x_0+2y_0\langle\mathbf{y},\mathbf{x}\rangle$ and $\mathbf{z}=y_0^2\mathbf{x}-2x_0y_0\mathbf{y}-(\mathbf{y}\bar{\mathbf{x}})\mathbf{y}$. Hence $\Phi(\chi)$ is an $Ad_{\spin(d+1)}$-invariant subspace in $S^2(\mathfrak{q})$. Moreover, since 
$$
(Ad(\Xi(y_0,\mathbf{y}))(\Phi(x_0,\mathbf{x})))^2=(\Phi(x_0,\mathbf{x})))^2=(x_0^2+\|\mathbf{x}\|^2)I_{d+1},
$$
The adjoint action on $\Phi(\chi)$ induces the standard representation $\Lambda_{d+1}$ on $\mathbb{R}^{d+1}$. Therefore,
\begin{align*}
\Phi\colon (\chi,\Lambda_{d+1})&\rightarrow \left(\Phi(\chi),Ad_{\spin(d+1)}\right)
\end{align*}
is $H$-equivariant and generates $W_1^-$.
\end{proof}

%\begin{remark}
%\label{nofreedom}
%Since $W_2^+/W_0^+\cong \mathbb{R}$, in principle there is a free variable for the second derivative of a smooth $D(t)$. However, with the geometric setting that $t$ is a unit speed geodesic and $\partial_t=v_0$ is perpendicular to $\mathfrak{p}_1$, the choice of $D_2$ is in fact determined by $D_0$.
%\end{remark}

With the generators known, we are ready to prove the following proposition.
\begin{proposition}
\label{Smooooooth}
The necessary and sufficient conditions for a metric $g=dt^2+g_{G/K}(t)$ on $M$ to extend to a smooth metric in a tubular neighborhood of the singular orbit $G/H$ are
\begin{equation}
\label{old initial}
\lim_{t\to 0}(f_1,f_2,f_3,\dot{f_1},\dot{f_2},\dot{f_3})
=(0,h_0,h_0,1,-h_1,h_1)
\end{equation}
for some $h_0>0$ and $h_1\in\mathbb{R}$.
\end{proposition}
\begin{proof}

The metric $g$ in LHS of \eqref{Ricci-flat metric on M} can be identified with a map
\begin{equation}
\label{metric map!}
g(t)\colon [0,\epsilon) \rightarrow S^2(\chi)^{K}\oplus S^2(\mathfrak{q})^{K}
\end{equation}
with Taylor expansion
\begin{equation}
\label{taylorforg}
g(t)=\sum_{l=0}^\infty g_lt^l.
\end{equation}
Write $g(t)=D(t)\oplus J(t)$, where $D(t)\colon [0,\infty)\rightarrow S^2(\chi)$ and $J(t)\colon [0,\infty)\rightarrow S^2(\mathfrak{q})$. The Taylor expansion \eqref{taylorforg} can be rewritten as
\begin{equation}
\label{specific taylor}
\begin{split}
&D(t)=D_0+D_1t+D_2t^2+\dots\\
&J(t)=J_0+J_1t+J_2t^2+\dots
\end{split}
\end{equation}

Since $W_2^+/W_0^+\cong \mathbb{R}$, in principle there is a free variable for the second derivative of a smooth $D(t)$. However, with the geometric setting that $t$ is a unit speed geodesic, the choice of $D_2$ is in fact determined by $D_0$. Hence we take $D_0=I_{d+1}$ and $D_2$ must be a multiple of $\left((\sum_{i=0}^d x_i^2)I_{d+1}-\Pi\right)(v_0)=
\begin{bmatrix}
0&\\
&I_d
\end{bmatrix}
$
with the multiplier determined by the choice of $D_0$. Since $H/K$ is and irreducible sphere, it is expected that there is no indeterminacy from $D(t)$. By Lemma \ref{EW}, the smooth condition for $D(t)$ with respect to background metric $\langle\cdot,\cdot\rangle$ is
$
D(t)=I_{d+1}+O(t^2).
$ This is consistent with Lemma 9.114 in \cite{besse2007einstein}.

As $g$ degenerates to an invariant metric on $G/H$ and the isotropy representation of $G/H$ is irreducible, $J_0$ is a positive multiple of $I_{2d}$. The evaluation of $\Phi$ at $v_0$ in Proposition \ref{w1-} is $\begin{bmatrix}
I_d&\\
&-I_d
\end{bmatrix}.
$ Hence by Lemma \ref{EW},
the smoothness condition for $J(t)$ is 
$$
J(t)=\begin{bmatrix}
f_2^2(t)I_{d}&\\
&f_3^2(t)I_{d}
\end{bmatrix}=c_0I_{2d}+c_1\begin{bmatrix}
I_d&\\
&-I_d
\end{bmatrix}t+O(t^2)
$$
for some $c_0>0$ and $c_1\in\mathbb{R}$. 

Recall \ref{spin(d+1)}, note that $\langle\cdot,\cdot\rangle=dt^2+t^2\left.Q\right|_{\mathfrak{p}_1}$. Switch the background metric to $dt^2+Q$, we conclude that the smoothness condition for $g$ is 
\begin{equation*}
\begin{split}
&f_1^2(t)=t^2+O(t^4)\\
&f_2^2(t)=c_0+c_1t+O(t^2)\\
&f_3^2(t)=c_0-c_1t+O(t^2)
\end{split}
\end{equation*}
Then the proof is complete.
\end{proof}

\begin{remark}
\label{already}
The Ricci-flat ODE system \eqref{Ricci-flat1} and \eqref{old conservation law} is invariant under the homothetic change $\kappa^2(dt^2+g_{G/K})$ with $ds=\kappa dt$. The smooth initial condition \ref{old initial} is transformed to  $(0, \kappa h_0, \kappa h_0, 1, h_1, -h_1).$ Hence if we abuse the notation. Multiplying $h_0$ by $\kappa>0$ while having $\dot{f}_j(0)$ unchanged give the smooth initial condition for metrics in the same homothetic family. Therefore, in the original coordinate, $h_1$ is the free variable that gives non-homothetic metrics. As shown in (2.27), only $h_1$ matters in producing different curves in the polynomial system.

Combine the analysis in Proposition \ref{Smooooooth} with the main result in \cite{eschenburg2000initial}, we conclude that there exists a 1-parameter family of Ricci-flat metric on a neighborhood around $G/H$ in $M$. We derive the same result in Section \ref{Local Existence} using a new coordinate.
\end{remark}

\begin{remark}
 Note that we always have $\lim\limits_{t\to 0}\frac{\dot{f}_3}{f_3}+\frac{\dot{f}_2}{f_2}=0$, i.e., the mean curvature of $G/H$ vanishes at $t=0$. This is consistent with Corollary 1.1 in \cite{hsianglawson}. The last two components of \eqref{old initial} shows that the smooth extension does not require $G/H$ to be totally geodesic. If $h_1$ in \eqref{old initial} vanishes, then we recover cases in \cite{bohm_inhomogeneous_1998}\cite{wink2017cohomogeneity} with $f_2\equiv f_3$.
\end{remark}

\begin{remark}
\label{symmetry}
It is worth pointing out that Equations \eqref{Ricci-flat1} and \eqref{old conservation law} are symmetric among $f_1$, $f_2$ and $f_3$. Therefore, initial condition \eqref{old initial} has two other counterparts where $f_2$ or $f_3$ collapses initially depending how $H$ is embedded in $G$. Without loss of generality, we will consider initial condition \eqref{old initial} in this article.
\end{remark}
%%%%%%%%%%%%%%%%%%%%%%%%%%%%%%%%%%%%%%%%%%%%%%%%%%%%%%%%%%%%%%%%%%%%%%%%
We end this section by identifying each vector bundle $M$ as a (sub)bundle of ASD $d$-form of lowest rank. Table \ref{MidForm} lists out $H$-decomposition of $\bigwedge\nolimits^{d}\mathfrak{q}\otimes\mathbb{C}$ and dimension of each irreducible summand. The subspace $\bigwedge\nolimits^{d}_-\mathfrak{q}\otimes\mathbb{C}$ consists of summands in brace brackets. Decomposition below is mostly computed via software $\mathsf{LiE}$, with reference in \cite{brown1972riemannian}\cite{salamon1989riemannian}\cite{lopez2009canonical}.
\begin{table}[h!]
\centering
\caption{}
\label{MidForm}
\begin{tabular}{ l| l l}
\hline
Case & $H$ & $H$-decomposition of $\bigwedge\nolimits^{d}\mathfrak{q}\otimes\mathbb{C}$ and Dimension of each Summand \\ 
\hline
I & $S(U(2)U(1))$ &$
\begin{array}{r@{}l@{}}
\\
\bigwedge\nolimits^2 \left(([1]\otimes \theta_2^3)\oplus([1]\otimes \theta_2^3)\right) & {}=(\mathbb{I}\otimes\theta_2^6)\oplus(\mathbb{I}\otimes\theta_2^{-6})\oplus\mathbb{I}\oplus\{[2]\otimes\mathbb{I}\}\\
\\
\mathbf{6} & {}=\mathbf{1}+\mathbf{1}+\mathbf{1} +\{\mathbf{3}\}\end{array}
$\\ \\
\hline
II & $Sp(2)Sp(1)$ & $\begin{array} {r@{}l@{}}
\\
\bigwedge\nolimits^4 [01]\otimes[1] & {}=[01]\otimes[2]\oplus[02]\otimes\mathbb{I}\oplus\mathbb{I}\otimes[4]\oplus\mathbb{I}\oplus\{[02]\otimes[2]\oplus[01]\otimes\mathbb{I}\}\\
\\
 \mathbf{70} & {}= \mathbf{15}+\mathbf{14}+\mathbf{5}+\mathbf{1}+\{\mathbf{30}+\mathbf{5}\}
 \end{array}$\\ \\
\hline
III & $\spin(9)$ & $\begin{array} {r@{}l@{}}
\\
\bigwedge\nolimits^8 [0001]&{}=[2010]\oplus[0020]\oplus[1002]\oplus[0200]\oplus[4000]\oplus[0010]\oplus[2000]\oplus\mathbb{I}
\\
& \quad \oplus\{[2002]\oplus[0110]\oplus[1010]\oplus[3000]\oplus[0002]\oplus[1000]\}\\
\\
\mathbf{12870} & {}= \mathbf{2457}+\mathbf{1980}+\mathbf{924}+\mathbf{495}+\mathbf{450}+\mathbf{84}+\mathbf{44}+\mathbf{1} \\ 
& \quad +\{\mathbf{3900}+\mathbf{1650}+\mathbf{594}+\mathbf{156}+\mathbf{126}+\mathbf{9}\}\end{array}$\\ \\ 
\hline
\end{tabular}
\end{table}

For Case I, it is known that the trivial representation generates the invariant K\"ahler form on $\mathbb{CP}^2$. The bundle that we study in this paper is the associated bundle with respect to representation $[2]\otimes\mathbb{I}$, which is the bundle of ASD 2-form $\bigwedge^2_-T^*\mathbb{CP}^2$ that admits a complete smooth $G_2$ metric \cite{bryant1989construction}\cite{gibbons1990Einstein}.

For Case II, the trivial representation generates a canonical 4-form for Quaternionic K\"ahler manifolds, as described in \cite{salamon1989riemannian}. Explicitly, given a Quaternionic K\"ahler manifold with a triple of complex structures $(I,J,K)$ and corresponding symplectic forms $(\omega_I,\omega_J,\omega_K)$, the canonical 4-form is defined as 
$
\Omega=\omega_I\wedge\omega_I+\omega_J\wedge\omega_J+\omega_K\wedge\omega_K.
$
By Table \ref{Slicerepn}, $M$ is an associate bundle with respect to representation $[01]\otimes\mathbb{I}$ in $\bigwedge_-^4\mathfrak{q}_2\otimes\mathbb{C}$. Therefore, $M$ is indeed an irreducible subbundle of $\bigwedge_-^4T^*\mathbb{HP}^2$.

For Case III, the trivial representation generates the canonical 8-form, whose existence is proved in \cite{brown1972riemannian}. Explicit formula for the canonical 8-form can be found in \cite{lopez2009canonical}. The 9-dimensional representation $[1000]$ is the (twisted) adjoint representation of $\spin(9)$ on $\mathbb{R}^9$. Similar to Case II, the bundle that we consider in this paper is an irreducible subbundle of $\bigwedge^8_-T^*\mathbb{OP}^2$.

In conclusion, the name ``(sub)bundle of ASD $d$-form of lowest rank'' for $M$ is justified.

\subsection{Coordinate Change and Linearization}
\label{Local Existence}
We apply the coordinate change introduced in \cite{dancer2008non}\cite{dancer2008some} to the Ricci-flat system in this section. The original ODE system is transformed to a polynomial one. As described in Remark \ref{important criti}, some critical points of the new system carry geometric data. Linearizations at these critical points provide guidance on how integral curves potentially behave, which help us to construct a compact invariant set in Section \ref{Completensesetotal} to prove the completeness.

As predicted by the result in the previous section (Remark \ref{already}), analysis on the new system shows that there exists a 1-parameter family of integral curves with each represents a homothetic class of Ricci-flat metrics on a neighborhood around $G/H$.

 Consider
\begin{equation}
\label{coordinate change}
d\eta=\trace(L) dt.
\end{equation}
Define 
\begin{equation}
\label{def of variable}
X_j:=\dfrac{\frac{\dot{f_j}}{f_j}}{\trace(L)},\quad Z_j:=\dfrac{\frac{f_j}{f_kf_l}}{\trace(L)}.
\end{equation}
And define
\begin{equation*}
\mathcal{R}_j:=\frac{r_j}{(\trace(L))^2}=aZ_kZ_l+b\left(Z_j^2-Z_k^2-Z_l^2\right),\quad \mathcal{G}:=\sum_{j=1}^3 d X_j^2, \quad \mathcal{H}:=\sum_{j=1}^3dX_j.
\end{equation*}
Use $'$ to denote derivative with respect to $\eta$. In the new coordinates given by \eqref{coordinate change} and \eqref{def of variable}, the system \eqref{Ricci-flat1} is transformed to
\begin{equation}
\label{new Ricci-flat}
\begin{split}
\begin{bmatrix}
X_1\\
X_2\\
X_3\\
Z_1\\
Z_2\\
Z_3
\end{bmatrix}'=V(X_1,X_2,X_3,Z_1,Z_2,Z_3):=\begin{bmatrix}
X_1(\mathcal{G}-1)+\mathcal{R}_1\\
X_2(\mathcal{G}-1)+\mathcal{R}_2\\
X_3(\mathcal{G}-1)+\mathcal{R}_3\\
Z_1\left(\mathcal{G}-\frac{\mathcal{H}}{d}+2X_1\right)\\
Z_2\left(\mathcal{G}-\frac{\mathcal{H}}{d}+2X_2\right)\\
Z_3\left(\mathcal{G}-\frac{\mathcal{H}}{d}+2X_3\right)
\end{bmatrix},
\end{split}
\end{equation}
and the conservation law \eqref{old conservation law} becomes
\begin{equation}
\label{conservation law}
C\colon \mathcal{G}-1+d\sum_j \mathcal{R}_j=0.
\end{equation}
As $\left(\frac{1}{\trace(L)}\right)'=\frac{\mathcal{G}}{\trace(L)}$, the original variables can be recovered by
\begin{equation}
\label{recovery}
t=\int_{\eta_0}^{\eta} \exp\left(\int_{\tilde{\eta}_0}^{\tilde{\eta}} \mathcal{G} d\tilde{\tilde{\eta}}+\tilde{t}_0\right) d\tilde{\eta}+t_0,\quad  f_j=\dfrac{\exp\left(\int_{\eta_0}^\eta \mathcal{G} d\tilde{\eta}+t_0\right)}{\sqrt{Z_kZ_l}}.
\end{equation}

\begin{remark}
\label{themeaningofz}
The new variables $X_j$'s record the relative size of each principal curvature of $G/K$. Variables $Z_j$'s carry the data of relative size of each $f_j$'s. Note that $\frac{Z_j}{Z_k}=\frac{f_j^2}{f_k^2}$.
\end{remark}

In the original coordinates, a smooth solution to \eqref{Ricci-flat1} is an integral curve with variable $t\in [0,\epsilon)$. Since by \eqref{coordinate change}, $\lim\limits_{t\to 0}\eta=\lim\limits_{t\to 0}\ln\left(f_1^{d}f_2^{d}f_3^{d}\right)+\hat{\eta}=-\infty$, the original solution is transformed to an integral curve with variable $\eta\in (-\infty,\epsilon')$ for some $\epsilon'\in\mathbb{R}$. Note that the graph of the integral curve does not change when homothetic change is applied to the original variable. Hence each integral curve to the new system represent a solution in the original coordinate up to homothety.

\begin{remark}
\label{symmetry 2}
It is clear that the symmetry mentioned in Remark \ref{symmetry} remains  among pairs $(X_j,Z_j)$'s in the new system \eqref{new Ricci-flat} with \eqref{conservation law}. In addition, by the observation on $Z_j$'s derivative. It is clear that they do not change sign along the integral curve. Without loss of generality, we focus on the region where these three variables are positive. This observation provides basic estimates needed in our construction of compact invariant set (the set $P$ introduced in \eqref{set S3's friends}).
\end{remark}

\begin{remark}
\label{smooth conservation}
It is clear that $\mathcal{H}\equiv 1$ by the definition variable $X_j$. In fact, since $\mathcal{H}'=(\mathcal{H}-1)(\mathcal{G}-1)$ on $C$, the set $C\cap \{\mathcal{H}\equiv 1\}$ is flow-invariant. Furthermore, $C\cap \{\mathcal{H}\equiv 1\}$ is diffeomorphic to a level set
$$
dX_1^2+dX_2^2+d\left(\frac{1}{d}-X_1-X_2\right)^2-1+d\sum_j \mathcal{R}_j=0
$$
in $\mathbb{R}^5$. Therefore, $C\cap \{\mathcal{H}\equiv 1\}$ is a $4$-dimensional smooth manifold by the inverse function theorem. System \eqref{new Ricci-flat} can be restricted to a $4$-dimensional subsystem on $C\cap \{\mathcal{H}\equiv 1\}$.
\end{remark}

\begin{proposition}
\label{critical points}
The complete list of critical points of system \eqref{new Ricci-flat} in $C\cap \{\mathcal{H}\equiv 1\}$ is the following:
\begin{enumerate}[I.]
\item 
the set $\left\{(x_1,x_2,x_3,0,0,0)\mid \sum_j^3x_j^2=\frac{1}{d},\quad \sum_j^3x_j=\frac{1}{d}\right\}$;
\item
$\left(-\frac{1}{d},\frac{1}{d},\frac{1}{d},\pm\frac{1}{d}\sqrt{\frac{3-d}{b}},0,0\right)$ and its counterparts with pairs $(X_j,Z_j)$'s permuted. This critical point occurs only for Case I;
\item
$\left(\frac{1}{d},0,0,0,\pm\frac{1}{d},\pm\frac{1}{d}\right)$ and its counterparts with pairs $(X_j,Z_j)$'s permuted;
\item
$\left(\frac{1}{n},\frac{1}{n},\frac{1}{n},\pm\frac{2b}{d-1}\frac{1}{n}\sqrt{\frac{(n-1)(d-1)}{b(a+2b)}},\pm\frac{2b}{d-1}\frac{1}{n}\sqrt{\frac{(n-1)(d-1)}{b(a+2b)}},\pm\frac{1}{n}\sqrt{\frac{(n-1)(d-1)}{b(a+2b)}}\right)$ and its counterparts with pairs $(X_j,Z_j)$'s permuted;
\item 
$\left(\frac{1}{n},\frac{1}{n},\frac{1}{n},\pm\frac{1}{n}\sqrt{\frac{n-1}{a-b}},\pm\frac{1}{n}\sqrt{\frac{n-1}{a-b}},\pm\frac{1}{n}\sqrt{\frac{n-1}{a-b}}\right)$.
\end{enumerate}
\end{proposition}
\begin{proof}
The proof is processed by direct computations.
\end{proof}

By Remark \ref{symmetry 2}, we focus on critical points with non-negative $Z_j$'s. 

\begin{remark}
\label{important criti}
Some critical points in Proposition \ref{critical points} have further geometric significance.

\begin{itemize}
\item $p_0:=\left(\frac{1}{d},0,0,0,\frac{1}{d},\frac{1}{d}\right)$\\
This critical point is the initial condition \eqref{old initial} under the new coordinate \eqref{coordinate change}-\eqref{def of variable}, i.e., \eqref{old initial} becomes $\lim\limits_{\eta\to-\infty}(X_1,X_2,X_3,Z_1,Z_2,Z_3)=p_0$. Hence we study integral curves emanating from $p_0$. In order to prove the completeness, we construct a compact invariant set in Section \ref{Completensesetotal} that contains $p_0$ in its boundary and traps the integral curve initially.

By Remark \ref{symmetry} and Remark \ref{symmetry 2}, its two other counterparts $p_0'=\left(0,\frac{1}{d},0,\frac{1}{d},0,\frac{1}{d}\right)$ and $ p_0''=\left(0,0,\frac{1}{d},\frac{1}{d},\frac{1}{d},0\right)$ also have the similar geometric meaning depending on how $H$ is embedded in $G$.
\item $p_1:=\left(\frac{1}{n},\frac{1}{n},\frac{1}{n},\frac{1}{n}\sqrt{\frac{n-1}{a-b}},\frac{1}{n}\sqrt{\frac{n-1}{a-b}},\frac{1}{n}\sqrt{\frac{n-1}{a-b}}\right)$\\
This critical point is symmetric among all $(X_j,Z_j)$'s. Note that $\frac{f_j^2}{f_k^2}(p_1)=\frac{Z_j}{Z_k}(p_1)=1$, all $f_j$'s are equal at this point. We prove in Section \ref{aymptoticlimit} that $p_1$ represents an AC end for the complete Ricci-flat metric represented by the integral curve emanating from $p_0$. The conical limit is a metric cone over a suitable multiple of the normal Einstein metric on $G/K$.

\item
$p_2:=\left(\frac{1}{n},\frac{1}{n},\frac{1}{n},\frac{2b}{d-1}\frac{1}{n}\sqrt{\frac{(n-1)(d-1)}{b(a+2b)}},\frac{2b}{d-1}\frac{1}{n}\sqrt{\frac{(n-1)(d-1)}{b(a+2b)}},\frac{1}{n}\sqrt{\frac{(n-1)(d-1)}{b(a+2b)}}\right)
$\\
Since $r_j(p_2)$ are all equal, this point represent an invariant Einstein metric on $G/K$ other than the one represented by $p_1$. In the following text, we call the metric the ``alternative Einstein metric''.  For Case I, it is a K\"ahler--Einstein metric. It has two other counterparts with permuted $Z_j$'s.

Although we do not find any integral curve with its limit as $p_2$, we show in Section \ref{a digression} that there exists an integral curve emanating from $p_2$ and tends to $p_1$, representing a singular Ricci-flat metric with a conical singularity and an AC end.
\end{itemize}
\end{remark}

The linearization $\mathcal{L}$ of vector field $V$ in \eqref{new Ricci-flat} is 
\begin{equation}
\label{linearization}
\begin{bmatrix}
\mathcal{G}-1+2dX_1^2&  2dX_1X_2& 2dX_1X_3& 2bZ_1& aZ_3-2bZ_2& aZ_2-2bZ_3\\
2dX_1X_2&\mathcal{G}-1+2dX_2^2& 2dX_2X_3& aZ_3-2bZ_1& 2bZ_2& aZ_1-2bZ_3\\
2dX_1X_3& 2dX_2X_3& \mathcal{G}-1+2dX_3^2& aZ_2-2bZ_1& aZ_1-2bZ_2& 2bZ_3\\
(2dX_1+1)Z_1& (2dX_2-1)Z_1& (2dX_3-1)Z_1& \mathcal{G}-\frac{\mathcal{H}}{d}+2X_1& 0& 0\\
(2dX_1-1)Z_2& (2dX_2+1)Z_2& (2dX_3-1)Z_2& 0& \mathcal{G}-\frac{\mathcal{H}}{d}+2X_2& 0\\
(2dX_1-1)Z_3& (2dX_2-1)Z_3& (2dX_3+1)Z_3& 0& 0& \mathcal{G}-\frac{\mathcal{H}}{d}+2X_3
\end{bmatrix}
\end{equation}

With \eqref{linearization} we can compute the dimension of the unstable subspace at $p_0$. As we are considering system \eqref{new Ricci-flat} on $C\cap \{\mathcal{H}\equiv 1\}$, we require each unstable eigenvector to be tangent to $C\cap \{\mathcal{H}\equiv 1\}$. The normal vector field to the hypersurfaces $C$ and $\{\mathcal{H}\equiv 1\}$ are respectively
\begin{equation}
\label{normal vector}
N_C=\begin{bmatrix}
2dX_1\\
2dX_2\\
2dX_3\\
adZ_2+adZ_3-2bdZ_1\\
adZ_1+adZ_3-2bdZ_2\\
adZ_2+adZ_1-2bdZ_3
\end{bmatrix}, \quad N_{\{\mathcal{H}\equiv 1\}}=\begin{bmatrix}
1\\
1\\
1\\
0\\
0\\
0
\end{bmatrix}.
\end{equation}
\begin{lemma}
\label{unstable dimension}
The unstable subspace of system \eqref{new Ricci-flat} at $p_0$, restricted on $C\cap \{\mathcal{H}\equiv 1\}$, is of dimension 2.
\end{lemma}
\begin{proof}
Hence the linearization at $p_0$ is  
\begin{equation}
\label{linerization matrix <}
\mathcal{L}(p_0)=\begin{bmatrix}
\frac{3}{d}-1&0&0&0&\frac{a-2b}{d}&\frac{a-2b}{d}\\
0&\frac{1}{d}-1&0&\frac{a}{d}&\frac{2b}{d}&-\frac{2b}{d}\\
0&0&\frac{1}{d}-1&\frac{a}{d}&-\frac{2b}{d}&\frac{2b}{d}\\
0&0&0&\frac{2}{d}&0&0\\
\frac{1}{d}&\frac{1}{d}&-\frac{1}{d}&0&0&0\\
\frac{1}{d}&-\frac{1}{d}&\frac{1}{d}&0&0&0
\end{bmatrix}.
\end{equation}
Eigenvalues and corresponding eigenvectors of \eqref{linerization matrix <} are
$$
\lambda_1=\frac{1}{d},\quad \lambda_2=\lambda_3=\frac{2}{d},\quad\lambda_4=\lambda_5=\frac{1}{d}-1,\quad \lambda_6=-1.
$$
\begin{equation}
\label{list of eigen}
v_1=\begin{bmatrix}
0\\
-1\\
1\\
0\\
-2\\
2
\end{bmatrix},\quad
v_2=\begin{bmatrix}
2\\
0\\
0\\
0\\
1\\
1
\end{bmatrix},\quad v_3=\begin{bmatrix}
0\\
\frac{a}{d+1}\\
\frac{a}{d+1}\\
1\\
0\\
0
\end{bmatrix},\quad
v_4=\begin{bmatrix}
1-d\\
0\\
0\\
0\\
1\\
1
\end{bmatrix},\quad 
v_5=\begin{bmatrix}
0\\
1\\
1\\
0\\
0\\
0
\end{bmatrix},\quad
v_6=\begin{bmatrix}
0\\
4b\\
-4b\\
0\\
-1\\
1
\end{bmatrix}.
\end{equation}

With Remark \ref{scalr of spher basic ineq}, Remark \ref{smooth conservation} and \eqref{normal vector}, it is clear that
$$T_{p_0}(C\cap \{\mathcal{H}\equiv 1\})=\spann\{v_1,(d+1)v_3-av_2,2v_4+(d-1)v_5,v_6\}.$$
By \eqref{list of eigen}, an unstable subspace at $p_0$ is spanned by $v_1$ and $(d+1)v_3-av_2$. 
\end{proof}

%Each point on an integral curve to \eqref{new Ricci-flat} on \eqref{conservation law} emanating from $p_0$ represent an invariant metric on the hypersurface $G/K$. 
Solutions of the linearized equations at $p_0$ have the form
\begin{equation}
\label{linearized solution<}
\begin{split}
\begin{bmatrix}
X_1\\
X_2\\
X_3\\
Z_1\\
Z_2\\
Z_3
\end{bmatrix}&=p_0+s_0e^{\frac{2\eta}{d}}((d+1)v_4-av_3)+s_1 e^{\frac{\eta}{d}}v_1=\begin{bmatrix}
\frac{1}{d}\\
0\\
0\\
0\\
\frac{1}{d}\\
\frac{1}{d}
\end{bmatrix}+s_0e^{\frac{2\eta}{d}}\begin{bmatrix}
-2a\\
a\\
a\\
d+1\\
-a\\
-a
\end{bmatrix}
+s_1 e^{\frac{\eta}{d}}\begin{bmatrix}
0\\
-1\\
1\\
0\\
-2\\
2
\end{bmatrix},
\end{split}
\end{equation}
for some $s_0>0$ and $s_1\in\mathbb{R}$. Recall Remark \ref{symmetry 2}. In order to let $Z_1$ be positive initially, the assumption $s_0>0$ is necessary.

It is clear that there is a 1 to 1 correspondence between the germ of linearized solution \eqref{linearized solution<} around $p_0$ and $[s_0:s_1^2]$ in $\mathbb{RP}^2$. We fix $s_0>0$ in the following text. By Hartman--Grobman theorem, there is a 1 to 1 correspondence between each \eqref{linearized solution<} and local solution to \eqref{new Ricci-flat}. Hence for a fixed $s_0>0$, there is no ambiguity to use $\gamma_{s_1}$ to denote an integral curve to system \eqref{new Ricci-flat} on \eqref{conservation law} with 
$$
\gamma_{s_1}\sim p_0+s_0e^{\frac{2\eta}{d}}((d+1)v_4-av_3)+s_1 e^{\frac{\eta}{d}}v_1
$$
near $p_0$.

Analysis above shows that there exists a 1-parameters family of short-time existing integral curves of system \eqref{new Ricci-flat} on \eqref{conservation law}. 
Since each curve corresponds to a homothetic class of Ricci-flat metrics defined on a neighborhood around singular orbit $G/H$, there exists a 1-parameters family of non-homothetic Ricci-flat metrics defined on a neighborhood around $G/H$. Recall Remark \ref{already}, the result is consistent with the main theorem in \cite{eschenburg2000initial}.
\begin{remark}
\label{y a subcase}

By the unstable version of Theorem 4.5 in \cite{coddington1955theory}, from \eqref{old initial} we know that
\begin{equation}
\label{s1ish1}
\frac{2h_1}{\sqrt{d}}=\lim\limits_{t\to 0}\frac{\left(\frac{\dot{f_3}}{f_3}-\frac{\dot{f_2}}{f_2}\right)\sqrt{f_2f_3}}{\sqrt{\trace(L)f_1}}=\lim_{\eta\to \infty}\frac{X_3-X_2}{\sqrt{Z_1}}=\dfrac{2s_1}{\sqrt{(d+1)s_0}}.
\end{equation}
Hence the parameter $s_1$ vanishes if and only if $h_1$ does. The solution with $s_1=0$ corresponds to the subsystem of \eqref{new Ricci-flat} where $(X_2,Z_2)\equiv (X_3,Z_3)$ is imposed, which corresponds to the subsystem of the original system \eqref{Ricci-flat1} where $f_2\equiv f_3$ is imposed. The reduced system is essentially the same as the one for the case where the isotropy representation has two inequivalent irreducible summands. For Case I, $\gamma_{0}$ represents the smooth complete $G_2$ metric in \cite{bryant1989construction}\cite{gibbons1990Einstein}. For Case II and Case III, Ricci-flat metrics with $s_1=0$ are proved to be complete in \cite{bohm_inhomogeneous_1998}\cite{wink2017cohomogeneity}.

Our construction does not assume the vanishing of $s_1$. By the symmetry of the ODE system, we mainly focus on the situation where $s_1\geq 0$ without loss of generality.
\end{remark}

Suppose an integral curve $\gamma_{s_1}$ is defined on $\mathbb{R}$, then by Lemma 5.1 in \cite{buzano2015family}, functions $f_j(t)$'s are defined on $[0,\infty)$. Therefore, Theorem \ref{main 1} is proved once $\gamma_{s_1}$ is shown to be defined on $\mathbb{R}$.

\section{Completeness}
\label{Completensesetotal}
With smooth extension of metrics represented by $\gamma_{s_1}$ proved, the next step is to show that $\gamma_{s_1}$ is defined on $\mathbb{R}$ so that the Ricci-flat metric it represents is complete. Our construction is divided into two parts. The first part is to find an appropriate compact invariant set $\hat{S}_3$ with $p_0$ sitting on its boundary. Although $p_0$ is in the boundary of $\hat{S}_3$, integral curves are not trapped in the set initially unless $s_1=0$. In the second step, we construct another compact set that serves as an \emph{entrance zone}. It traps $\gamma_{s_1}$ initially as long as $s_1$ is close enough to zero. Moreover, integral curves trapped in this set cannot escape through some part of its boundary and they are forced to enter $•\hat{S}_3$. Hence such a $\gamma_{s_1}$ must be defined on $\mathbb{R}$.

\subsection{Compact Invariant Set}
\label{Completensese}

We describe the first step in this section. There is a subtle difference between the compact invariant set for Case I and ones for Case II and III. We first construct the set for Case II and III since it is simpler.

Let $\rho=\sqrt{\frac{a+2b}{2}}$. It is clear that $\rho\geq 1$ and equality holds exactly in Case I. Define
\begin{equation}
\label{set S3's friends}
\begin{split}
&P=\left\{Z_1,Z_2,Z_3\geq 0\right\}\\
&\tilde{S_3}=\bigcap_{j=1}^2 \left\{Z_3-Z_j\geq 0,\quad X_3-X_j+\rho(Z_3-Z_j)\geq 0,\quad X_3\geq 0\right\}.
\end{split}
\end{equation}
And define 
\begin{equation}
\label{Set S3}
S_3=C\cap \{\mathcal{H}\equiv 1\}\cap P\cap \tilde{S_3}.
\end{equation}

Before doing further analysis on $S_3$, we give some explanations as to why it is constructed in this way. Note that the positivity of $Z_j$'s are immediate by Remark \ref{symmetry 2}. The first inequality in $\tilde{S_3}$ is to require $Z_3$ to be the largest variable among $Z_j$'s. Equivalently, it requires $f_3$ to be the largest among $f_j$'s in the original coordinate. This condition is indicated by the subscript of $\tilde{S}_3$ and $S_3$. A direct consequence of this assumption is that we can assume $X_3\geq 0$ along $\gamma_{s_1}$ as shown in \eqref{ricci3}.

It is easy to check that $p_0\in S_3$ hence the set is nonempty. Each inequality in \eqref{Set S3} defines a closed subset in $\mathbb{R}^7$ whose boundary is defined by the equality. Therefore, a point $x\in \partial S_3$ if there exists at least one defining inequality in \eqref{set S3's friends} reaches equality at $x$. For Case II and III, functions
\begin{equation}
\label{where p_0 locate}
X_3, Z_1, Z_3-Z_2, X_3-X_2+\rho(Z_3-Z_2)
\end{equation}
among those in \eqref{set S3's friends} vanish at $p_0$. The point is hence in $\partial S_3$.  Substitute \eqref{linearized solution<} to functions in \eqref{where p_0 locate}. It is clear that $\gamma_{s_1}$ is trapped in $S_3$ initially if $s_1\geq 0$. By Remark \ref{y a subcase}, we know that $\gamma_0$ is trapped in $\partial S_3$ with $(X_2,Z_2)\equiv (X_3,Z_3)$. 
%In the similar fashion we can construct sets $\tilde{S_2}$ and $S_2$ so that $Z_2$ is the largest variable among $Z_j$'s. In this case, $p_0\in \partial S_2$ and $\gamma_{s_1}$ is trapped in $S_2$ initially if $s_1\leq 0$.

\begin{proposition}
\label{sharp}
In the set $S_3\cap\{2bZ_3-a(Z_1+Z_2)\leq 0\}$, we have estimate
\begin{equation}
\label{sharp Z_1+Z_2}
Z_1+Z_2\leq 2\sqrt{\frac{n-1}{n^2(a-b)}}.
\end{equation}
\end{proposition}
\begin{proof}
By the conservation law \eqref{conservation law}, it follows that 
\begin{equation}
\label{1}
0\geq \frac{1}{n}-1+da(Z_2Z_3+Z_1Z_3+Z_1Z_2)-db(Z_1^2+Z_2^2+Z_3^2).
\end{equation}
Note that the RHS of \eqref{1} is symmetric between $Z_1$ and $Z_2$. It is convenient to find the maximum of $Z_1+Z_2$ on $S_3\cap \{Z_2\geq Z_1\}$ first. By the symmetry between $Z_1$ and $Z_2$ in \eqref{1}, such a maximum is the maximum of $Z_1+Z_2$ in $S_3$. With the assumption $Z_2\geq Z_1$, we write $Z_1=\nu Z_2$ for some $\nu\in [0,1]$. Fix such a $\nu$. Then \eqref{1} becomes
\begin{equation}
\label{2}
\begin{split}
0&\geq \frac{1}{n}-1+da(Z_2Z_3+\nu Z_2Z_3+\nu Z_2^2)-db(\nu^2 Z_2^2+Z_2^2+Z_3^2)\\
&= \frac{1}{n}-1+d(-bZ_3^2+a(1+\nu)Z_2Z_3+(a\nu  -b(1+\nu^2))Z_2^2).
\end{split}
\end{equation}
Define $\mathcal{F}(Z_3)=-bZ_3^2+a(1+\nu)Z_2Z_3+(a\nu-b(1+\nu^2))Z_2^2.$ Consider the set $S_3\cap\{2bZ_3-a(Z_1+Z_2)\leq 0\}\cap \{Z_1=\nu Z_2\}$, we have 
$$
Z_2\leq Z_3\leq \frac{a}{2b}(1+\nu)Z_2.
$$
Hence for any fixed $\nu$ and $Z_2$, the minimum of $\mathcal{F}$ in $S_3\cap\{2bZ_3-a(Z_1+Z_2)\leq 0\}\cap \{Z_1=\nu Z_2\}$ is reached at $Z_3=Z_2$. Therefore, computation \eqref{2} continues as
\begin{equation}
\label{3}
\begin{split}
0&\geq \frac{1}{n}-1+d\left(-b+a(1+\nu)+(\nu a-b(1+\nu^2))\right)Z_2^2\\&=\frac{1}{n}-1+d\left(-b\nu^2+2a\nu +a-2b\right)Z_2^2.
\end{split}
\end{equation}
The coefficient of $Z_2^2$ in \eqref{3} can be easily checked to be positive. It follows that
\begin{equation}
\begin{split}
\label{4}
(Z_1+Z_2)^2=(1+\nu)^2Z_2^2&\leq \left(1-\frac{1}{n}\right)\dfrac{(1+\nu)^2}{d(-b\nu^2+2a\nu +a-2b)}\\
%&=\left(1-\frac{1}{n}\right)\dfrac{(1+\nu)^2}{-db(1+\nu)^2+2(db+da)(\nu+1)-3db-da}\\
&=\left(1-\frac{1}{n}\right)\dfrac{1}{d\left(-b+2(a+b)\frac{1}{1+\nu}-(a+3b)\frac{1}{(1+\nu)^2}\right)}.
\end{split}
\end{equation}
Consider function $h\left(\frac{1}{1+\nu}\right)=-(a+3b)\frac{1}{(1+\nu)^2}+2(a+b)\frac{1}{1+\nu}-b$. Since by Remark \ref{scalr of spher basic ineq}, we have
$
\frac{1}{2}\leq \frac{a+b}{a+3b}\leq 1,
$
the minimum of $h$ is either $h\left(\frac{1}{2}\right)$ or $h(1)$. Computation shows $h\left(\frac{1}{2}\right)<h(1)$. We conclude that
$
(Z_1+Z_2)^2\leq \left(1-\frac{1}{n}\right)\frac{1}{d}\frac{1}{h\left(\frac{1}{2}\right)}=4\frac{n-1}{n^2(a-b)}.
$
Hence the proof is complete. Note that the equality in \eqref{sharp Z_1+Z_2} is reached by $p_1$.
\end{proof}

\begin{proposition}
\label{S_3}
For Case II and III, integral curves $\gamma_{s_1}$ to system \eqref{new Ricci-flat} on $C_0\cap \{\mathcal{H}\equiv 1\}$ emanating from $p_0$ with $s_1\geq 0$ do not escape $S_3$.
\end{proposition}
\begin{proof}
Two perspectives can be taken in the following computations that frequently appear through out this article. First is to view algebraic expressions in \eqref{set S3's friends} as functions along $\gamma_{s_1}$ and they all vanish at $p_0$. Integral curves emanating from $p_0$ being trapped in $S_3$ initially is equivalent to these defining functions being positive near $p_0$. To show that $\gamma_{s_1}$ does not escape $S_3$ is to show the non-negativity of these functions along the integral curves. Suppose one of these functions vanishes at some point along the integral curves for the first time. We want to show that its derivative at that point is non-negative.

The second perspective is to consider $\partial S_3$ as a union of subsets of a collection of linear and quadratic varieties. Require the restriction of the vector field $V$ in \eqref{new Ricci-flat} on each of these subsets to point inward $S_3$. If such a requirement is met, then it is impossible for the integral curves to escape if they are initially in $S_3$. Both perspectives lead to the same computation of inner product between $V$ and the gradient of each defining function in \eqref{set S3's friends}. Then require the inner product to be non-negative if the gradient points inward $S_3$. It might not be true that the inner product is non-negative on each variety globally. But all we need is the non-negativity on its subsets that $\partial S_3$ consists of.

By definition of $S_3$, we automatically have
\begin{equation}
\label{ricci3}
\begin{split}
\mathcal{R}_3&=aZ_1Z_2+b(Z_3^2-Z_1^2-Z_2^2)=\left\{\begin{array}{cc}
Z_2(aZ_1-bZ_2)+b(Z_3^2-Z_1^2)\geq 0& \text{if $Z_1\geq Z_2$}\\
Z_1(aZ_2-bZ_1)+b(Z_3^2-Z_2^2)\geq 0& \text{if $Z_2\geq Z_1$}
\end{array}\right..
\end{split}
\end{equation}
On $X_3=0$, we have
$
 \left. \langle\nabla(X_3),V\rangle \right|_{X_3=0}=\mathcal{R}_3\geq 0
$ by \eqref{ricci3}.
Hence $X_3$ is non-negative along every $\gamma_{s_1}$ that is trapped in $S_3$ initially.

Next we need to show that the integral curves cannot escape from the part of $\partial \tilde{S_3}$ that is in $\partial S_3$. For distinct $j,k\in\{1,2\}$, it follows that
\begin{equation*}
\begin{split}
\left. \langle\nabla(Z_3-Z_j),V\rangle \right|_{Z_3-Z_j=0}&=Z_3\left(\mathcal{G}-\frac{1}{d}+2X_3\right)-Z_j\left(\mathcal{G}-\frac{1}{d}+2X_j\right)\\
&=2Z_3(X_3-X_j) \quad \text{since $Z_3-Z_j=0$}\\
&\geq 2\rho Z_3(Z_j-Z_3) \quad \text{by definition of $S_3$}\\
&=0 \quad \text{since $Z_3-Z_j=0$}.
\end{split}
\end{equation*}

Although it is not clear if $X_3-X_j\geq 0$ along $\gamma_{s_1}$, we impose a weaker condition, which is the second inequality in $\tilde{S_3}$. What it means is to allow $Z_3-Z_j$ to decrease, yet the rate of its decreasing cannot be too steep so that $Z_3-Z_j$ increases before it could decrease to zero. Fortunately, the weaker condition does hold along the integral curves.

\begin{equation}
\label{X3-X_1+k(Z_3-Z_1)pre}
\begin{split}
&\left. \langle\nabla(X_3-X_j+\rho(Z_3-Z_j)),V\rangle \right|_{X_3-X_j+\rho(Z_3-Z_j)=0}\\
&=(X_3-X_j+\rho(Z_3-Z_j))\left(\mathcal{G}-1\right)+\mathcal{R}_3-\mathcal{R}_j+\rho Z_3\left(1-\frac{1}{d}+2X_3\right)-\rho Z_j\left(1-\frac{1}{d}+2X_j\right)\\
&= (Z_3-Z_j)\left(2b(Z_3+Z_j)-aZ_k+\rho\left(1-\frac{1}{d}\right)+2\rho X_3-2\rho^2Z_j\right) \quad \text{since $X_j=X_3+\rho(Z_3-Z_j)$}\\
&\geq (Z_3-Z_j)\left(2bZ_3-a(Z_j+Z_k)+\rho\left(1-\frac{1}{d}\right)\right)\quad \text{since $X_3\geq 0$ in $S_3$}\\
&= (Z_3-Z_j)\left(2bZ_3-a(Z_1+Z_2)+\rho\left(1-\frac{1}{d}\right)\right)
\end{split}.
\end{equation}
If $2bZ_3-a(Z_1+Z_2)\geq 0$, then the last line of computation above is obviously non-negative. If $2bZ_3-a(Z_1+Z_2)\leq 0$, then \eqref{X3-X_1+k(Z_3-Z_1)pre} continues as
\begin{equation}
\label{X3-X_1+k(Z_3-Z_1)}
\geq (Z_3-Z_j)\left((b-a)(Z_1+Z_2)+\rho\left(1-\frac{1}{d}\right)\right)
\end{equation}
since $Z_3\geq \frac{Z_1+Z_2}{2}$ in $S_3$. Apply Proposition \ref{sharp}, we know that \eqref{X3-X_1+k(Z_3-Z_1)} is non-negative if 
\begin{equation}
\label{i=1 fail}
\dfrac{\rho(d-1)}{d(a-b)}\geq 2\sqrt{\frac{n-1}{n^2(a-b)}}.
\end{equation}
Straightforward computations show that 
$$
\begin{tabular}{ l| l l l} 
\hline
Case & $\rho$  & $\dfrac{\rho(d-1)}{d(a-b)}$ & $2\sqrt{\frac{n-1}{n^2(a-b)}}$\\ 
\hline
I & 1  & $\frac{2}{5}$& $\frac{2}{3}$\\ 
\hline
II & $\sqrt{\frac{5}{2}}$  & $\frac{3\sqrt{10}}{28}\approx 0.339$ & $\frac{\sqrt{154}}{42}\approx 0.295$\\ 
\hline
III & $\sqrt{\frac{11}{2}}$  & $\frac{7\sqrt{22}}{128}\approx 0.257$ & $\frac{\sqrt{46}}{48}\approx 0.141$\\ 
\hline
\end{tabular}.
$$
Inequality \eqref{i=1 fail} holds only for Case II and III. 
Hence for Case II and III, integral curves $\gamma_{s_1}$ emanating from $p_0$ does not escape $S_3$ if $s_1\geq 0$.

Although estimate \eqref{sharp Z_1+Z_2} is sharp in $S_3$, inequality \eqref{X3-X_1+k(Z_3-Z_1)pre} has room to be improved as we dropped a non-negative term $2\rho X_3$ in the computation. It turns out \eqref{X3-X_1+k(Z_3-Z_1)pre} can be proved to be non-negative for Case I with an additional inequality, as demonstrated in Proposition \ref{T_3}.
\end{proof}

We move on to Case I. Recall that the construction in Proposition \ref{S_3} is not successful just because inequality \eqref{i=1 fail} does not hold in this case. To fix this issue, an additional inequality is needed. Define 
\begin{equation}
\label{G2condition}
F_j:=X_k+X_l-Z_j.
\end{equation}
Computations show
\begin{equation*}
\begin{split}
&\langle \nabla F_j,V\rangle=F_j\left(\mathcal{G}-1\right)+\dfrac{3Z_j}{2}\left(\dfrac{1}{3}F_j-F_k-F_l\right).
\end{split}
\end{equation*}
\begin{remark}
\label{G2}
The condition $F_1\equiv F_2\equiv F_3\equiv 0$ is in fact the $G_2$ condition on cohomogeneity one manifold with principal orbit $SU(3)/T^2$. Hence $\cap_{j=1}^3\{F_j\equiv 0\}$ is flow-invariant and it contains the integral curve $\gamma_0$ that represents the complete smooth $G_2$ metric on $M$, which is firstly discovered in \cite{bryant1989construction}\cite{gibbons1990Einstein}.
\end{remark}

In the following text, we still use $\tilde{S_3}$ and $S_3$ to denote invariant sets constructed. If necessary, we use the phrase such as ``$S_3$ for Case I'' to refer to the case in particular. Define
\begin{equation}
\label{set T3's friends}
\begin{split}
&\tilde{S_3}=\bigcap_{j=1}^2 \left\{Z_3-Z_j\geq 0,\quad F_j-F_3\geq 0, \quad X_3\geq 0\right\}\cap \{3F_1+3F_2-F_3\geq 0\}.
\end{split}
\end{equation}
And define 
\begin{equation}
\label{Set T3}
S_3=C\cap \{\mathcal{H}\equiv 1\}\cap P\cap \tilde{S_3}.
\end{equation}
Note that $F_j-F_3\geq 0$ is simply the second defining inequality in the $\tilde{S_3}$ in \eqref{set S3's friends} with $\rho =1$.

It is easy to check that $p_0\in S_3$ hence $S_3$ is nonempty. Since functions $X_3, Z_1, Z_3-Z_2, F_j-F_3$ and $3F_1+3F_2-F_3$
vanish at $p_0$ among those in \eqref{set T3's friends}, the point is in $\partial S_3$. With the same argument as the one for Case II and III, we know that $\gamma_{s_1}$ is trapped in $S_3$ initially if $s_1\geq 0$.

\begin{proposition}
\label{T_3}
Integral curves $\gamma_{s_1}$ to system \eqref{new Ricci-flat} on $C\cap \{\mathcal{H}\equiv 1\}$ emanating from $p_0$ with $s_1\geq 0$ do not escape $S_3$.
\end{proposition}
\begin{proof}
The idea of proving Proposition \ref{T_3} is the same as the one of Proposition \ref{S_3}. Besides, almost all computations for Proposition \ref{S_3} still hold except the one for $F_j-F_3\geq 0$ since \eqref{i=1 fail} is not true for Case I. With the additional inequality, it follows that
\begin{equation*}
\begin{split}
&\left.\langle\nabla (F_j-F_3),V\rangle\right|_{F_j-F_3=0}\\
&=(F_j-F_3)\left(\mathcal{G}-1\right)+\dfrac{3Z_j}{2}\left(\dfrac{1}{3}F_j-F_k-F_3\right)-\dfrac{3Z_3}{2}\left(\dfrac{1}{3}F_3-F_j-F_k\right)\\
&=\dfrac{3Z_j}{2}\left(\dfrac{1}{3}F_j-F_k-F_3\right)-\dfrac{3Z_3}{2}\left(\dfrac{1}{3}F_3-F_j-F_k\right)\quad \text{since $F_j=F_3$}\\
&=\dfrac{3Z_j}{2}\left(\dfrac{1}{3}F_j-F_k-F_j\right)-\dfrac{3Z_3}{2}\left(\dfrac{1}{3}F_j-F_j-F_k\right) \quad \text{since $F_j=F_3$}\\
&=F_k\dfrac{3}{2}(Z_3-Z_j)+F_j(Z_3-Z_j)\\
&=\dfrac{1}{2}(Z_3-Z_j)(3F_j+3F_k-F_3) \quad \text{since $F_j=F_3$}\\
&\geq 0.
\end{split}
\end{equation*}
Notice that we do not drop any non-negative term in the computation above like we do in \eqref{X3-X_1+k(Z_3-Z_1)pre}. The estimate for $\left.\langle\nabla (F_j-F_3),V\rangle\right|_{F_j-F_3=0}$ hence becomes sharper. Finally, we need to show that the additional inequality holds along the integral curves. Indeed, since
\begin{equation*}
\begin{split}
&\left.\langle\nabla (3F_1+3F_2-F_3),V\rangle\right|_{3F_1+3F_2-F_3=0}\\
&=(3F_1+3F_2-F_3)\left(\mathcal{G}-1\right)\\
&\quad +\dfrac{3Z_1}{2}\left(F_1-3F_2-3F_3\right)+\dfrac{3Z_2}{2}\left(F_2-3F_1-3F_3\right)-\dfrac{3Z_3}{2}\left(\dfrac{1}{3}F_3-F_1-F_2\right)\\
&=\dfrac{3Z_1}{2}\left(F_1-3F_2-3F_3\right)+\dfrac{3Z_2}{2}\left(F_2-3F_1-3F_3\right)\quad \text{since $3F_1+3F_2-F_3=0$}\\
&=\dfrac{3Z_1}{2}\left(4F_1-4F_3\right)+\dfrac{3Z_2}{2}\left(4F_2-4F_3\right)\quad \text{since $3F_1+3F_2-F_3=0$}\\
&\geq 0\quad \text{definition of $S_3$ for $i=1$}
\end{split},
\end{equation*}
$3F_1+3F_2-F_3$ remains non-negative along the integral curves. Therefore, integral curves $\gamma_{s_1}$ do not escape $S_3$ in Case I if $s_1\geq 0$.
\end{proof}

\begin{remark}
One may want to integrate the additional inequality in $S_3$ for Case I to the other two cases so that all cases can be discussed by a single construction. Specifically, one can define
$$
F_j:=X_k+X_l-\rho Z_j.
$$
Then the additional inequality analogous to $3F_1+3F_2-F_3\geq 0$ for Case II and III is $
aF_1+aF_2-2bF_3\geq 0.
$
But
\begin{equation*}
\begin{split}
&\left.\langle\nabla (aF_1+aF_2-2bF_3),V\rangle\right|_{aF_1+aF_2-2bF_3=0}\\
&=\frac{aZ_1}{k}(a+2b)(F_1-F_3)+\frac{aZ_2}{k}(a+2b)(F_2-F_3)+\dfrac{\zeta}{k}(aZ_1+aZ_2-2bZ_3)
\end{split},
\end{equation*}
where $\zeta=\frac{(3-d)a-(2+2d)b}{2d}\leq 0$. It only vanishes in Case I. Hence whether $aF_1+aF_2-2bF_3$ is non-negative along the integral curves in $S_3$ is not clear. The analogous $F_j$ defined for Case II and Case III may not have too much meaning after all because there is no special holonomy for odd dimension other than $7$.
\end{remark}

We are ready to construct the compact invariant set mentioned at the beginning of this section. Define 
$$
\hat{S}_3=S_3\cap \{Z_1+Z_2-Z_3\geq 0\}\cap\{Z_1(X_1-X_3)+Z_2(X_2-X_3)\geq 0\}
$$
for all three cases. We have the following lemma.
\begin{lemma}
\label{never escape}
$\hat{S}_3$ is a compact invariant set.
\end{lemma}
\begin{proof}
Because $Z_1+Z_2-Z_3\geq 0$ in $\hat{S}_3$, we can apply Proposition \ref{sharp} so that $Z_1+Z_2$ is bounded above. Then all $Z_j$'s are bounded in $\hat{S}_3$. By conservation law \eqref{conservation law}, we immediately conclude that all variables are bounded. The compactness of $\hat{S}_3$ is hence proved. 

To check that $\hat{S_3}$ is flow invariant, consider the hyperplane $Z_1+Z_2-Z_3=0$. It follows that 
\begin{equation*}
\begin{split}
\left. \langle\nabla(Z_1+Z_2-Z_3),V\rangle \right|_{Z_1+Z_2-Z_3=0}&=(Z_1+Z_2-Z_3)\left(\mathcal{G}-\frac{1}{d}\right)+2Z_1X_1+2Z_2X_2-2Z_3X_3\\
&=2Z_1(X_1-X_3)+2Z_2(X_2-X_3) \quad \text{since $Z_1+Z_2-Z_3=0$}\\
&\geq 0 \quad \text{definition of $\hat{S}_3$}
\end{split}.
\end{equation*}
On hypersurface $Z_1(X_1-X_3)+Z_2(X_2-X_3)=0$, we have
\begin{equation}
\label{messy!}
\begin{split}
&\left. \langle\nabla(Z_1(X_1-X_3)+Z_2(X_2-X_3)),V\rangle \right|_{Z_1(X_1-X_3)+Z_2(X_2-X_3)=0}\\
&=\left. \left\langle\nabla \left(Z_3\left(\frac{Z_1}{Z_3}(X_1-X_3)+\frac{Z_2}{Z_3}(X_2-X_3)\right)\right),V\right\rangle \right|_{Z_1(X_1-X_3)+Z_2(X_2-X_3)=0} \\
&= Z_3\left(\mathcal{G}-\frac{1}{d}+2X_3\right)\left(\frac{Z_1}{Z_3}(X_1-X_3)+\frac{Z_2}{Z_3}(X_2-X_3)\right)+Z_3\left(2\dfrac{Z_1}{Z_3}(X_1-X_3)^2+2\dfrac{Z_2}{Z_3}(X_2-X_3)^2\right)\\
&\quad+Z_1\left((X_1-X_3)\left(\mathcal{G}-1\right)+\mathcal{R}_1-\mathcal{R}_3\right)+Z_2\left((X_2-X_3)\left(\mathcal{G}-1\right)+\mathcal{R}_2-\mathcal{R}_3\right)\\
&=2Z_1(X_1-X_3)^2+2Z_2(X_2-X_3)^2+Z_1(\mathcal{R}_1-\mathcal{R}_3)+Z_2(\mathcal{R}_2-\mathcal{R}_3)\\
&\quad \text{ since $Z_1(X_1-X_3)+Z_2(X_2-X_3)=0$}\\
&\geq Z_1(\mathcal{R}_1-\mathcal{R}_3)+Z_2(\mathcal{R}_2-\mathcal{R}_3)\\
&=Z_1(Z_3-Z_1)(aZ_2-2b(Z_3+Z_1))+Z_2(Z_3-Z_2)(aZ_1-2b(Z_3+Z_2))
\end{split}
\end{equation}
For distinct $j,k\in\{1,2\}$, take $A_j=Z_j(Z_3-Z_j)$ and $B_j=aZ_k-2b(Z_3+Z_j)$. Apply identity $$A_1B_1+A_2B_2=\frac{1}{2}((A_1+A_2)(B_1+B_2)+(A_1-A_2)(B_1-B_2)).$$ Then the computation \eqref{messy!} continues as
\begin{equation}
\label{usefulfor p3}
\begin{split}
&= \frac{1}{2}(Z_1(Z_3-Z_1)+Z_2(Z_3-Z_2))((a-2b)(Z_1+Z_2)-4bZ_3)+\frac{1}{2}(Z_1-Z_2)^2(Z_1+Z_2-Z_3)(a+2b)\\
&\geq \frac{1}{2}(Z_1(Z_3-Z_1)+Z_2(Z_3-Z_2))(a-6b)Z_3 \quad \text{since $Z_1+Z_2\geq Z_3$}\\
&\geq 0 \quad \text{Remark \ref{scalr of spher basic ineq}}
\end{split}.
\end{equation}
Therefore, $\hat{S}_3$ is flow-invariant.
\end{proof}

\begin{remark}
\label{S2}
By the symmetry between $(X_2,Z_2)$ and $(X_3,Z_3)$, constructions of $S_3$ and $\hat{S}_3$ above can be carried over to defining $S_2$ and $\hat{S}_2$. With the same arguments, it can be shown that $\gamma_{s_1}$ does not escape $S_2$ whenever $s_1\leq 0$ and $\hat{S}_2$ is a compact invariant set.
\end{remark}

\begin{remark}
\label{Mcontainy}
It is clear that $p_0\in\partial \hat{S}_3$. One can check that $\gamma_{0}$ is trapped in $\hat{S}_3$ initially. Hence the long time existence for $\gamma_0$ is proved. By Remark \ref{y a subcase}, it is trapped in $\hat{S}_3\cap\{X_2\equiv X_3, Z_2\equiv Z_3\}$. Hence $\hat{S}_3$ can be used to prove the long time existence for the special case where $(X_2,Z_2)\equiv (X_3,Z_3)$ is imposed. In fact, the compact invariant set for cohomogeneity one manifolds of two summands can be constructed by a little modification on $\hat{S}_3\cap \{X_2\equiv X_3, Z_2\equiv Z_3\}$, reproducing the same result in \cite{bohm_inhomogeneous_1998}\cite{wink2017cohomogeneity}. For Case I in particular, $\gamma_0$ represents the complete $G_2$ metric discovered in in \cite{bryant1989construction}\cite{gibbons1990Einstein}.
\end{remark}

\begin{remark}
\label{positivericci}
Not only $\mathcal{R}_3$ is non-negative in $\hat{S}_3$. This is in fact  the case for all $\mathcal{R}_j$'s. For distinct $j,k\in\{1,2\}$, we have
\begin{equation*}
\begin{split}
\mathcal{R}_j=aZ_3Z_k+b(Z_j^2-Z_k^2-Z_3^2)&\geq aZ_3Z_k+b(Z_j^2+Z_k^2-(Z_j+Z_k)^2)\quad \text{by definition of $\hat{S}_3$}\\
&=aZ_3Z_k-2bZ_jZ_k\\
&\geq (a-2b)Z_jZ_k \quad \text{by definition of $\hat{S}_3$}\\
&\geq 0
\end{split}.
\end{equation*}
Therefore, one geometric feature of complete Ricci-flat metrics represented by $\gamma_0$ is that hypersurface has positive Ricci tensor for all $t\in(0,\infty)$. As discussed in Remark \ref{posiscaler}, Ricci-flat metrics represented by $\gamma_{s_1}$ with $s_1\neq 0$ does not hold such a property.
\end{remark}

Although $\gamma_{s_1}$ is trapped in $S_3$ if $s_1\geq 0$, functions $Z_1+Z_2-Z_3$ and $Z_1(X_1-X_3)+Z_2(X_2-X_3)$ are negative initially if $s_1> 0$. Hence $\gamma_{s_1}$ is not trapped in $\hat{S}_3$ initially if $s_1> 0$. 
To include the case where $s_1> 0$, we need to enlarge $\hat{S}_3$ a little bit so that it initially traps all $\gamma_{s_1}$ with $s_1$ close enough to zero. That leads us to the second step of our construction.

\subsection{Entrance Zone}
\label{Entrance Zone}
In this section, we assume $s_1>0$ and work with the set $S_3$.  We construct an entrance zone that forces $\gamma_{s_1}$ to enter $\hat{S}_3$ eventually. Our goal is to show that for all small enough $s_1>0$, $\gamma_{s_1}$ will enter $\hat{S}_3$ in a compact set. 
As shown in computation \eqref{messy!}, it is more convenient to compute with variables
$
\omega_1=\frac{Z_1}{Z_3}$ and $\omega_2=\frac{Z_2}{Z_3},
$
whose respective derivatives are
$
\omega_1'=2\omega_1(X_1-X_3)$ and $\omega_2'=2\omega_2(X_2-X_3).
$
By the definition of $S_3$, we have $Z_3\geq Z_1,Z_2$. Therefore $\omega_1,\omega_2\in[0,1]$. For another point of view, we can also consider the problem on $\omega_1\omega_2$-plane as shown Figure \ref{Fig1}. Whatever $\gamma_{s_1}$ looks like, we can always project its $Z_1$ and $Z_2$ coordinate to $\omega_1 \omega_2$-plane. And we want to prove the projection is bounded away from $(0,0)$ and hopefully going through the line
$$l_0\colon \omega_1+\omega_2-1=0,$$
which is the projection of hyperplane $Z_1+Z_2-Z_3=0.$ Note that any homogeneous variety in $Z_j$'s of degree $D$ can be projected to an algebraic curve on $\omega_1\omega_2$-plane by dividing by $Z_3^D$. Before the construction, we establish the following basic fact.

\begin{figure}[h!]
\begin{center}
\includegraphics[width=2.5in]{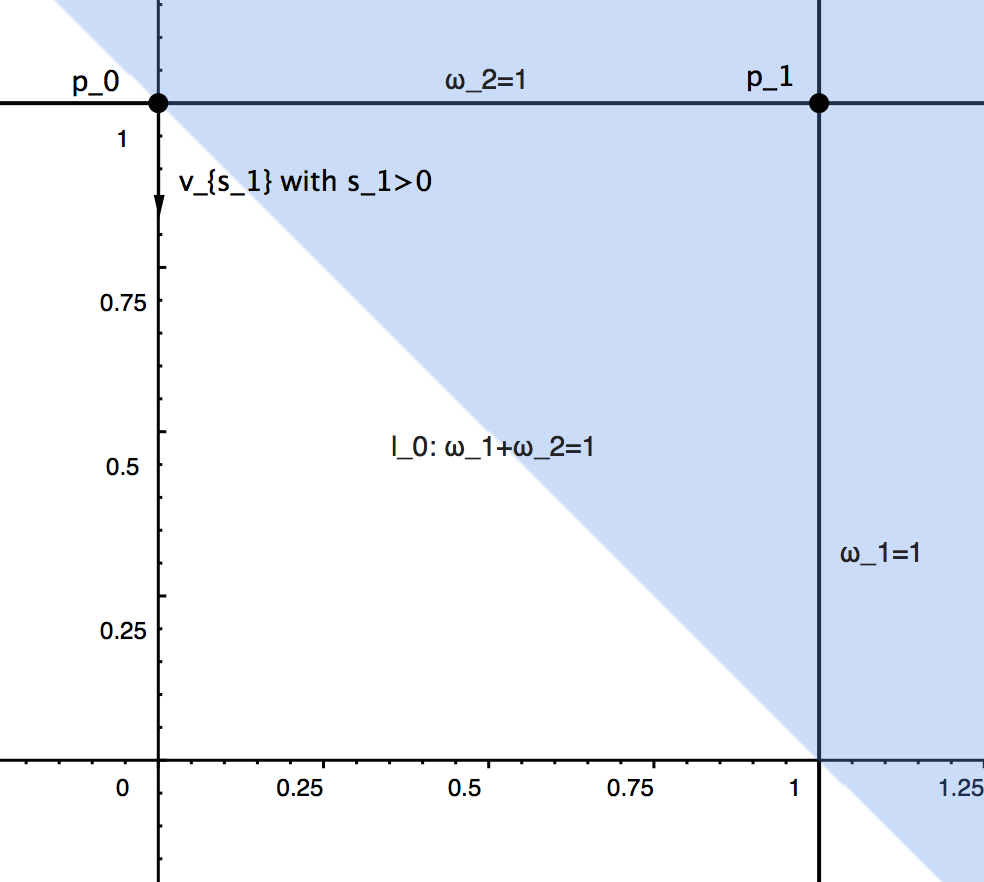}
\caption{Projection to $\omega_1\omega_2$-plane}
\label{Fig1}
\end{center}
\end{figure}

\begin{proposition}
\label{at least to Z_1=Z_2}
$\frac{Z_1}{Z_2}$ is strictly increasing along $\gamma_{s_1}$ as long as $Z_2> Z_1$.
\end{proposition}
\begin{proof}
Initially we have $(X_1-X_2)(p_0)=\frac{1}{d}$. If $Z_2>Z_1$, we have
\begin{equation}
\begin{split}
\left.\left(X_1-X_2\right)'\right|_{X_1-X_2=0}&=(X_1-X_2)(\mathcal{G}-1)+\mathcal{R}_1-\mathcal{R}_2\\
&=(Z_2-Z_1)(aZ_3-2bZ_1-2bZ_2)\\
&\geq (Z_2-Z_1)(a-4b)Z_2 \quad \text{since $Z_3\geq Z_2>Z_1$}\\
&> 0 \quad \text{Remark \ref{scalr of spher basic ineq}}
\end{split}.
\end{equation}
Hence $X_1-X_2>0$ along $\gamma_{s_1}$ when $Z_2>Z_1$. But then
\begin{equation}
\begin{split}
\left(\frac{Z_1}{Z_2}\right)'&=2\frac{Z_1}{Z_2}(X_1-X_2)> 0
\end{split}
\end{equation}
when $Z_2>Z_1$. Therefore $\frac{Z_1}{Z_2}$ is strictly increasing along $\gamma_{s_1}$ as long as $Z_2>Z_1$.
\end{proof}

Substitute solution \eqref{linearized solution<} of linearized equation to $\mathcal{R}_1-\mathcal{R}_3$ and $\mathcal{R}_3-\mathcal{R}_2$. It is clear that they are positive initially. Hence at the beginning, the integral curve is trapped in
\begin{equation}
\label{U_0}
U_0=S_3\cap \{Z_1+Z_2-Z_3\leq 0, \mathcal{R}_1-\mathcal{R}_3\geq 0, \mathcal{R}_3-\mathcal{R}_2\geq 0\},
\end{equation}
whose projection on $\omega_1\omega_2$-plane for all three cases is illustrated in Figure \ref{Fig 2}.

\begin{figure}[h!]
\begin{subfigure}{.5\textwidth}
  \centering
  \includegraphics[width=1\linewidth]{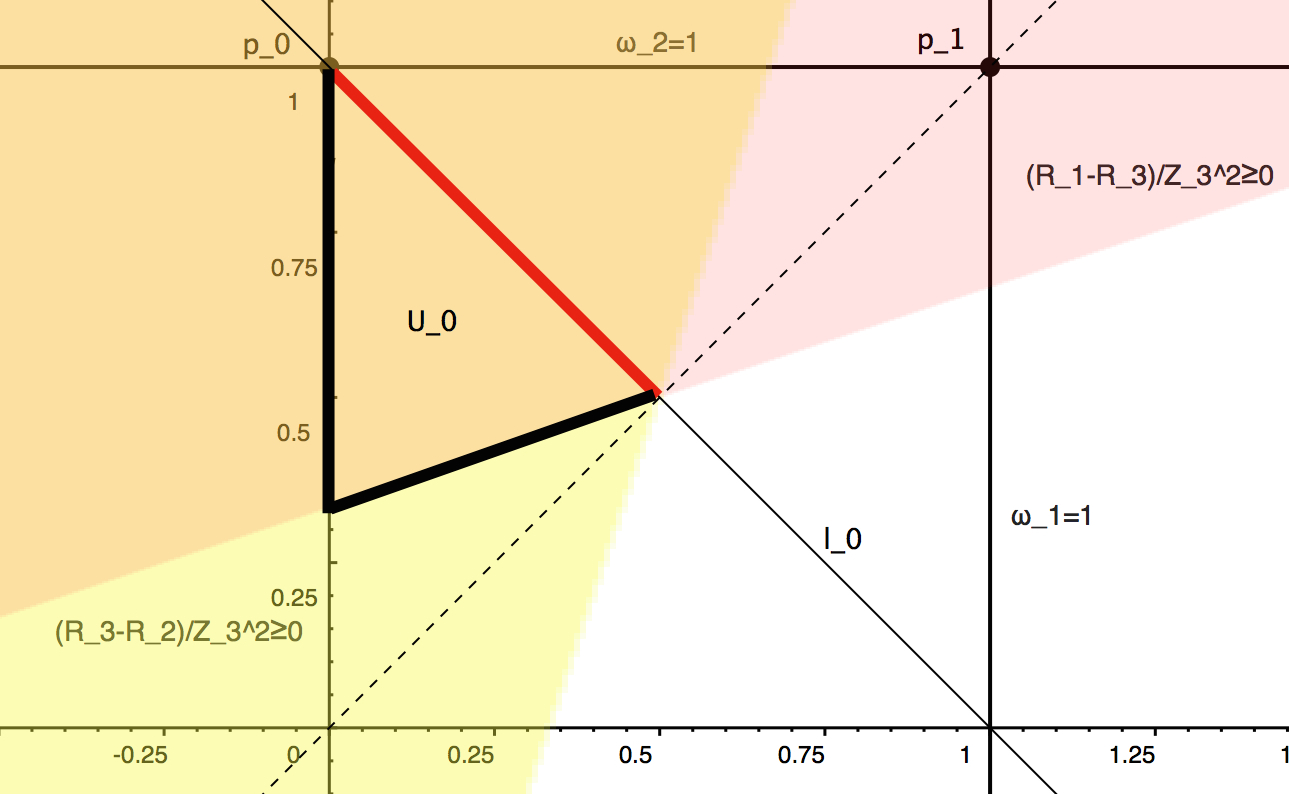}
  \caption{Case I}
  \label{fig:sfig2-0}
\end{subfigure}
\begin{subfigure}{.5\textwidth}
  \centering
  \includegraphics[width=1\linewidth]{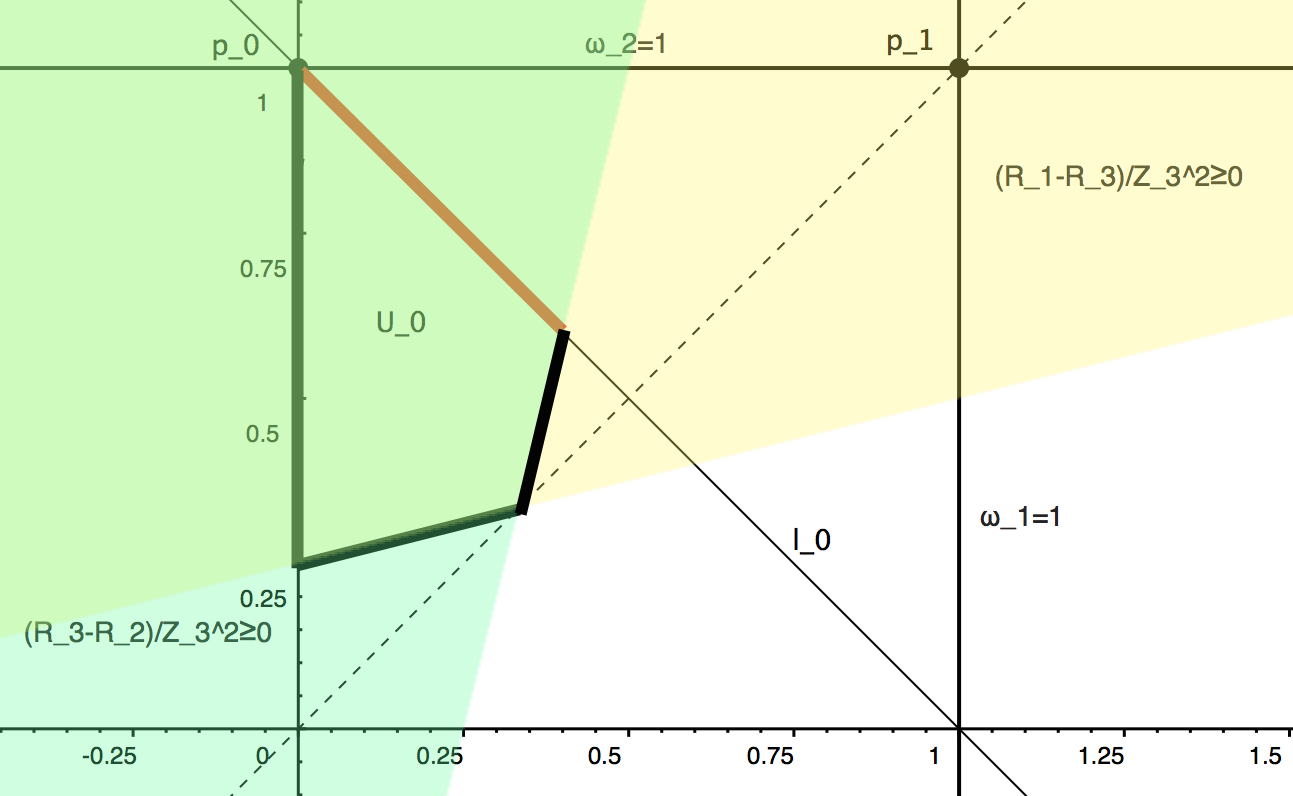}
  \caption{Case II}
  \label{fig:sfig2-1}
\end{subfigure}
\begin{subfigure}{.5\textwidth}
  \centering
  \includegraphics[width=1\linewidth]{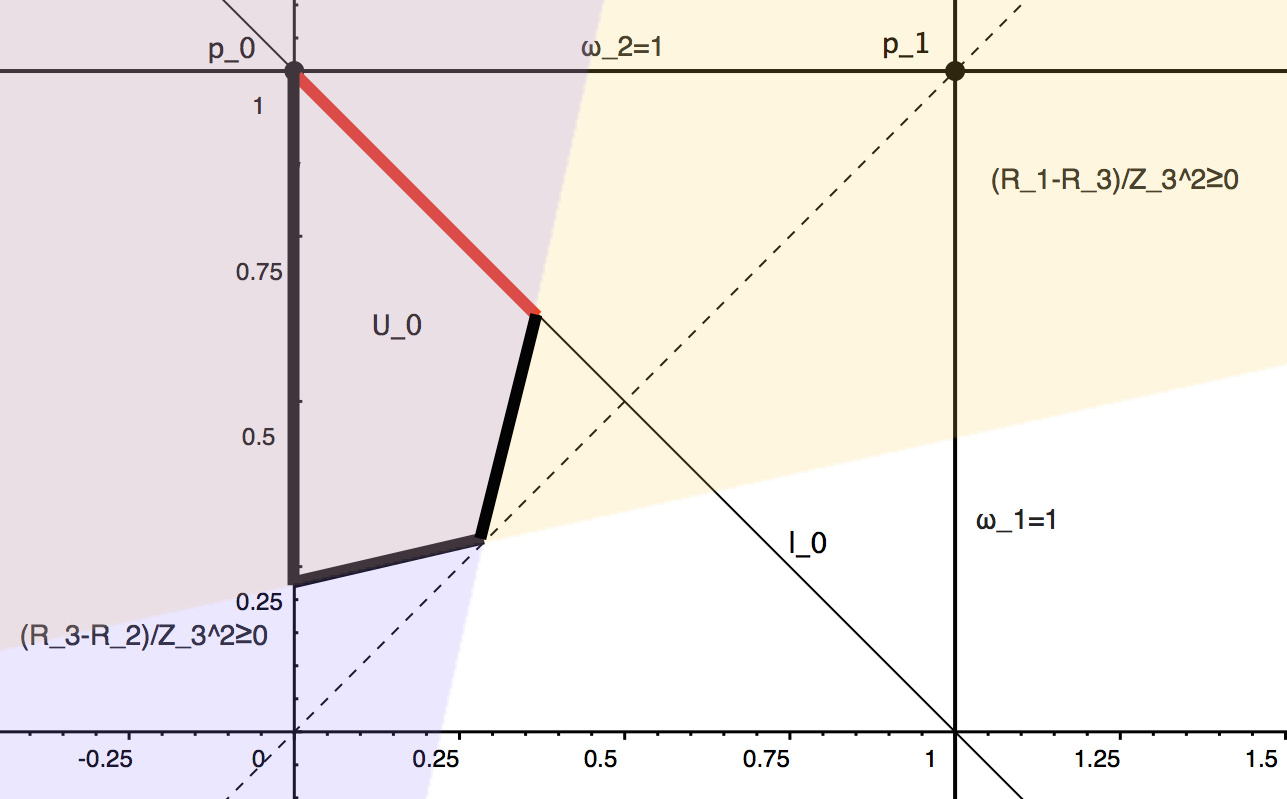}
  \caption{Case III}
  \label{fig:sfig2-2}
\end{subfigure}
\caption{Projection of $U_0$ (enclosed by bold line segments) on $\omega_1\omega_2$-plane for all three cases}
\label{Fig 2}
\end{figure}
 
By Proposition \ref{at least to Z_1=Z_2}, we know that in principal, the projection of $\gamma_{s_1}$ on $\omega_1\omega_2$-plane can get arbitrarily closed to $\omega_1-\omega_2=0$, represented the dashed lines in Figure \ref{Fig 2}. Therefore, an integral curve that is initially trapped in $U_0$ has to escape. The question is whether it will escape $U_0$ through $Z_1+Z_2-Z_3=0$, represented by the red line segment. It turns out that a subset of $U_0$ can be constructed in a way that it contains a part of $Z_1+Z_2-Z_3=0$ and $\gamma_{s_1}$ has to escape that subset through $Z_1+Z_2-Z_3=0$. Specifically, the construction is based on the following three ideas.

1. Since $Z_1+Z_2-Z_3\leq 0$ initially along $\gamma_{s_1}$, the main task is to bound $Z_3$ from above. For computation conveniences, we prefer to bound $Z_3$ from above by some homogeneous algebraic varieties in $Z_j$'s. In other words, defining inequalities of the entrance zone should include $Z_1+Z_2-Z_3\leq 0$ and $B(Z_1,Z_2,Z_3)\geq 0$ for some homogeneous polynomial $B$ in $Z_j$'s.

2. In order to show that $\gamma_{s_1}$ does not escape through $B=0$, we need to show that $\left.\langle \nabla(B),V\rangle\right|_{B=0}$ is non-negative along $\gamma_{s_1}$ in the entrance zone. This idea is discussed in the proof of Proposition \ref{Set S3}. It might be difficult to determine the sign of $\left.\langle \nabla(B),V\rangle\right|_{B=0}$ even we are allowed to mod out $B=0$ in the computation result. But notice that $B':=\left.\langle \nabla(P),V\rangle\right|_{B=0}=0$ vanishes at $p_0$, and inequality $B'\geq 0$ can potentially be added to the definition of the entrance zone.

3. If we want to impose $B'\geq 0$, the trade-off is to show that $\left.\langle \nabla(B'),V\rangle\right|_{B'=0}\geq 0$ along $\gamma_{s_1}$ in the entrance zone. The homogeneous polynomial $B$ that we find consists of two parameters. They allow us to tune the entrance zone to satisfy some technical inequalities. Once these inequalities are satisfied, we can show that $\left.\langle \nabla(B'),V\rangle\right|_{B'=0}\geq 0$ in the entrance zone and $\gamma_{s_1}$ is forced to escape through $Z_1+Z_2-Z_3=0$.

We first proceed the construction by having those technical inequalities in part 3 ready. In this process, the first parameter for $B$ is introduced and how they interact with these technical inequalities are explained. Then we reveal the definition for $B$ and its last parameter.

\begin{proposition}
\label{X_3+X_2>0}
In $S_3$, $X_2+X_3> 0$ along $\gamma_{s_1}$ always.
\end{proposition}
\begin{proof}
It is clear that $X_2+X_3$ is positive initially along the curves.
Since
\begin{equation}
\begin{split}
\left. \langle\nabla(X_2+X_3),V\rangle \right|_{X_2+X_3=0}&=(X_2+X_3)(\mathcal{G}-1)+\mathcal{R}_2+\mathcal{R}_3\\
&=\mathcal{R}_2+\mathcal{R}_3\quad \text{since $X_2+X_3=0$}\\
&\geq Z_1(a(Z_2+Z_3)-2bZ_1)\\
&> 0 \quad \text{since $Z_3\geq Z_1$ and $a-2b=d-1>0$}
\end{split},
\end{equation}
$X_2+X_3$ stays positive along $\gamma_{s_1}$.
\end{proof}

\begin{proposition}
\label{compare X_3 with X_2}
For any fixed $\delta\geq 0$, $X_3- (1+\delta)X_2>0$ initially along $\gamma_{s_1}$ and stay positive in the region where $\mathcal{R}_3-(1+\delta) \mathcal{R}_2\geq 0$.
\end{proposition}
\begin{proof}
Substitute solution \eqref{linearized solution<} of linearized equation to $X_3-(1+\delta)X_2$. We have 
$$
(2+\delta)s_1e^{\frac{\eta}{d}}-a\delta s_0e^\frac{2\eta}{d} \sim (2+\delta)s_1e^{\frac{\eta}{d}}>0
$$
near $p_0$.
Since 
\begin{equation}
\label{comp for X_3and X_2}
\begin{split}
\left. \langle\nabla(X_3-(1+\delta X_2)),V\rangle \right|_{X_3-(1+\delta) X_2=0} &=(X_3-(1+\delta) X_2)\left(\mathcal{G}-1\right)+\mathcal{R}_3-(1+\delta)\mathcal{R}_2\\
&=\mathcal{R}_3-(1+\delta)\mathcal{R}_2\quad \text{since $X_3-(1+\delta) X_2=0$},
\end{split}
\end{equation}
the proof is complete.
\end{proof}

Define
\begin{equation}
\label{Udelta}
U_\delta=U_0\cap \{\mathcal{R}_3-(1+\delta)\mathcal{R}_2\geq 0\}.
\end{equation}
It is easy to check that $U_\delta$ is a subset of $U_0$ and $\gamma_{s_1}$ is initially trapped in $U_{\delta}$ if $s_1>0$. Therefore, $X_3-(1+\delta)X_2>0$ when $\gamma_{s_1}$ is in $U_\delta$ by Proposition \ref{compare X_3 with X_2}.

The fixed value of $\delta$ needs to be picked in a certain range for the following two technical reasons. Firstly, we want inequality $X_3-(1+\delta)X_2> 0$ to hold at least until $\gamma_{s_1}$ enters $\hat{S}_3$. Hence by proposition \ref{compare X_3 with X_2}, we need to pick $\delta$ that make $\mathcal{R}_3-(1+\delta)\mathcal{R}_2\geq 0$ contains a subset of $Z_1+Z_2-Z_3=0$. Secondly, because $U_0\subset S_3\cap \{Z_2-Z_1>0\}$ and the behavior of $\gamma_{s_1}$ is better known in $U_0$, we want $\gamma_{s_1}$ passes though the part of $Z_1+Z_2-Z_3=0$ that $Z_2-Z_1\geq 0$ is satisfied. In summary, we have the following proposition.

\begin{proposition}
\label{R3-R2}
If $\delta\in \left(\frac{6b-a}{2(d-1)},\frac{4b}{d-1}\right)$, then $\{\mathcal{R}_3-(1+\delta)\mathcal{R}_2\geq 0\}$ contains a subset of $\{Z_1+Z_2-Z_3=0\}\cap \{Z_2-Z_1>0\}$ in $U_0$. 
\end{proposition}
\begin{proof}
If $\delta\in\left(\frac{6b-a}{2(d-1)},\frac{4b}{d-1}\right)$, then we have 
$\frac{4b-(d-1)\delta}{(d-1)(1+\delta)}\in (0,1)$. 
Suppose $\frac{4b-(d-1)\delta}{(d-1)(1+\delta)}\geq \frac{Z_1}{Z_2}$, then 
\begin{equation}
\label{R3R2}
\begin{split}
&\left.(\mathcal{R}_3-(1+\delta)\mathcal{R}_2)\right|_{Z_1+Z_2-Z_3=0}\\
&=(Z_3-Z_2)(2b(Z_3+Z_2)-aZ_1)-\delta(aZ_1Z_3+b(Z_2^2-Z_1^2-Z_3^2))\quad \text{since }Z_1+Z_2-Z_3=0\\
&=Z_1(2b(2Z_2+Z_1)-aZ_1)-\delta(aZ_1Z_2+aZ_1^2-2bZ_1^2-2bZ_1Z_2)\quad \text{since }Z_1+Z_2-Z_3=0\\
%&=Z_1(-(a-2b)(1+\delta)Z_1+(4b-(a-2b)\delta)Z_2)\\
&=Z_1(-(d-1)(1+\delta)Z_1+(4b-(d-1)\delta)Z_2)\quad \text{Remark \ref{scalr of spher basic ineq}}\\
&\geq 0
\end{split}.
\end{equation}
The proof is complete.
\end{proof}

\begin{remark}
Perhaps a better way to illustrate Proposition \ref{R3-R2} is to consider the projection on the $\omega_1\omega_2$-plane. For $\mathcal{R}_3-(1+\delta)\mathcal{R}_2=0$, we obtain an algebraic curve
$$l_1\colon (1-\omega_2)(2b(1+\omega_2)-a\omega_1)-\delta(a\omega_1+b(\omega_2^2-\omega_1^2-1))=0.$$
\begin{figure}[h!]
\begin{center}
\includegraphics[width=2.6in]{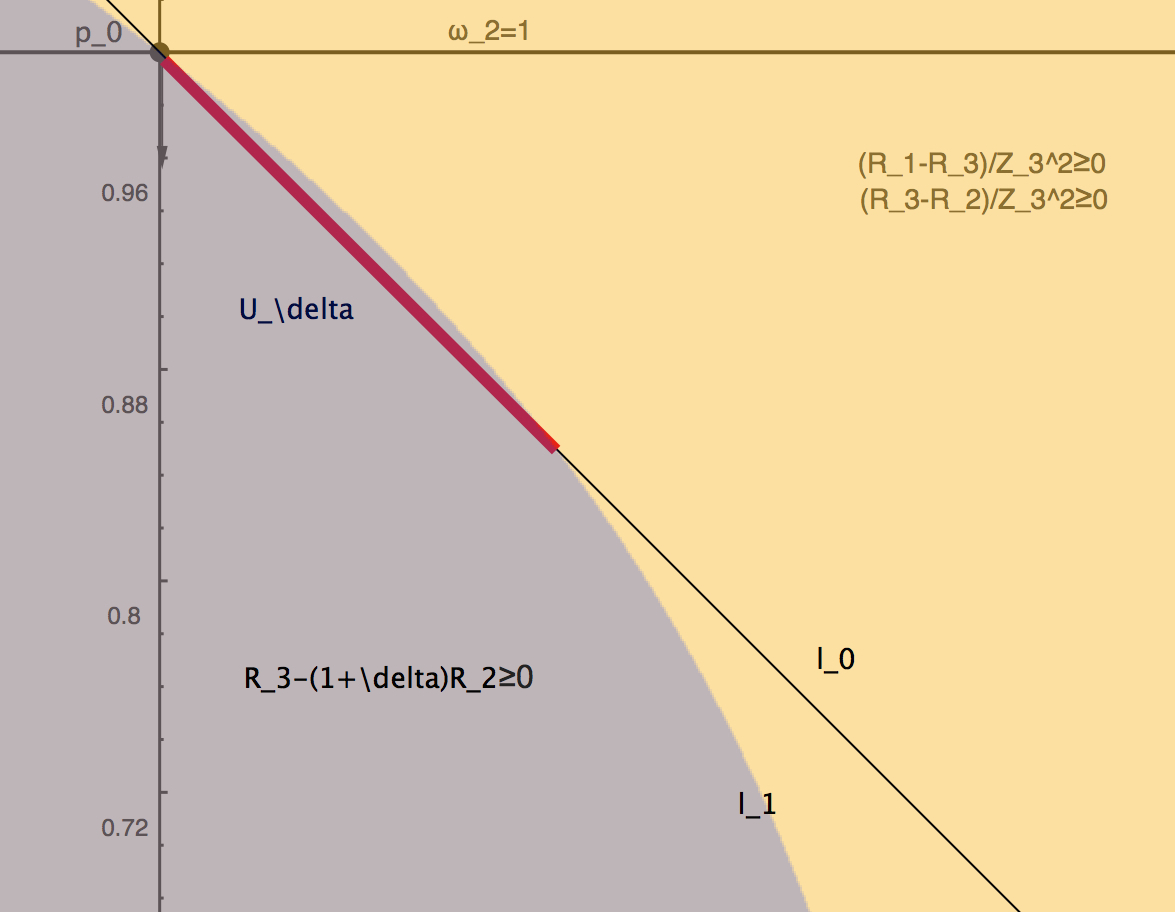}
\caption{$\delta=0.7$ for Case I}
\label{Fig 3}
\end{center}\end{figure}
Straightforward computation shows that $l_1$ intersect with $\omega_1+\omega_2=1$ at points $\left(0,1\right)$ and $\left(\frac{4b-\delta(d-1)}{a+2b},\frac{(d-1)(1+\delta)}{a+2b}\right)$. If $\delta \in \left(\frac{6b-a}{d-1},\frac{4b}{d-1}\right)$, then the second intersection point $\left(\frac{4b-\delta(d-1)}{a+2b},\frac{(d-1)(1+\delta)}{a+2b}\right)$ is in the region where $\omega_2-\omega_1> 0$. Hence $U_\delta$, denoted by the darker area in Figure \ref{Fig 3}, can include a segment of $l_0$ in $U_0$, represented by the bold segment, that is away from $\omega_1-\omega_2=0$.
\end{remark}

\begin{remark}
Note that Case I is the only case where the admissible $\delta$ must be positive.
\end{remark}

The entrance zone we construct is a subset of $U_\delta$. We impose $\delta\in \left(0,\frac{4b}{d-1}\right).$ As shown in the following technical proposition, $\delta>0$ is needed for the sake of conveniences. The first parameter in the definition of $B$ is also introduced.

\begin{proposition}
\label{from X to scalar curvature!}
In $U_\delta$, we can find a $p$ large enough such that 
\begin{equation}
\label{from X to scalar curvature}
((X_1-X_2)+(p-1)(X_3-X_2))(X_1-X_2+(p+1)(X_3-X_2))\geq \frac{1-\mathcal{G}}{d(d-1)}
\end{equation}
along $\gamma_{s_1}$ in $U_\delta$.
\end{proposition}
\begin{proof}
Since $X_1=\frac{1}{d}-X_2-X_3$, we can write inequality \eqref{from X to scalar curvature} with respect to $\tilde{X}=X_3+X_2$ and $\tilde{Y}=X_3-X_2$. Straight forward computation shows that inequality \eqref{from X to scalar curvature} is equivalent to 
\begin{equation}
\label{huge inequality in X}
\begin{split}
&\left(\left(p-\frac{1}{2}\right)\left(p+\frac{3}{2}\right)+\frac{1}{2(d-1)}\right)\tilde{Y}^2-\left(3p+\frac{3}{2}\right)\tilde{X}\tilde{Y}\\
&+\left(\frac{9}{4}+\frac{3}{2(d-1)}\right)\tilde{X}^2+\frac{2p+1}{d}\tilde{Y}-\frac{3d-1}{d(d-1)}\tilde{X}\geq 0.
\end{split}
\end{equation}
Note that $\tilde{X}$ and $\tilde{Y}$ are positive along $\gamma_{s_1}$ in $U_\delta$ by Proposition \ref{X_3+X_2>0} and \ref{compare X_3 with X_2}.
Moreover, in $U_\delta$, we have $X_3- (1+\delta)X_2> 0$ along $\gamma_{s_1}$ by Proposition \ref{compare X_3 with X_2}. Rewrite this condition in terms of $\tilde{X}$ and $\tilde{Y}$ so we have 
$
(2+\delta)\tilde{Y}-\delta\tilde{X}> 0
$ along $\gamma_{s_1}$ in $U_\delta$. Hence the LHS of \eqref{huge inequality in X} is larger than
\begin{equation}
\begin{split}
&\left(\left(\left(p-\frac{1}{2}\right)\left(p+\frac{3}{2}\right)+\frac{1}{2(d-1)}\right)\frac{\delta}{2+\delta}-\left(3p+\frac{3}{2}\right)\right)\tilde{X}\tilde{Y}\\
&+\left(\frac{9}{4}+\frac{3}{2(d-1)}\right)\tilde{X}^2+\left(\frac{2p+1}{d}\frac{\delta}{2+\delta}-\frac{3d-1}{d(d-1)}\right)\tilde{X}
\end{split}.
\end{equation}
Since $\delta\in \left(0,\frac{4b}{d-1}\right)$ is fixed, we can choose $p$ large enough so that
\begin{equation}
\begin{split}
\label{lower bound for p}
&\left(\left(p-\frac{1}{2}\right)\left(p+\frac{3}{2}\right)+\frac{1}{2(d-1)}\right)\frac{\delta}{2+\delta}\geq 3p+\frac{3}{2}\\
&\frac{2p+1}{d}\frac{\delta}{2+\delta}\geq \frac{3d-1}{d(d-1)}
\end{split}
\end{equation}
are satisfied. Then inequality \eqref{huge inequality in X} is satisfied.
\end{proof}

Now we are ready to reveal the definition for $B$ and its last parameter. Define 
$$B_{p,k}(Z_1,Z_2,Z_3):=kZ_1Z_3^{p+1}-Z_2^p(Z_3-Z_2)^2.$$ For a fixed $\delta\in \left(0,\frac{4b}{d-1}\right)$, choose a $p$  that  satisfies inequalities \eqref{lower bound for p}. Then define
\begin{equation}
\label{yes}
\begin{split}
U_{(\delta,p,k)}&=S_3\cap \{Z_1+Z_2-Z_3\leq 0\}\cap \{B_{p,k}\geq 0\}\\
&\quad \cap\{(Z_3-Z_2)(X_1-X_3)+(p(Z_3-Z_2)-2Z_2)(X_3-X_2)\geq 0\},
\end{split}
\end{equation}
More requirements on the choice of $p$ and $k$ are added later. Before that, we prove the following.
\begin{proposition}
\label{initial trap}
For any fixed $k>0$, $\gamma_{s_1}$ is initially trapped in $U_{(\delta,p,k)}$ as long as $s_1\in \left(0,\sqrt{\frac{ks_0(d+1)}{16d}}\right)$.
\end{proposition}
\begin{proof}
With discussion in Section \ref{Completensese}, we know that $\gamma_{s_1}$ is initially in $S_3$ if $s_1>0$. Since all the other inequalities in \eqref{yes} reach equality at $p_0$, we need to substitute solution \eqref{linearized solution<} of linearized equation in each one of them. For $Z_1+Z_2-Z_3$,
we have 
\begin{equation}
(d+1)s_0e^{\frac{2\eta}{d}}-4s_1e^{\frac{\eta}{d}}\sim-4s_1e^{\frac{\eta}{d}}<0
\end{equation}
if $s_1<0$.

Substitute solution \eqref{linearized solution<} of linearized equation to $kZ_1Z_3^{p+1}-Z_2^p(Z_3-Z_2)^2$, we have
\begin{equation}
\begin{split}
&k(d+1)s_0e^{\frac{2\eta}{d}}\left(\frac{1}{d}-as_0e^{\frac{2\eta}{d}}+2s_1e^{\frac{\eta}{d}}\right)^{p+1}-\left(\frac{1}{d}-as_0e^{\frac{2\eta}{d}}-2s_1e^{\frac{\eta}{d}}\right)^p 16 s_1^2e^\frac{2\eta}{d}\\
&\sim \left(\frac{1}{d}\right)^p\left(\frac{ks_0(d+1)}{d}-16s_1^2\right) e^{\frac{2\eta}{d}}.
\end{split}
\end{equation}
Hence $kZ_1Z_3^{p+1}-Z_2^p(Z_3-Z_2)^2>0$ initially along the projection of $\gamma_{s_1}$ when $s_1^2<\frac{ks_0(d+1)}{16d}.$

Finally, for $(Z_3-Z_2)(X_1-X_3)+(p(Z_3-Z_2)-2Z_2)(X_3-X_2)$, we have
\begin{equation}
\label{linearised deriv}
4s_1e^{\frac{\eta}{d}}\left(\frac{1}{d}-3as_0e^{\frac{2\eta}{d}}-s_1e^{\frac{\eta}{d}}\right)+\left(4ps_1e^{\frac{\eta}{d}}-2\left(\frac{1}{d}-as_0e^{\frac{2\eta}{d}}-2s_1e^{\frac{\eta}{d}}\right)\right)2s_1e^{\frac{\eta}{d}}\sim (4+8p)s_1^2e^{\frac{2\eta}{d}}>0.
\end{equation}

Hence $\gamma_{s_1}$ is indeed trapped in $U_{(\delta,p,k)}$ initially when $s_1\in \left(0,\sqrt{\frac{ks_0(d+1)}{16d}}\right)$.
\end{proof}

We now specify our choice for $p$ and $k$. Projected to the $\omega_1\omega_2$-plane, the first two inequalities in \eqref{yes} is equivalent to
$$
\frac{\omega_2^p(1-\omega_2)^2}{k}\leq \omega_1\leq 1-\omega_2.
$$
Write $l_0$ as a function $\mathcal{C}_0(\omega_2)=1-\omega_2$. Define $l_2\colon \mathcal{C}_2(\omega_2)=\frac{\omega_2^p(1-\omega_2)^2}{k}$.
It is clear that $\mathcal{C}_0-\mathcal{C}_2=0$ at $\omega_2=1$.
Our goal is to choose $p$ and $k$ so that $\mathcal{C}_0-\mathcal{C}_2$ vanishes again at some $\omega_*<1$. Then we define $\hat{U}_{(\delta,p,k)}$ to be the compact subset of $U_{(\delta,p,k)}$ where $\omega_2\in [\omega_*,1]$ and $\mathcal{C}_0>\mathcal{C}_2$ for $\omega_2\in (\omega_*,1)$. Moreover, because we want to utilize Proposition \ref{from X to scalar curvature!}, parameters $p$ and $k$ are chosen to guarantee that $\omega_*$ is not too small so that $\hat{U}_{(\delta,p,k)}\subset U_{\delta}.$ Specifically, we have the following proposition.

\begin{proposition}
\label{verytechnicalthatitsiboring}
Let $p\geq 2$ be a fixed number large enough that it satisfies inequalities \eqref{lower bound for p} and 
\begin{equation}
\label{more lower bound for p}
\frac{p}{p+1}\geq \frac{(d-1)(1+\delta)}{a+2b}.
\end{equation}
Let $k>0$ be a number small enough so that
\begin{equation}
\label{small k}
k<\left(\frac{p}{p+1}\right)^p\frac{1}{p+1}.
\end{equation}
Then there exists some $\omega_*\in \left(\frac{p}{p+1},1\right)$ such that
\begin{equation}
\label{hatyes}
\hat{U}_{(\delta,p,k)}:=U_{(\delta,p,k)}\cap \{Z_2-\omega_*Z_3\geq 0\}
\end{equation} is a compact subset of $U_\delta$.
\end{proposition}
\begin{proof}
Although the proposition is true as long as $p>0,$ the technical condition $p\geq 2$ is imposed for computations  in Lemma \ref{upper X_3} and \eqref{Bump Derivative}.
We first claim that $p$ exists. Because $\delta$ is a fixed number in $(0,\frac{4b}{d-1})$, we have $\frac{(d-1)(1+\delta)}{a+2b}<\frac{d-1+4b}{a+2b}=1.$
Hence we can choose $p$ large enough on top of inequalities \eqref{lower bound for p} to satisfies inequalities \eqref{more lower bound for p}.

Consider the function 
\begin{equation}
\label{function C}
\mathcal{C}=\mathcal{C}_0-\mathcal{C}_2=1-\omega_2-\frac{\omega_2^p(1-\omega_2)^2}{k}=\frac{1-\omega_2}{k}\left(k-\omega_2^p(1-\omega_2)\right).
\end{equation}
It is clear that $\mathcal{C}$ vanishes at $\omega_2=1$ and $\mathcal{C}>0$ near that point.
Let $\tilde{\mathcal{C}}=k-\omega_2^p(1-\omega_2)$. Since
$
\frac{d\tilde{\mathcal{C}}}{d\omega_2}=\omega_2^{p-1}(\omega_2-p(1-\omega_2)),
$
we have
$$
\begin{tabular}{  c|c c  c  c c} 
\hline
$\omega_2$ & $0$ & $\left(0,\frac{p}{p+1}\right)$ &$\frac{p}{p+1}$& $\left(\frac{p}{p+1},1\right)$& $1$\\ 
\hline
$\frac{d\tilde{\mathcal{C}}}{d\omega_2}$ & $0$ & $<0$ & $0$ & $>0$ & 1\\ 
\hline
$\tilde{\mathcal{C}}$ & $k$ & \text{Decrease} & \text{Local Minimum}& \text{Increase}&$k$\\ 
\hline
\end{tabular}.
$$
Therefore, for an arbitrary $p$, inequality \eqref{small k} is satisfied if and only if $\tilde{\mathcal{C}}\left(\frac{p}{p+1}\right)<0$. Then there exists some $\omega_*\in \left(\frac{p}{p+1},1\right)$ such that $\tilde{\mathcal{C}}(\omega_*)=0$ and $\tilde{\mathcal{C}}(\omega_2)> 0$ in $\left(\omega_*,1\right)$. Since $\omega_2\leq 1$, that means for such an $\omega_*$, we must have $\mathcal{C}(\omega_*)=0$ and $\mathcal{C}(\omega_2)> 0$ in $\left(\omega_*,1\right)$.

The $\omega_2$-coordinate of the intersection point between $l_0$ and $l_1$ is $\frac{(d-1)(1+\delta)}{a+2b}$. By \eqref{more lower bound for p} and Remark \ref{scalr of spher basic ineq}, the root $\omega_*$ discussed above satisfies 
\begin{equation}
\label{thankyouomega*}
\omega_*>\frac{p}{p+1}\geq \frac{(d-1)(1+\delta)}{a+2b}>\frac{a-2b}{a+2b}
\end{equation}

We are ready to prove that $\hat{U}_{(\delta,p,k)}\subset U_\delta$. In other words, with our choice of $p$ and $k$ above, inequalities in the definition \eqref{yes} of $U_{(\delta,p,k)}$ and \eqref{hatyes} of $\hat{U}_{(\delta,p,k)}$ imply all inequalities in definition \eqref{U_0} of $U_0$ and \eqref{Udelta} of $U_\delta$.

Firstly, we need to show $\hat{U}_{(\delta,p,k)}\subset U_0$. With $-Z_1\geq Z_2-Z_3$ and $Z_2\geq \omega_*Z_3$ satisfied in $S_3$, we have
\begin{equation}
\begin{split}
\label{living in U0 1}
\mathcal{R}_1-\mathcal{R}_3 & =(Z_3-Z_1)(aZ_2-2bZ_1-2bZ_3)\\
%&\geq (Z_3-Z_1)((a+2b)Z_2-4bZ_3)\quad \text{since $Z_1+Z_2-Z_3\leq 0$}\\
&\geq  (Z_3-Z_1)((a+2b)Z_3\omega_*-4bZ_3)\\ %\quad \text{definition of $\hat{U}_{(\delta,p,k)}$}\\
&\geq (Z_3-Z_1)((a+2b)Z_3\omega_*-(a-2b)Z_3)\quad \text{Remark \ref{scalr of spher basic ineq}}\\
&\geq 0 \quad \text{by \eqref{thankyouomega*} and definition of $S_3$}
\end{split}
\end{equation}
and
\begin{equation}
\begin{split}
\label{living in U0 2}
\mathcal{R}_3-\mathcal{R}_2 & =(Z_3-Z_2)(2bZ_3+2bZ_2-aZ_1)\\
%&\geq (Z_3-Z_2)((a+2b)Z_2-(a-2b)Z_3)\quad \text{since $Z_1+Z_2-Z_3\leq 0$}\\
&\geq  (Z_3-Z_2)((a+2b)\omega_*Z_3-(a-2b)Z_3) \\%\quad \text{definition of $\hat{U}_{(\delta,p,k)}$}\\
%&= (Z_3-Z_2)Z_3((a+2b)\omega_*-(d-1))\quad \text{Remark \ref{scalr of spher basic ineq}}\\
&\geq 0 \quad \text{by \eqref{thankyouomega*} and definition of $S_3$}
\end{split}.
\end{equation}
Hence $\hat{U}_{(\delta,p,k)}\subset U_0$. 

In $\hat{U}_{(\delta,p,k)}$, we have 
\begin{equation}
\begin{split}
\mathcal{R}_3-(1+\delta)\mathcal{R}_2&=(Z_3-Z_2)(2b(Z_3+Z_2)-aZ_1)-\delta(aZ_1Z_3+b(Z_2^2-Z_1^2-Z_3^2))\\
&=2b(Z_3^2-Z_2^2)-aZ_1(Z_3-Z_2)-\delta a Z_1Z_3-\delta b Z_2^2+\delta b Z_1^2+\delta b Z_3^2\\
&=(2+\delta )bZ_3^2-(1+\delta)aZ_1Z_3+\delta bZ_1^2-(\delta b+2b) Z_2^2+aZ_1Z_2.
\end{split}
\end{equation}
Treat the result of the computation above as a function of $Z_3$. It is a parabola centered at $\frac{(1+\delta)a}{(2+\delta)2b}Z_1.$ By \eqref{thankyouomega*}, it is clear that $\frac{1}{1-\omega_*}>\frac{a+2b}{4b}$. Since $\delta\in\left(0,\frac{4b}{a-2b}\right)$, it is straightforward to deduce that $\frac{a+2b}{4b}>\frac{(1+\delta)a}{(2+\delta)2b}$. From $Z_2\geq \omega_*Z_3\geq \omega_*(Z_1+Z_2)$ we also deduce
\begin{equation}
\label{Z_1Z_2}
Z_2\geq \frac{\omega_*}{1-\omega*}Z_1.
\end{equation}
Therefore, we know that $
Z_1+Z_2\geq \frac{1}{1-\omega_*}Z_1\geq \frac{a+2b}{4b}Z_1\geq \frac{(1+\delta)a}{(2+\delta)2b}Z_1
$
in $\hat{U}_{(\delta,p,k)}$.
%
%we have 
%\begin{equation}
%\label{Z_1 over Z_2}
%\begin{split}
%\dfrac{Z_1}{Z_2}&\leq \frac{Z_3}{Z_2}-1\quad  \text{since $Z_1+Z_2-Z_3\leq 0$}\\
%&\leq \frac{1}{\omega_*}-1\quad \text{by \eqref{hatyes}}\\
%&< \frac{a+2b}{(d-1)(1+\delta)}-1\quad \text{by \eqref{thankyouomega*}}\\
%&= \frac{4b-\delta(d-1)}{(1+\delta)(d-1)}< \frac{4b+2\delta b}{(1+\delta)(d-1)}< \frac{4b+2\delta b}{a-4b+\delta(d-1)}=\frac{4b+2\delta b}{(1+\delta)a-(4b+2b\delta)}
%\end{split}.
%\end{equation}
Hence
\begin{equation}
\begin{split}
\mathcal{R}_3-(1+\delta)\mathcal{R}_2&\geq  \left.(\mathcal{R}_3-(1+\delta)\mathcal{R}_2)\right|_{Z_3=Z_1+Z_2}\\
&\geq 0 \quad\text{by \eqref{R3R2}}
\end{split}.
\end{equation}

Finally, we need to show that $\hat{U}_{(\delta,p,k)}$ is compact. Since  $Z_2-\omega_*Z_3\geq 0$, we automatically have
$
Z_1+Z_2\geq \omega_*Z_3
$
in $\hat{U}_{(\delta,p,k)}$. By \eqref{thankyouomega*}, we can deduce $\omega_*>\frac{a-2b}{a+2b}>\frac{2b}{a}$
in $\hat{U}_{(\delta,p,k)}$, where the last inequality is from Remark \ref{scalr of spher basic ineq}. Hence $a(Z_1+Z_2)-2b Z_3\geq 0$ in $\hat{U}_{(\delta,p,k)}$. Proposition \ref{sharp} can be applied and all $Z_j$'s are bounded above. By the conservation law \eqref{conservation law}, we know that all variables are bounded. Hence $\hat{U}_{(\delta,p,k)}$ is compact. 
The proof is complete.
\end{proof}

\begin{figure}[h]
\label{Fig 4}
\begin{center}
\includegraphics[width=3in]{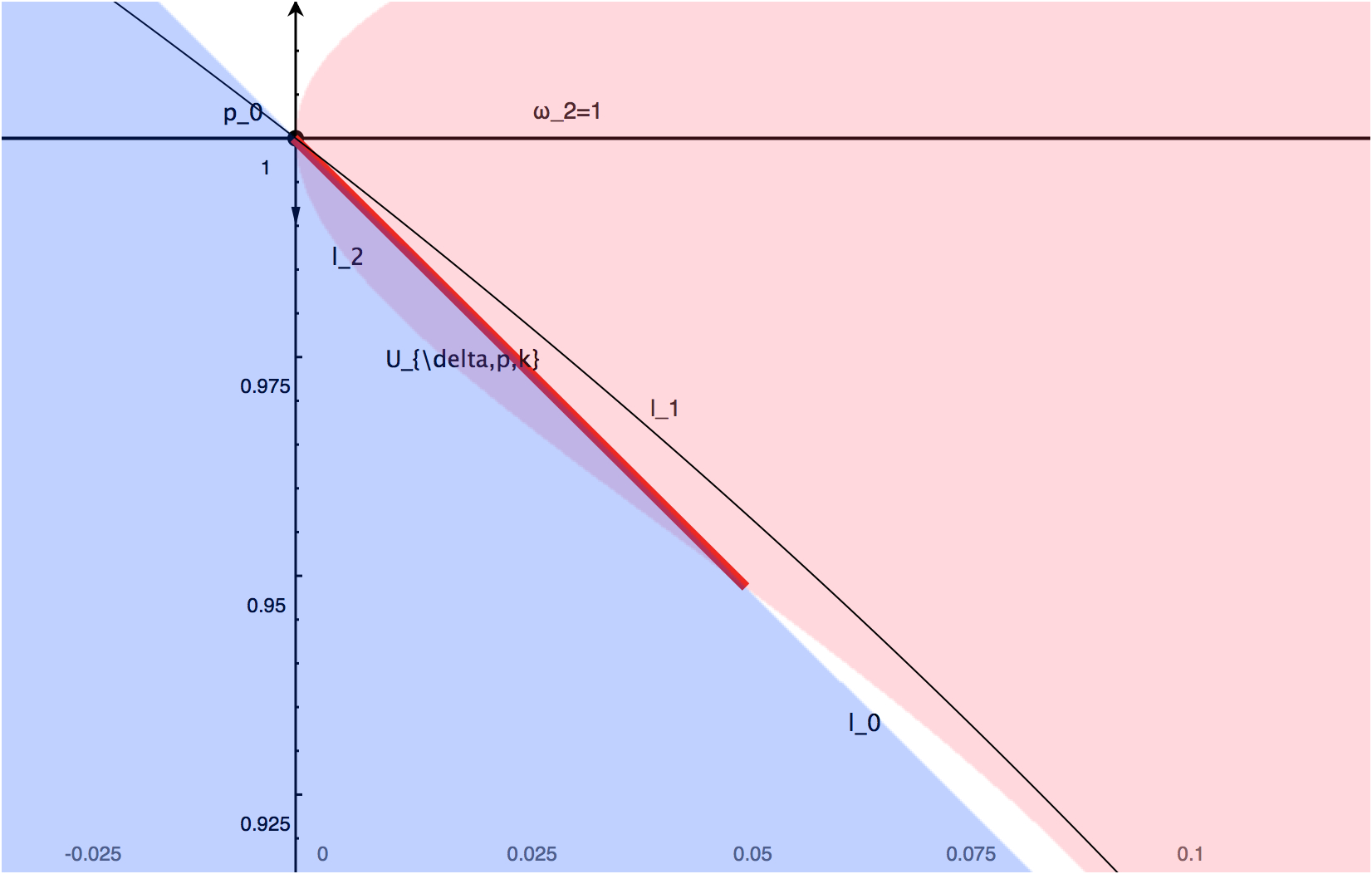}
\caption{$\delta=0.7$, $p=12$, $k=\frac{1}{13+1}\left(\frac{12}{13}\right)^{12}$ for Case I}
\end{center}\end{figure}

We are ready show that $\hat{U}_{(\delta,p,k)}$ is the entrance zone.
\begin{lemma}
\label{only through l0}
For $s_1\in\left(0,\sqrt{\frac{k(d+1)s_0}{16d}}\right)$ and suitable choice of $\delta,p$ and $k$ as described above, the integral curve $\gamma_{s_1}$ escapes $\hat{U}_{(\delta,p,k)}$ through $Z_1+Z_2-Z_3=0$.
\end{lemma}
\begin{proof}
Suppose $\gamma_{s_1}$ does not escape through $Z_1+Z_2-Z_3=0$, then it can only escape through either $kZ_1Z_3^{p+1}-Z_2^{p}(Z_3-Z_2)^2=0$ or $(Z_3-Z_2)(X_1-X_3)+(p(Z_3-Z_2)-2Z_2)(X_3-X_2)=0$. We prove that these situations are impossible.

Since
\begin{equation}
\label{The bump}
\begin{split}
&\left. \langle\nabla(kZ_1Z_3^{p+1}-Z_2^p(Z_3-Z_2)^2),V\rangle \right|_{kZ_1Z_3^{p+1}-Z_2^p(Z_3-Z_2)^2=0}\\
&=\left. \langle\nabla(Z_3^{p+2}(k\omega_1-\omega_2^p(1-\omega_2)^2)),V\rangle \right|_{kZ_1Z_3^{p+1}-Z_2^p(Z_3-Z_2)^2=0}\\
&=(p+2)Z_3^{p+2}\left(\mathcal{G}-\frac{1}{d}+2X_3\right)(k\omega_1-\omega_2^p(1-\omega_2)^2)\\
&\quad +Z_3^{p+2}(2k\omega_1(X_1-X_3)-2p\omega_2^p(X_2-X_3)(1-\omega_2)^2+4\omega_2^p(1-\omega_2)\omega_2(X_2-X_3))\\
&=Z_3^{p+2}(2k\omega_1(X_1-X_3)-2p\omega_2^p(X_2-X_3)(1-\omega_2)^2+4\omega_2^p(1-\omega_2)\omega_2(X_2-X_3))\\
&\quad \text{ since $kZ_1Z_3^{p+1}-Z_2^p(Z_3-Z_2)^2=0$}\\
&=Z_3^{p+2}(2\omega_2^p(1-\omega_2)^2(X_1-X_3)-2p\omega_2^p(X_2-X_3)(1-\omega_2)^2+4\omega_2^p(1-\omega_2)\omega_2(X_2-X_3))\\
&\quad \text{ since $kZ_1Z_3^{p+1}-Z_2^p(Z_3-Z_2)^2=0$}\\
%\\
%&=2Z_2^p(Z_3-Z_2)((Z_3-Z_2)(X_1-X_3)-p(X_2-X_3)(Z_3-Z_2)+2Z_2(X_2-X_3))\\
&=2Z_2^p(Z_3-Z_2)((Z_3-Z_2)(X_1-X_3)+(p(Z_3-Z_2)-2Z_2)(X_3-X_2))\\
&\geq 0 \quad \text{ definition of $\hat{U}_{(\delta,p,k)}$}
\end{split},
\end{equation}
it is impossible for $\gamma_{s_1}$ to escape $\hat{U}_{(\delta,p,k)}$ through $kZ_1Z_3^{p+1}-Z_2^p(Z_3-Z_2)^2=0$.

For the other defining inequality, we have
\begin{equation}
\label{Bump Derivative}
\begin{split}
&\left. \langle\nabla((Z_3-Z_2)(X_1-X_3)+(p(Z_3-Z_2)-2Z_2)(X_3-X_2)),V\rangle \right|_{(Z_3-Z_2)(X_1-X_3)+(p(Z_3-Z_2)-2Z_2)(X_3-X_2)=0}\\
&=\left. \langle\nabla(Z_3((1-\omega_2)(X_1-X_3)+(p(1-\omega_2)-2\omega_2)(X_3-X_2))),V\rangle \right|_{(Z_3-Z_2)(X_1-X_3)+(p(Z_3-Z_2)-2Z_2)(X_3-X_2)=0}\\
&=Z_3\left(\mathcal{G}-\frac{1}{d}+2X_3\right)((1-\omega_2)(X_1-X_3)+(p(1-\omega_2)-2\omega_2)(X_3-X_2)))\\
&\quad +Z_3((1-\omega_2)(X_1-X_3)+(p(1-\omega_2)-2\omega_2)(X_3-X_2)))(\mathcal{G}-1)\\
&\quad +Z_3((1-\omega_2)(\mathcal{R}_1-\mathcal{R}_3)+(p(1-\omega_2)-2\omega_2)(\mathcal{R}_3-\mathcal{R}_2))\\
&\quad +Z_3(2\omega_2(X_3-X_2)(X_1-X_3)+2(p+2)\omega_2(X_3-X_2)^2)\\
&=Z_3((1-\omega_2)(\mathcal{R}_1-\mathcal{R}_3)+(p(1-\omega_2)-2\omega_2)(\mathcal{R}_3-\mathcal{R}_2))\\
&\quad +Z_3(2\omega_2(X_3-X_2)(X_1-X_3)+2(p+2)\omega_2(X_3-X_2)^2)\\
&\quad \text{ since $(Z_3-Z_2)(X_1-X_3)+(p(Z_3-Z_2)-2Z_2)(X_3-X_2))=0$}\\
&=(Z_3-Z_2)(\mathcal{R}_1-\mathcal{R}_3)+(p(Z_3-Z_2)-2Z_2)(\mathcal{R}_3-\mathcal{R}_2)\\
&\quad +2Z_2(X_3-X_2)(X_1-X_3)+2(p+2)Z_2(X_3-X_2)^2\\
&=(Z_3-Z_2)(\mathcal{R}_1-\mathcal{R}_3+p(\mathcal{R}_3-\mathcal{R}_2))-2Z_2(Z_3-Z_2)(2bZ_3+2bZ_2-aZ_1)\\
&\quad +2Z_2(X_3-X_2)((X_1-X_3)+(p+2)(X_3-X_2))\\
&=(Z_3-Z_2)(\mathcal{R}_1-\mathcal{R}_3+p(\mathcal{R}_3-\mathcal{R}_2)+2Z_2(aZ_1-2bZ_2-2bZ_3))\\
&\quad +(Z_3-Z_2)((X_1-X_2)+(p-1)(X_3-X_2))(X_1-X_2+(p+1)(X_3-X_2))\\
&\quad \text{ since $2Z_2(X_3-X_2)=(Z_3-Z_2)(X_1-X_3)+p(Z_3-Z_2)(X_3-X_2)$}
\end{split}.
\end{equation}
Because $\hat{U}_{(\delta,p,k)}\subset U_\delta$, we can apply Proposition \ref{from X to scalar curvature!} to the last line of \eqref{Bump Derivative} and continue the computation as
\begin{equation}
\begin{split}
&\geq (Z_3-Z_2)\left(\mathcal{R}_1-\mathcal{R}_3+p(\mathcal{R}_3-\mathcal{R}_2)+2Z_2(aZ_1-2bZ_2-2bZ_3)+\frac{1-\mathcal{G}}{d(d-1)}\right) \\
&=(Z_3-Z_2)\left(\mathcal{R}_1-\mathcal{R}_3+p(\mathcal{R}_3-\mathcal{R}_2)+2Z_2(aZ_1-2bZ_2-2bZ_3)+\frac{1}{d-1}(\mathcal{R}_1+\mathcal{R}_2+\mathcal{R}_3)\right) \quad \text{by \eqref{conservation law}}\\
&=\left(2-\frac{1}{d-1}\right)bZ_1^2+\left(aZ_3((p+1)\omega_2-p)+\frac{a}{d}(Z_2+Z_3)\right)Z_1\\
&\quad +Z_3^2\left(-\left(2bp+4b+\frac{b}{d-1}\right)\omega_2^2+\left(
\frac{a}{d-1}+a-4b\right)\omega_2+\left(2pb-2b-\frac{b}{d-1}\right)\right).
\end{split}
\end{equation}
The first term of the computation result above is obviously positive. The second term is positive because $\omega_2\geq \omega_*>\frac{p}{p+1}$ in $\hat{U}_{(\delta,p,k)}$. The positivity of the last term depends on the one of parabola
$$
\pi(\omega_2)=-\left(2bp+4b+\frac{b}{d-1}\right)\omega_2^2+\left(
\frac{a}{d-1}+a-4b\right)\omega_2+\left(2pb-2b-\frac{b}{d-1}\right).
$$
Since we impose $p\geq 2$, it is clear that $\pi(0)$ is positive. As the coefficient of the first term is negative, we know that $\pi$ has two roots with different signs. It is easy to verify that $\pi(1)=0$. Then we conclude that $\pi$ is non-negative for all $\omega_2\in [0,1]$.
Therefore, the computation of \eqref{Bump Derivative} is non-negative and only vanishes when $Z_1=0$ and $Z_2=Z_3$.

Notice that there is no need to check the possibility that $\gamma_{s_1}$ may escape through $Z_2-\omega_*Z_3=0$. Because when the equality of $Z_2-\omega_*Z_3\geq 0$ is reached at some point $\gamma_{s_1}(\eta_*)$, it implies that the function $\mathcal{C}$ in \eqref{function C} vanishes at that point. Specifically, we have
$
1-\omega_2=\frac{\omega_2^p(1-\omega_2)^2}{k}
$
at that point. But then
$$
Z_1Z_3^{p+1}\leq Z_3^{p+1}(Z_3-Z_2)=\frac{Z_2^p(Z_3-Z_2)^2}{k}\leq Z_1Z_3^{p+1},
$$
which implies $kZ_1Z_3^{p+1}-Z_2^p(Z_3-Z_2)^2=0$ at that point and this case is included in the computation at the beginning of the proof.
\end{proof}

\begin{proposition}
\label{critical points in Uhat}
The only critical points in $\hat{U}_{(\delta,p,k)}$ are $p_0$ and those of Type I.
\end{proposition}
\begin{proof}
By proposition \ref{critical points}, it is clear that $p_0$ and critical points of Type I are in $\hat{U}_{(\delta,p,k)}$.  We first eliminate critical points with negative $Z_j$ entry. Since $\hat{U}_{(\delta,p,k)}\subset S_3$, we can eliminate critical points with $Z_3$ smaller than the other two $Z_j$'s. Because $X_3\geq 0$ in $S_3$, there is no critical points of Type II. Since $Z_2\geq pZ_1\geq Z_1$ in $\hat{U}_{(\delta,p,k)}$ by \eqref{thankyouomega*} and \eqref{Z_1Z_2}, there is no critical points other than $p_0$ and those of Type I in $\hat{U}_{(\delta,p,k)}$.
\end{proof}

\begin{proposition}
\label{Z1Z2Z3}
The function $Z_1Z_2Z_3$ stays positive and increases along $\gamma_{s_1}$.
\end{proposition}
\begin{proof}
Since $\mathcal{H}\equiv 1$, it is clear that $\mathcal{G}\geq \frac{1}{n}$. Hence
\begin{equation}
\label{increase}
(Z_1Z_2Z_3)'=Z_1Z_2Z_3\left(3\mathcal{G}-\frac{1}{d}\right)\geq 0.
\end{equation}
Since $Z_1Z_2Z_3$ is initially positive along $\gamma_{s_1}$, the proof is complete.
\end{proof}

We are ready to prove the completeness of Ricci-flat metrics represented by $\gamma_{s_1}$ with $s_1$ close enough to zero.
\begin{lemma}
\label{long existence}
There exists a $k>0$ such that an unstable integral curve $\gamma_{s_1}$ to \eqref{new Ricci-flat} on $C\cap \{\mathcal{H}\equiv 1\}$ emanating from $p_0$ is defined on $\mathbb{R}$ if $s_1\in \left(-\sqrt{\frac{k(d+1)s_0}{16d}},\sqrt{\frac{k(d+1)s_0}{16d}}\right)$.
\end{lemma}
\begin{proof}
If $s_1>0$, the curve $\gamma_{s_1}$ is initially trapped in $\hat{U}_{(\delta,p,k)}$ as long as $s_1\in\left(0,\sqrt{\frac{k(d+1)s_0}{16d}}\right)$. The function $Z_1+Z_2-Z_3$ vanishes at $p_0$ and it is negative along $\gamma_{s_1}$ in $\hat{U}_{(\delta,p,k)}$. By Lemma \ref{only through l0}, the function $Z_1+Z_2-Z_3$ must vanish at $\gamma_{s_1}(\eta_*)$ for some $\eta_*\in\mathbb{R}$. Then we must have  $(Z_1+Z_2-Z_3)'(\gamma_{s_1}(\eta_*))\geq 0$. But
\begin{equation}
\begin{split}
(Z_1+Z_2-Z_3)'(\gamma_{s_1}(\eta_*))&=\langle\nabla(Z_1+Z_2-Z_3),V\rangle(\gamma_{s_1}(\eta_*))\\
&=\left(\left.\langle\nabla(Z_1+Z_2-Z_3),V\rangle\right|_{Z_1+Z_2-Z_3=0}\right)(\gamma_{s_1}(\eta_*))\\
&=\left(Z_1(X_1-X_3)+Z_2(X_2-X_3)\right)(\gamma_{s_1}(\eta_*)).
\end{split}
\end{equation}
Hence $\gamma_{s_1}(\eta_*)$ is in $\partial \hat{S}_3.$

By Proposition \ref{verytechnicalthatitsiboring}, we know that $\hat{U}_{(\delta,p,k)}$ is in $U_0$, where $\mathcal{R}_1-\mathcal{R}_3\geq 0$ and $\mathcal{R}_3-\mathcal{R}_2\geq 0$ hold. Then with the similar argument in Proposition \ref{compare X_3 with X_2}, we know that $X_1>X_3>X_2$ along $\gamma_{s_1}$ in $\hat{U}_{(\delta,p,k)}$. Hence the intersection point $\gamma_{s_1}(\eta_*)$ is not $p_0$.
By Proposition \ref{Z1Z2Z3}, we know that $\gamma_{s_1}(\eta_*)$ cannot be a critical point of Type I. By Proposition \ref{critical points in Uhat}, we know that $\gamma_{s_1}(\eta_*)$ is not a critical point. Then by Lemma \ref{never escape}, $\gamma_{s_1}$ continue to flows inward $\hat{S}_3$ from $\gamma_{s_1}(\eta_*)$ and never escape. Therefore, such a $\gamma_{s_1}$ is defined on $\mathbb{R}$. 

By symmetry, similar result can be obtained for $s_1\in\left(-\sqrt{\frac{k(d+1)s_0}{16d}},0\right)$. If $s_1=0$, then we are back to the special case by Remark \ref{Mcontainy}. 
\end{proof}

By the discussion at the end of Section \ref{Local Existence}, Lemma \ref{long existence} proves the first half of Theorem \ref{main 1}.

\begin{remark}
\label{posiscaler}
For $\gamma_{s_1}$ with $s_1\in \left(0, \sqrt{\frac{k(d+1)s_0}{16d}}\right)$, it can be shown that $\mathcal{R}_2$ is negative initially by substituting \eqref{linearized solution<}. Hence the Ricci-flat metrics represented does not have the property introduced in Remark \ref{positivericci}. By straightforward computation, however, it processes a weaker condition that the scalar curvature of each hypersurface remain positive.
\end{remark}

\section{Asymptotic Limit}
\label{aymptoticlimit}
In this section, we study the asymptotic behavior of complete Ricci-flat metrics constructed above. Each integral curve $\gamma_{s_1}$ mentioned below satisfies the condition in Lemma \ref{long existence}, i.e., each $\gamma_{s_1}$ is trapped in $\hat{U}_{(\delta,p,k)}$ initially and then enter $\hat{S}_3$ in finite time.

\begin{lemma}
\label{omega_1+omega_2isnever1}
Let $\gamma_{s_1}$ be a long time existing integral curve that intersects with $\hat{S}_3$ at a non-critical point $\gamma_{s_1}(\eta_*)$. Then function $\omega_1+\omega_2>1$ along $\gamma_{s_1}(\eta)$ for $\eta\in(\eta_*,\infty)$.
\end{lemma}
\begin{proof}
Note that $(\omega_1+\omega_2)(\gamma_{s_1}(\eta_*))=1$. By Lemma \ref{never escape}, we know that $\gamma_{s_1}(\eta)\in \hat{S}_3$ for $\eta\geq \eta_*$. We have 
\begin{equation}
\begin{split}
(\omega_1+\omega_2)'(\gamma_{s_1}(\eta_*))&=(2\omega_1(X_1-X_3)+2\omega_2(X_2-X_3))(\gamma_{s_1}(\eta_*))\\
&\geq 0\quad\text{by definition of $\hat{S}_3$}.
\end{split}
\end{equation}
Suppose $(\omega_1+\omega_2)'(\gamma_{s_1}(\eta_*))=0$. Recall in the proofs of Lemma \ref{long existence}, we know that $X_1>X_3>X_2$ at $\gamma_{s_1}(\eta_*)$. By \eqref{messy!} and \eqref{usefulfor p3}, we have
\begin{equation}
\begin{split}
(\omega_1+\omega_2)''(\gamma_{s_1}(\eta_*))&\geq \left(4\omega_1(X_1-X_3)^2+4\omega_2(X_2-X_3)^2\right)(\gamma_{s_1}(\eta_*))>0
\end{split}.
\end{equation}
Suppose there exists $\eta_{1}\in(\eta_*,\infty)$ that $(\omega_1+\omega_1)(\gamma_{s_1}(\eta_{1}))=1$. We know from the computation above that there exists $\eta_{2}\in(\eta_{*},\eta_1)$ such that $(\omega_1+\omega_2)(\gamma_{s_1}(\eta_{2}))>1$. By mean value theorem, there exists $\eta_3\in[\eta_2,\eta_1]$ such that
$(\omega_1+\omega_2)'(\gamma_{s_1}(\eta_3))=(2\omega_1(X_1-X_3)+2\omega_2(X_2-X_3))(\gamma_{s_1}(\eta_3))<0,$ a contradiction to the definition of $\hat{S}_3$.
\end{proof}

\begin{lemma}
\label{upper X_3}
The variable $X_3$ is smaller than $\frac{1}{n}$ along integral curves $\gamma_{s_1}$.
\end{lemma}
\begin{proof}
Since $\mathcal{H}\equiv 1$, $X_3\leq \frac{1}{n}$ is equivalent to $X_1+X_2-2X_3\geq 0$. The function $X_1+X_2-2X_3$ is positive at $p_0$. Suppose the function vanishes along $\gamma_{s_1}$ at some point in $\hat{U}_{(\delta,p,k)}$, then we have
\begin{equation}
\label{X_3deri}
\begin{split}
\left.(X_1+X_2-2X_3)'\right|_{X_1+X_2-2X_3=0}&=(X_1+X_2-2X_3)(\mathcal{G}-1)+\mathcal{R}_1+\mathcal{R}_2-2\mathcal{R}_3\\
&=\mathcal{R}_1+\mathcal{R}_2-2\mathcal{R}_3 \quad \text{since $X_1+X_2-2X_3=0$}\\
&=a(Z_2Z_3+Z_1Z_3-2Z_1Z_2)-2b(2Z_3^2-Z_1^2-Z_2^2).
\end{split}
\end{equation}
Consider the computation result above as a function $$\mathcal{J}(Z_3)=-4bZ_3^2+a(Z_1+Z_2)Z_3+2bZ_1^2+2bZ_2^2-2aZ_1Z_2.$$

Since $Z_1+Z_2\leq  Z_3\leq \frac{Z_2}{\omega_*}$ in $\hat{U}_{(\delta,p,k)}$, the positivity of $\mathcal{J}$ is implied by those of $\mathcal{J}(Z_1+Z_2)$ and $\mathcal{J}\left(\frac{Z_2}{\omega_*}\right)$. With the choice $p\geq 2$, inequality \eqref{thankyouomega*} implies $\omega_*>\frac{p}{p+1}\geq \frac{2}{3}\geq \frac{4b}{a}$. Hence it is sufficient to prove a stronger condition: the positivity of $\mathcal{J}(Z_1+Z_2)$ and $\mathcal{J}\left(\frac{a}{4b}Z_2\right)$.
We have
\begin{equation}
\label{Joneend1}
\begin{split}
\mathcal{J}(Z_1+Z_2)&=(a-2b)(Z_1^2+Z_2^2)-8bZ_1Z_2\\
&\geq 4b(Z_1-Z_2)^2\quad\text{Remark \ref{scalr of spher basic ineq}}\\
&\geq 0
\end{split}.
\end{equation}
And we have
\begin{equation}
\label{Joneend2}
\begin{split}
\mathcal{J}\left(\frac{a}{4b}Z_2\right)&=\left(\frac{a^2}{4b}-2a\right)Z_1Z_2+2b(Z_1^2+Z_2^2)\\
&\geq \left(\frac{a^2}{4b}-2a\right)Z_1Z_2+4bZ_1Z_2\\
&\geq 0
\end{split}.
\end{equation}
All $Z_j$'s are positive along $\gamma_{s_1}$. Hence by \eqref{Joneend1} and \eqref{Joneend2}, computation \eqref{X_3deri} can vanish only if $Z_1=Z_2=\frac{Z_3}{2}$. But with $p\geq 2$ imposed, $Z_2\geq \omega_*Z_3\geq \frac{2}{3}Z_3\geq\frac{Z_3}{2}$ in $\hat{U}_{(\delta,p,k)}$. Hence $\mathcal{J}$ can only vanish at the origin of $Z$-space, which is impossible for $\gamma_{s_1}$ to reach by \eqref{increase}. Therefore, $X_1+X_2-2X_3$ never vanishes along $\gamma_{s_1}$ at least till $\gamma_{s_1}$ intersect with $\partial \hat{S}_3$ at some $\gamma_{s_1}(\eta_*)$.

$\gamma_{s_1}$ is in $\hat{S}_3$ for $\eta\in[\eta_*,\infty)$. The function $X_1+X_2-2X_3$ is positive at $\gamma_{s_1}(\eta_*)$. Suppose the function vanishes at some point along $\gamma_{s_1}$ in $\hat{S}_3$, then
\begin{equation}
\label{X3der}
\begin{split}
\left.(X_1+X_2-2X_3)'\right|_{X_1+X_2-2X_3=0}&=\mathcal{R}_1+\mathcal{R}_2-2\mathcal{R}_3\\
&=a(Z_2Z_3+Z_1Z_3-2Z_1Z_2)-2b(2Z_3^2-Z_1^2-Z_2^2)\\
&\geq a(Z_2Z_3+Z_1Z_3-2Z_1Z_2)-\dfrac{a}{3}(2Z_3^2-Z_1^2-Z_2^2)  \quad \text{Remark \ref{scalr of spher basic ineq}}\\
%&=\dfrac{a}{3}(3Z_2Z_3+3Z_1Z_3-6Z_1Z_2      -2Z_3^2+Z_1^2+Z_2^2)\\
%&\geq \dfrac{a}{3}(3Z_2Z_3+3Z_1Z_3-6Z_1Z_2-2Z_3^2+Z_1^2+Z_2^2-2(Z_1-Z_2)^2)\\
&\geq  \dfrac{a}{3}(Z_1+Z_2-Z_3)(2Z_3-Z_1-Z_2)\\
&\geq 0 \quad\text{definition of $\hat{S}_3$}
\end{split}.
\end{equation}
By Proposition \ref{Z1Z2Z3}, there is no need to consider the case where each $Z_j$ vanishes. For Case I-III, suppose computation above vanishes at some point on $\gamma_{s_1}$. Then one possibility is that $Z_1=Z_2=Z_3$ at that point. But then $X_3-X_j+\rho(Z_3-Z_j)=X_3-X_j=\frac{1}{n}-X_j\geq 0$ at that point by the definition of $\hat{S}_3$. Then we must have $X_j=\frac{1}{n}$ for each $j$. Hence the point must be the critical point $p_1$, a contradiction. For Case I in particular, there is an extra possibility where $Z_1=Z_2=\frac{Z_3}{2}$ at that point. It is ruled out by Lemma \ref{omega_1+omega_2isnever1}. Hence $X_3<\frac{1}{n}$ along $\gamma_{s_1}$ all the way.
\end{proof}
%It turns out that once $\gamma_{s_1}$ enter $\hat{S}_3$, its behavior is more clear. With Lemma \ref{upper X_3}, we are able to prove a stronger version of Lemma \ref{omega_1+omega_2isnever1}.
%\begin{lemma}
%\label{i=1 only}
%Let $\gamma_{s_1}$ be a long time existing integral curve intersecting with $\hat{S}_3$ at a non-critical point $\gamma_{s_1}(\eta_*)$. Then function $\omega_1+\omega_2$ is strictly increasing along $\gamma_{s_1}(\eta)$ for $\eta\in[\eta_*,\infty)$.
%\end{lemma}
%\begin{proof}
%In $\hat{S}_3$, for any $\eta\in(\eta_*,\infty)$, we have
%\begin{equation}
%\begin{split}
%(\omega_1+\omega_2)'(\gamma_{s_1}(\eta))&=(2\omega_1(X_1-X_3)+2\omega_2(X_2-X_3))(\gamma_{s_1}(\eta))\\
%&\geq 0\quad\text{by definition of $\hat{S}_3$}.
%\end{split}
%\end{equation}
%Suppose $(\omega_1+\omega_2)'(\gamma_{s_1}(\eta))=0$, then
%\begin{equation}
%\begin{split}
%(\omega_1+\omega_2)''(\gamma_{s_1}(\eta))&\geq \left(4\omega_1(X_1-X_3)^2+4\omega_2(X_2-X_3)^2\right) (\gamma_{s_1}(\eta))\quad \text{by \eqref{messy!} and \eqref{usefulfor p3}}\\
%&\geq 0
%\end{split}.
%\end{equation}
%We claim that $(\omega_1+\omega_2)''(\gamma_{s_1}(\eta))>0$. Suppose  $(\omega_1+\omega_2)''(\gamma_{s_1}(\eta))=0$, then we must have  $X_3=\frac{1}{n}$ at $\gamma_{s_1}(\eta)$, contradicting to Lemma \ref{upper X_3}. Hence we know that $\omega_1+\omega_2$ is strictly increasing along $\gamma_{s_1}(\eta)$ for $\eta\in[\eta_*,\infty)$.
%\end{proof}

We can now describe the asymptotic limit of $\gamma_{s_1}$.
\begin{lemma}
\label{ricci flat limit}
The integral curve $\gamma_{s_1}$ converges to $p_1$.
\end{lemma}
\begin{proof}
Since $\gamma_{s_1}$ does not hit any critical point in $\hat{U}_{(\delta,p,k)}$ by Lemma \ref{long existence},  we can focus on the behavior of the integral curve in the set $\hat{S}_3$.
By Proposition \ref{Z1Z2Z3}, we know that $Z_1Z_2Z_3$ converges to some positive number along $\gamma_{s_1}$. There exists a sequence $\{\eta_m\}$ such that $\lim\limits_{m\to \infty}\eta_m=\infty$ and $\lim\limits_{m\to \infty}\mathcal{G}=\frac{1}{n}$. Hence $\lim\limits_{m\to \infty}X_j(\eta_m)=\frac{1}{n}$ for each $j$. 
But then
\begin{equation}
\begin{split}
0&=\lim_{m\to\infty} (X_1+X_2-2X_3)'(\eta_m)\\
&=\lim_{m\to\infty}\left((X_1+X_2-2X_3)(\mathcal{G}-1)+\mathcal{R}_1+\mathcal{R}_2-2\mathcal{R}_3\right)(\eta_m)\\
&=\lim_{m\to\infty}\left(\mathcal{R}_1+\mathcal{R}_2-2\mathcal{R}_3\right)(\eta_m)\\
%&\geq \lim_{m\to\infty}\dfrac{a}{3}(Z_1+Z_2-Z_3)(2Z_3-Z_1-Z_2)(\eta_m)\quad \text{by \eqref{X3der}}\\
&\geq 0  \quad\text{by \eqref{X3der}}
\end{split}.
\end{equation}
Therefore, either $\lim\limits_{m\to \infty} Z_3(\omega_1+\omega_2-1)(\eta_m)=0$ or $\lim\limits_{m\to\infty}(2Z_3-Z_1-Z_2)(\eta_m)=0$. It is clear that $\lim\limits_{m\to\infty}Z_3(\eta_m)\neq 0$ as $Z_3\geq Z_1,Z_2$ in $S_3$ and $Z_1Z_2Z_3$ increases along $\gamma_{s_1}$. By Lemma \ref{omega_1+omega_2isnever1}, we know that $\lim\limits_{m\to\infty}(\omega_1+\omega_2-1)(\eta_m)\neq 0$. Hence $\lim\limits_{m\to\infty}(2Z_3-Z_1-Z_2)(\eta_m)=0$. Since $Z_3\geq Z_1,Z_2$ in $S_3$, we conclude that $\lim\limits_{m\to \infty}(Z_3-Z_1)(\eta_m)=\lim\limits_{m\to \infty}(Z_3-Z_2)(\eta_m)=0$. With \eqref{conservation law}, we conclude that $\lim\limits_{m\to\infty} \gamma_{s_1}(\eta_m)=p_1$. Hence $p_1$ is in the $\omega$-limit set of $\gamma_{s_1}$.

Consider $p_1=\left(\frac{1}{n},\frac{1}{n},\frac{1}{n},\alpha,\alpha,\alpha\right)$, where $\alpha=\frac{1}{n}\sqrt{\frac{n-1}{a-b}}$. By \eqref{linearization}, the linearization at $p_1$ is 
\begin{equation}
\label{linearizationat p1}
\mathcal{L}(p_1)=\begin{bmatrix}
\frac{5}{3n}-1&\frac{2}{3n}&\frac{2}{3n}&2b\alpha&(a-2b)\alpha&(a-2b)\alpha\\
\frac{2}{3n}&\frac{5}{3n}-1&\frac{2}{3n}&(a-2b)\alpha&2b\alpha &(a-2b)\alpha\\
\frac{2}{3n}&\frac{2}{3n}&\frac{5}{3n}-1&(a-2b)\alpha &(a-2b)\alpha &2b\alpha\\
\frac{5}{3}\alpha&-\frac{1}{3}\alpha&-\frac{1}{3}\alpha&0&0&0\\
-\frac{1}{3}\alpha&\frac{5}{3}\alpha&-\frac{1}{3}\alpha&0&0&0\\
-\frac{1}{3}\alpha&-\frac{1}{3}\alpha&\frac{5}{3}\alpha&0&0&0\\
\end{bmatrix}.
\end{equation}
Its eigenvalues and corresponding eigenvectors are
$$
\lambda_1=\frac{1}{n}-1,\quad \lambda_2=\lambda_3=\beta_1,\quad\lambda_4=\lambda_5=\beta_2,\quad \lambda_6=\frac{2}{n}.
$$
$$
v_1=\begin{bmatrix}
n-1\\
n-1\\
n-1\\
-n\alpha\\
-n\alpha\\
-n\alpha
\end{bmatrix},
v_2=\begin{bmatrix}
-\frac{\beta_1}{2\alpha}\\
\frac{\beta_1}{2\alpha}\\
0\\
-1\\
1\\
0
\end{bmatrix},v_3=\begin{bmatrix}
-\frac{\beta_1}{2\alpha}\\
0\\
\frac{\beta_1}{2\alpha}\\
-1\\
0\\
1
\end{bmatrix},
v_4=\begin{bmatrix}
-\frac{\beta_2}{2\alpha}\\
\frac{\beta_2}{2\alpha}\\
0\\
-1\\
1\\
0
\end{bmatrix}, 
v_5=\begin{bmatrix}
-\frac{\beta_2}{2\alpha}\\
0\\
\frac{\beta_2}{2\alpha}\\
-1\\
0\\
1
\end{bmatrix},
v_6=\begin{bmatrix}
2\\
2\\
2\\
n\alpha\\
n\alpha\\
n\alpha
\end{bmatrix},
$$
where 
$$
\beta_1=-\dfrac{n-1+\sqrt{(n-1)^2-8n^2\alpha^2(a-4b)}}{2n}<0,\quad \beta_2=-\dfrac{n-1-\sqrt{(n-1)^2-8n^2\alpha^2(a-4b)}}{2n}<0.
$$
Evaluate \eqref{normal vector} at $p_1$, it is clear that 
$ T_{p_1}(C\cap \{\mathcal{H}\equiv 1\})=\spann\{v_2,v_3,v_4,v_5\}.$
Critical point $p_1$ is a sink. Hence $\lim\limits_{\eta\to\infty}\gamma_{s_1}=p_1$.
\end{proof}

\begin{lemma}
\label{ACCCCCmetric}
Ricci-flat metrics represented by $\gamma_{s_1}$ are AC.
\end{lemma}
\begin{proof}
For each $j$, we have
\begin{equation}
\label{ACrate}
\begin{split}
\lim_{t\to \infty} \dot{f_j}=\lim_{\eta\to \infty}\frac{X_j}{\sqrt{Z_kZ_l}}=\sqrt{\dfrac{a-b}{n-1}}.
\end{split}
\end{equation}
Therefore by Definition \ref{AC}, the Ricci-flat metric represented by $\gamma_{s_1}$ has conical asymptotic limit $dt^2+t^2\dfrac{a-b}{n-1} Q.$
\end{proof}

Lemma \ref{ricci flat limit} and Lemma \ref{ACCCCCmetric} imply Theorem \ref{main 2}.

\section{Singular Ricci-flat Metrics}
\label{a digression}
This section is dedicated to singular Ricci-flat metrics. Note that critical points $p_1$ and $p_2$ can be viewed as integral curves defined on $\mathbb{R}$. They correspond to singular Ricci-flat metrics $g=dt+t^2\dfrac{a-b}{n-1} Q$. This is consistent with the fact that the Euclidean metric cone over a proper scaled homogeneous Einstein manifold is Ricci-flat. For Case I in particular, the normal Einstein metric on $G/K$ is strict nearly K\"ahler. Hence the metric cone represented by $p_1$ is the singular $G_2$ metric discovered in \cite{bryant1987metrics}. Note that functions $F_j$'s in \eqref{G2condition} do note vanish at $p_2$. Therefore, the Euclidean metric cone over the K\"ahler--Einstein metric has generic holonomy.

There are also singular Ricci-flat metrics represented by nontrivial integral curves. Recall Remark \ref{G2}, The cohomogeneity one $G_2$ condition is given by $F_j\equiv 0$ for each $j$. Eliminate $X_j$'s in the conservation law $C$ shows that
\begin{equation*}
\begin{split}
\blacktriangle= &C\cap \{\mathcal{H}\equiv 1\} \cap P\cap \{F_1\equiv F_2\equiv F_3\equiv 0\}\\
&=\{Z_1+Z_2+Z_3-1\equiv 0\}\cap P\cap \{F_1\equiv F_2\equiv F_3\equiv 0\}
\end{split}
\end{equation*}
is an invariant 2-dimensional plane with boundary. Its projection in $Z$-space is plotted in Figure \ref{Fig G2}. Black squares are critical poitns of Type II. Linearization at these points shows that they are sources. Furthermore, for any $\xi\in\mathbb{R}$, $\blacktriangle\cap \{Z_3(Z_1-Z_2)-\xi Z_2(Z_1-Z_3)\equiv 0\}$ is a pair of integral curves that connects three critical points. If $\xi\neq 0,1$, then these two integral curves connect $p_1$ with two distinct critical points of Type II. These integral curves represent singular cohomogeneity one $G_2$ metrics on $(0,\infty)\times G/K$ that do not have smooth extension to $G/H$\cite{cvetivc2002cohomogeneity}\cite{cleyton_cohomogeneity-one_2002}. They all share the same AC limit as the metric cone over $G/K$ equipped with the normal Einstein metric.

When $\xi=0,1$, then one of the integral curve connects a critical point of Type II with $p_1$ and the other one connects a critical point of Type III with $p_1$. In particular, if $\xi=1$, then we recover $\gamma_0$ that represents the $G_2$ metric, connecting $p_0$ and $p_1$.

\begin{figure}[h!]
\begin{center}
\includegraphics[width=6.3in]{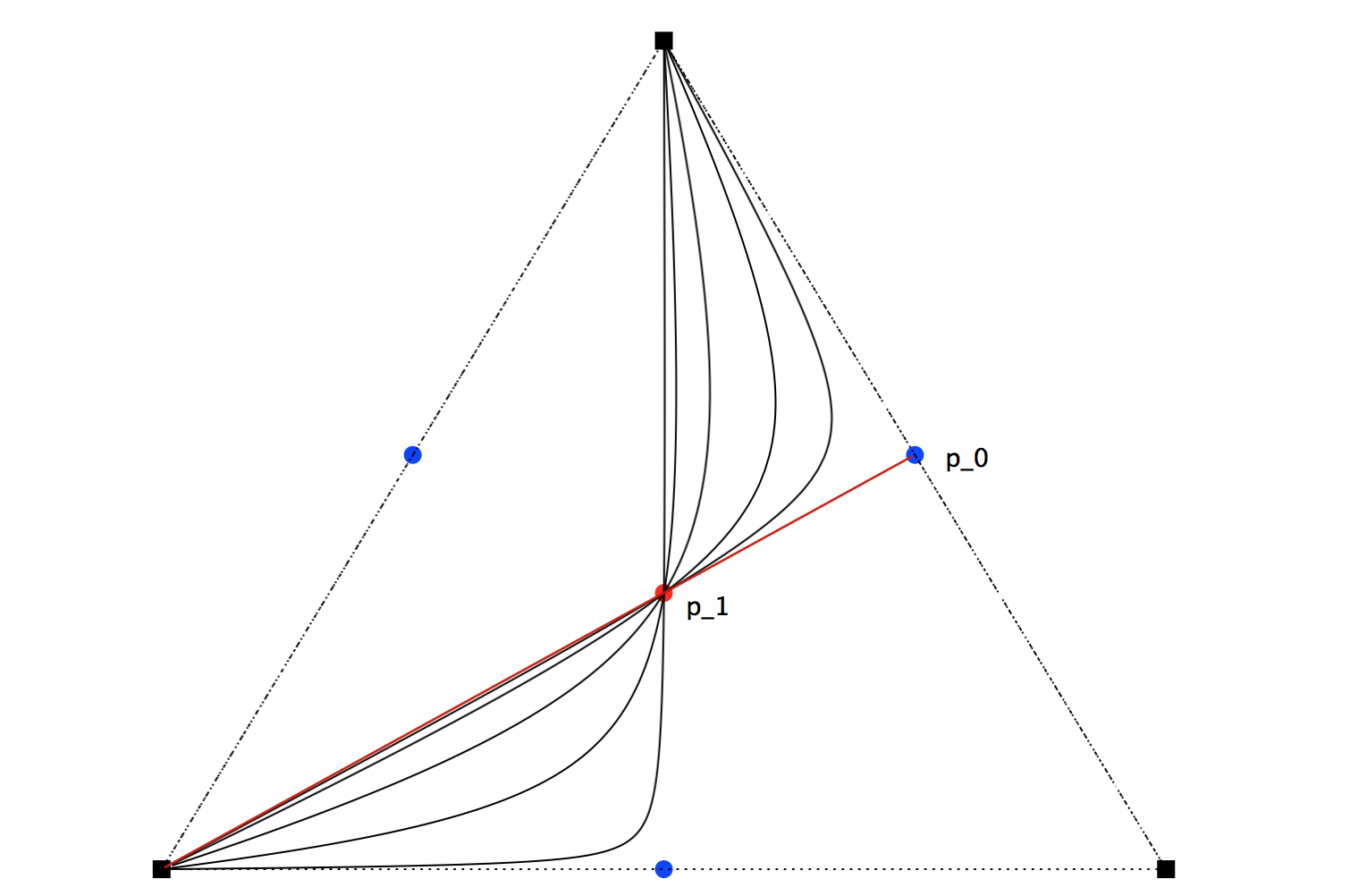}
\caption{Integral curves on $\blacktriangle$ with $0<\xi\leq 1$}
\label{Fig G2}
\end{center}
\end{figure}

%Note that since all $F_j$'s are identically zero, an integral curve on $\blacktriangle$ are completely determined by its projection on $Z$-space. Furthermore, $S_3\cap \blacktriangle$ is an invariant set. Hence for each $\xi\in\mathbb{R}\in(-\infty,1]$, $\blacktriangle\cap \{Z_3(Z_1-Z_2)-\xi Z_2(Z_1-Z_3)\equiv 0\}$ can be projected to the $\omega_1\omega_2$-plane.
%We have the following plot:
%\begin{figure}[h!]
%\begin{subfigure}{.5\textwidth}
%  \centering
%  \includegraphics[width=1\linewidth]{Xi01.png}
%  \caption{$\xi\in[0,1]$}
%  \label{fig:sfig2-0}
%\end{subfigure}
%\begin{subfigure}{.5\textwidth}
%  \centering
%  \includegraphics[width=1\linewidth]{Xiinfty0.png}
%  \caption{$\xi\in(-\infty,0)$}
%  \label{fig:sfig2-1}
%\end{subfigure}
%\caption{Projection of $U_0$ (enclosed by bold line segments) on $\omega_1\omega_2$-plane for all three cases}
%\label{Fig 2}
%\end{figure}

There are singular metrics with generic holonomy. We construct a new compact invariant set whose boundary includes $p_1$ and $p_2$. Consider 
$$\check{S}_3=S_3\cap \{X_1\equiv X_2, Z_1\equiv Z_2\}\cap\{X_1+X_2-2X_3\geq 0\}\cap \left\{(d-1)^2Z_1Z_2-4b^2Z_3^2\geq0\right\}.$$
\begin{proposition}
$\check{S}_3$ is a compact invariant set.
\end{proposition}
\begin{proof}
It is easy to show that $\{X_1\equiv X_2, Z_1\equiv Z_2\}$ is flow invariant. In fact, even if we define $\check{S}_3$ without $\{X_1\equiv X_2, Z_1\equiv Z_2\}$, the set is still compact and invariant. However, considering the subsystem does make the computation easier.

In $\check{S}_3$, we have
\begin{equation}
\begin{split}
4b^2Z_3^2&\leq (d-1)^2Z_1Z_2<a^2Z_1Z_2\leq a^2(Z_1+Z_2)^2.
\end{split}
\end{equation}
Hence we can apply Proposition \ref{sharp} and conclude that inequality \eqref{sharp Z_1+Z_2} holds in $\check{S}_3$. As $Z_3$ is bounded above by $\frac{d-1}{2b}\sqrt{Z_1Z_2}$ in $\check{S}_3$, the compactness follows by \eqref{conservation law}.

To show that $\check{S}_3$ is invariant, 
consider
\begin{equation*}
\begin{split}
\left. \langle\nabla(X_1+X_2-2X_3),V\rangle \right|_{X_1+X_2-2X_3=0}&=(X_1+X_2-2X_3)(\mathcal{G}-1)+\mathcal{R}_1+\mathcal{R}_2-2\mathcal{R}_3\\
%&=\mathcal{R}_1+\mathcal{R}_2-2\mathcal{R}_3 \quad \text{since $X_1+X_2-2X_3=0$}\\
&=2\mathcal{R}_2-2\mathcal{R}_3\quad \text{ since $Z_1\equiv Z_2$ in $\check{S}_3$ and $X_1+X_2-2X_3=0$}\\
%&=2(Z_3-Z_2)(aZ_1-2b(Z_2+Z_3))\\
&=2(Z_3-Z_2)((d-1)\sqrt{Z_1Z_2}-2bZ_3)\quad \text{ since $Z_1\equiv Z_2$ in $\check{S}_3$}\\
&\geq 0 \quad \text{by definition of $\check{S}_3$}
\end{split}.
\end{equation*}
Moreover, we have 
\begin{equation*}
\begin{split}
&\left. \langle\nabla((d-1)^2Z_1Z_2-4b^2Z_3^2),V\rangle \right|_{(d-1)^2Z_1Z_2-4b^2Z_3^2=0}\\
&=\left.\nabla\left(Z_3^2\left((d-1)^2\dfrac{Z_1Z_2}{Z_3^2}-4b^2\right)\right)\right|_{(d-1)^2Z_1Z_2-4b^2Z_3^2=0}\\
&= \left((d-1)^2Z_1Z_2-4b^2Z_3\right)\left(\mathcal{G}-\frac{1}{d}+2X_3\right)+2(d-1)^2Z_1Z_2\left(X_1+X_2-2X_3\right)\\
&=2(d-1)^2Z_1Z_2\left(X_1+X_2-2X_3\right)\quad \text{ since $(d-1)^2Z_1Z_2-4b^2Z_3^2=0$}\\
&\geq 0 \quad \text{ by definition of $\check{S}_3$}
\end{split}.
\end{equation*}
Hence $\check{S}_3$ is a compact invariant set.
\end{proof}

\begin{lemma}
\label{singularcurve}
There exists an integral curve $\Gamma$ defined on $\mathbb{R}$ emanating from $p_2$ in $\check{S}_3$.
\end{lemma}
\begin{proof}
Consider $p_2=\left(\frac{1}{n},\frac{1}{n},\frac{1}{n},\frac{2}{n}\sqrt{\frac{(n-1)b}{(d-1)(a+2b)}},\frac{2}{n}\sqrt{\frac{(n-1)b}{(d-1)(a+2b)}},\frac{1}{n}\sqrt{\frac{(n-1)(d-1)}{b(a+2b)}}\right)$. For simplicity, denote $Z_*=\frac{2}{n}\sqrt{\frac{(n-1)b}{(d-1)(a+2b)}}$. 
The linearization at $p_2$ is
\begin{equation}
\mathcal{L}(p_2)=\begin{bmatrix}
\frac{5}{9d}-1&\frac{2}{9d}&\frac{2}{9d}&2bZ_*&\left(\frac{a(d-1)}{2b}-2b\right)Z_*&2bZ_*\\
\frac{2}{9d}&\frac{5}{9d}-1&\frac{2}{9d}&\left(\frac{a(d-1)}{2b}-2b\right)Z_*&2bZ_*&2bZ_*\\
\frac{2}{9d}&\frac{2}{9d}&\frac{5}{9d}-1&(d-1)Z_*&(d-1)Z_*&(d-1)Z_*\\
\frac{5}{3}Z_*&-\frac{1}{3}Z_*&-\frac{1}{3}Z_*&0&0&0\\
-\frac{1}{3}Z_*&\frac{5}{3}Z_*&-\frac{1}{3}Z_*&0&0&0\\
-\frac{d-1}{6b}Z_*&-\frac{d-1}{6b}Z_*&-\frac{5(d-1)}{6b}Z_*&0&0&0\\
\end{bmatrix}
\end{equation}
Straightforward computation shows that for all cases, $L(p_2)$ is a hyperbolic critical point that has only one unstable eigenvalues with the corresponding eigenvector as
$$
\check{\lambda}=\dfrac{1}{2n}\left(\sqrt{(n-1)^2+96n(d-1)(a-4b)Z_*^2}-(n-1)\right),\quad 
\check{v}=\begin{bmatrix}
b\check{\lambda}\\
b\check{\lambda}\\
-2b\check{\lambda}\\
2bZ_*\\
2bZ_*\\
-2(d-1)Z_*
\end{bmatrix}
$$
Evaluate \eqref{normal vector} at $p_2$, it is clear that $\check{v}$ are tangent to $C\cap \{\mathcal{H}\equiv 1\}$. Fix $\check{s}_0>0$, there exists a unique trajectory $\Gamma$ emanating from $p_2$ with
$
\Gamma\sim p_2+\check{s}_0e^{\check{\lambda}\eta}\check{v}.
$

It is easy to check that $p_2\in\partial\check{S}_3$ with only $X_1+X_2-2X_3$ and $(d-1)^2Z_1Z_2-4b^2Z_3^2$
vanished at $p_2$. By straightforward computation, we know that $\Gamma$ is trapped in $\check{S}_3$ initially. The integral curve is hence defined on $\mathbb{R}$. Functions $f_j$'s that correspond to solutions $\Gamma$ are defined on $[0,\infty)$. 
\end{proof}

\begin{lemma}
\label{limitGamma}
The integral curve $\Gamma$ converges to $p_1$.
\end{lemma}
\begin{proof}
Since $\check{S}_3$ is a compact invariant set with $X_3\leq \frac{1}{n}$. Arguments in Lemma \ref{ricci flat limit} and  \ref{ACCCCCmetric} carry over. Hence for $\Gamma$, we have 
$\lim\limits_{\eta\to\infty}\Gamma=p_1$. 
\end{proof}

For each $j$, we have
\begin{equation}
\label{ACpointy}
\begin{split}
&\lim_{t\to 0} \dot{f_j}=\lim_{\eta\to -\infty}\frac{X_j}{\sqrt{Z_kZ_3}}=\frac{\check{\lambda}}{2Z_*} \quad j,k\in\{1,2\}\\
&\lim_{t\to 0} \dot{f_3}=\lim_{\eta\to -\infty}\frac{X_3}{\sqrt{Z_1Z_2}}=\frac{b\check{\lambda}}{(d-1)Z_*}
\end{split}.
\end{equation}
Hence $f_1=f_2\sim \frac{\check{\lambda}}{2Z_*}t$ and $f_3\sim \frac{b\check{\lambda}}{(d-1)Z_*}t$ as $t\to 0$. Since $\lim\limits_{t\to \infty} \dot{f_1}\neq \lim\limits_{t\to \infty} \dot{f_3}$, $\Gamma$ represents a singular metric whose end at $t\to 0$ is a conical singularity as a metric cone over the alternative Einstein metrics. Lemma \ref{singularcurve} and Lemma \ref{limitGamma} then prove the following theorem.

\begin{theorem}
\label{pointy1}
Up to homothety, there exists a unique singular Ricci-flat metric on $(0,\infty)\times G/K$ that at the end with $t\to 0$, it admits conical singularity as the metric cone over $G/K$ with alternative Einstein metric. It has an AC limit at the end with $t\to \infty$ as the metric cone over $G/K$ with normal Einstein metric.
\end{theorem}

%\section{Plot}
%\label{plot}
%By \eqref{s1ish1} and Lemma \ref{long existence}, we prove that with a $k>0$ small enough, cohomogeneity one metrics \eqref{Ricci-flat metric on M} on $M$ with initial condition $\lim\limits_{t\to 0}(f_1,f_2,f_3,\dot{f}_1,\dot{f}_2,\dot{f}_3)=(0,h_0,h_0,1,h_{1,2},h_{1,3})$ with $h_{1,3}-h_{1,2}\in \left(-\frac{k}{2},\frac{k}{2}\right)$ are complete. In addition, if $h_{1,3}+h_{1,2}=0$, then the metrics are smooth and complete. By Remark \ref{s1ish1}, all admissible initial conditions are demonstrated below by the grey area and the bold dash line. The bold line included in the grey area is the smooth initial condition. The bold dash line is the initial condition for the subsystem with $f_2\equiv f_3$.
%
%\begin{figure}[h!]
%\begin{center}
%\includegraphics[width=5.5in]{initial.png}
%\label{Fig 1}
%\end{center}\end{figure}
Results of this article can be summarized by the plot in the following page. It shows the projection of integral curves to \eqref{new Ricci-flat} on the $Z$-space for Case I. It is computed by MATLAB using the 4th order Runge--Kutta method.

\begin{table}[h!]
\centering
\label{Lambda}
\begin{tabular}{l l} 
\hline
Integral Curves &\enskip Metric Type \\
\hline
$\gamma_{0}$    & 
$\begin{array}{l}
\text{Smooth metric with vanished principal curvatures on } G/H;\\
\text{Nonsmooth metric with non-vanishing mean curvatures on } G/H
\end{array}$\\
\hline
$\gamma_{s_1}, s_1\neq 0$  &
$\begin{array}{l}
\text{Smooth metrics with non-zero principal curvatures on } G/H \\
\text{Nonsmooth metrics with non-vanishing mean curvatures on } G/H
\end{array}$\\
\hline
$\Gamma$        &
$\begin{array}{l}
\text{Conical Singularity as alternative Einstein metric on $G/K$};
\end{array}$\\
\hline
\end{tabular}.
\end{table}
\begin{figure}[h!]
\begin{center}
\includegraphics[width=6.3in]{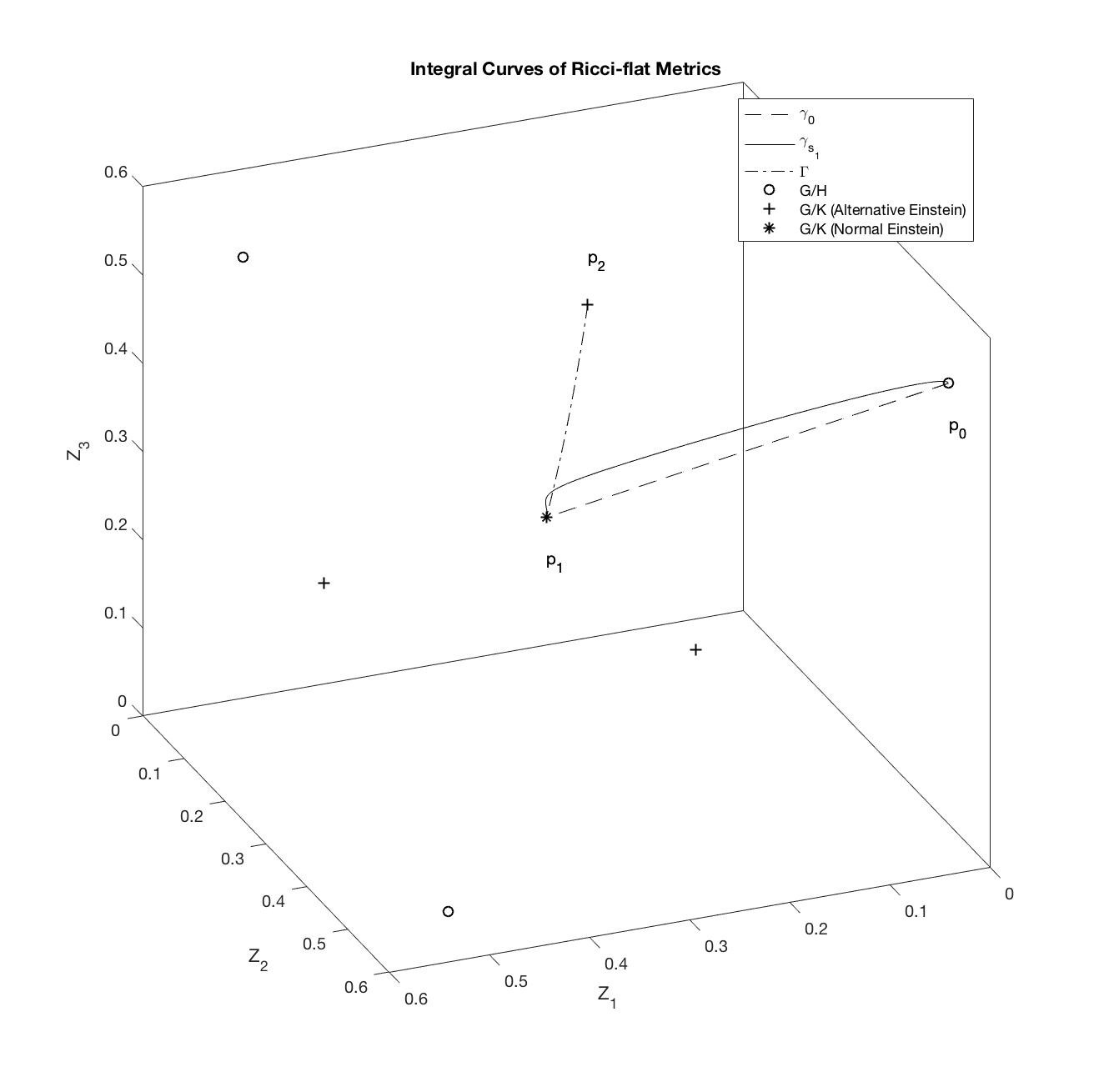}
\label{Fig 1}
\end{center}\end{figure}

\newpage

\bibliography{C1EMoW}

\begin{thebibliography}{10}

\bibitem{back1986local}
A.~Back.
\newblock Local theory of equivariant {E}instein metrics and {R}icci
  realizability on {K}ervaire spheres.
\newblock {\em Preprint}, 1986.

\bibitem{bergery1982nouvelles}
L.~B\'{e}rard-Bergery.
\newblock Sur de nouvelles vari\'{e}t\'{e}s riemanniennes d'{E}instein.
\newblock 6:1--60, 1982.

\bibitem{besse2007einstein}
A.~L. Besse.
\newblock {\em Einstein manifolds}.
\newblock Classics in Mathematics. Springer-Verlag, Berlin, 2008.
\newblock Reprint of the 1987 edition.

\bibitem{bohm_inhomogeneous_1998}
C.~B{\"o}hm.
\newblock Inhomogeneous {Einstein} metrics on low-dimensional spheres and other
  low-dimensional spaces.
\newblock {\em Inventiones Mathematicae}, 134(1):145--176, September 1998.

\bibitem{brown1972riemannian}
R.~B. Brown and A.~Gray.
\newblock Riemannian manifolds with holonomy group {${\rm Spin}$}(9).
\newblock pages 41--59, 1972.

\bibitem{bryant1987metrics}
R.~L. Bryant.
\newblock Metrics with exceptional holonomy.
\newblock {\em Annals of Mathematics}, pages 525--576, 1987.

\bibitem{bryant1989construction}
R.~L. Bryant and S.~M. Salamon.
\newblock On the construction of some complete metrics with exceptional
  holonomy.
\newblock {\em Duke Math. J.}, 58(3):829--850, 1989.

\bibitem{buzano_non-kahler_2015}
M.~Buzano, A.~S. Dancer, M.~Gallaugher, and M.~Y. Wang.
\newblock Non-{K{\"a}hler} expanding {Ricci} solitons, {Einstein} metrics, and
  exotic cone structures.
\newblock {\em Pacific Journal of Mathematics}, 273(2):369--394, January 2015.

\bibitem{buzano2015family}
M.~Buzano, A.~S. Dancer, and M.~Wang.
\newblock A family of steady {R}icci solitons and {R}icci flat metrics.
\newblock {\em Comm. Anal. Geom.}, 23(3):611--638, 2015.

\bibitem{calabi1975construction}
E.~Calabi.
\newblock A construction of nonhomogeneous {E}instein metrics.
\newblock In {\em Differential geometry ({P}roc. {S}ympos. {P}ure {M}ath.,
  {V}ol. {XXVII}, {S}tanford {U}niv., {S}tanford, {C}alif., 1973), {P}art 2},
  pages 17--24. Amer. Math. Soc., Providence, R.I., 1975.

\bibitem{calabi1979metriques}
E.~Calabi.
\newblock M{\'e}triques {K}{\"a}hl{\'e}riennes et fibr{\'e}s holomorphes.
\newblock In {\em Annales Scientifiques de l'{\'E}cole Normale Sup{\'e}rieure},
  volume~12, pages 269--294. Elsevier, 1979.

\bibitem{lopez2009canonical}
M.~Castrill\'{o}n~L\'{o}pez, P.~M. Gadea, and I.~V. Mykytyuk.
\newblock The canonical eight-form on manifolds with holonomy group {${\rm
  Spin(9)}$}.
\newblock {\em Int. J. Geom. Methods Mod. Phys.}, 7(7):1159--1183, 2010.

\bibitem{chen2011examples}
D.~Chen.
\newblock Examples of einstein manifolds in odd dimensions.
\newblock {\em Annals of Global Analysis and Geometry}, 40(3):339, 2011.

\bibitem{cleyton_cohomogeneity-one_2002}
R.~Cleyton and A.~Swann.
\newblock Cohomogeneity one ${G}_2$-structures.
\newblock {\em Journal of Geometry and Physics}, 44(2-3):202--220, December
  2002.

\bibitem{coddington1955theory}
E.~A. Coddington and N.~Levinson.
\newblock {\em Theory of ordinary differential equations}.
\newblock McGraw-Hill Book Company, Inc., New York-Toronto-London, 1955.

\bibitem{cvetivc2002cohomogeneity}
M.~Cveti{\v{c}}, G.~W. Gibbons, H.~L{\"u}, and C.~N. Pope.
\newblock Cohomogeneity one manifolds of {S}pin(7) and ${G}_2$ holonomy.
\newblock {\em Physical Review D}, 65(10):106004, 2002.

\bibitem{cvetivc2004new}
M.~Cveti\v{c}, G.~W. Gibbons, H.~L\"{u}, and C.~N. Pope.
\newblock New cohomogeneity one metrics with {S}pin(7) holonomy.
\newblock {\em J. Geom. Phys.}, 49(3-4):350--365, 2004.

\bibitem{KEcoho1}
A.~S. Dancer and M.~Y. Wang.
\newblock K\"{a}hler-{E}instein metrics of cohomogeneity one.
\newblock {\em Math. Ann.}, 312(3):503--526, 1998.

\bibitem{dancer2008non}
A.~S. Dancer and M.~Y. Wang.
\newblock Non-{K}\"{a}hler expanding {R}icci solitons.
\newblock {\em Int. Math. Res. Not. IMRN}, (6):1107--1133, 2009.

\bibitem{dancer2008some}
A.~S. Dancer and M.~Y. Wang.
\newblock Some new examples of non-{K}\"{a}hler {R}icci solitons.
\newblock {\em Math. Res. Lett.}, 16(2):349--363, 2009.

\bibitem{davila_homogeneous_2012}
J.~C.~González Dávila and F.~Martín Cabrera.
\newblock Homogeneous nearly {K{\"a}hler} manifolds.
\newblock {\em Annals of Global Analysis and Geometry}, 42(2):147--170, August
  2012.

\bibitem{eguchi1979self}
T.~Eguchi and A.~J. Hanson.
\newblock Self-dual solutions to {E}uclidean gravity.
\newblock {\em Annals of Physics}, 120(1):82--106, 1979.

\bibitem{eschenburg2000initial}
J.-H. Eschenburg and M.~Y. Wang.
\newblock The initial value problem for cohomogeneity one {E}instein metrics.
\newblock {\em J. Geom. Anal.}, 10(1):109--137, 2000.

\bibitem{foscolo_infinitely_2018}
L.~Foscolo, M.~Haskins, and J.~Nordström.
\newblock Infinitely many new families of complete cohomogeneity one
  ${G}_2$-manifolds: ${G}_2$ analogues of the {Taub}-{NUT} and
  {Eguchi}-{Hanson} spaces.
\newblock {\em arXiv:1805.02612 [hep-th]}, May 2018.
\newblock arXiv: 1805.02612.

\bibitem{gibbons1990Einstein}
G.~W. Gibbons, D.~N. Page, and C.~N. Pope.
\newblock Einstein metrics on {$S^3,\;{\bf R}^3$} and {${\bf R}^4$} bundles.
\newblock {\em Comm. Math. Phys.}, 127(3):529--553, 1990.

\bibitem{harvey_spinors_1990}
F.~R. Harvey.
\newblock {\em Spinors and {Calibrations}}.
\newblock Perspectives in mathematics. Elsevier Science, 1990.

\bibitem{hsianglawson}
W.~Hsiang and H.~B. Lawson, Jr.
\newblock Minimal submanifolds of low cohomogeneity.
\newblock {\em J. Differential Geometry}, 5:1--38, 1971.

\bibitem{nikonorov2014classification}
Y.~G. Nikonorov.
\newblock Classification of generalized {W}allach spaces.
\newblock {\em Geom. Dedicata}, 181:193--212, 2016.

\bibitem{salamon1989riemannian}
S.~M. Salamon.
\newblock {\em Riemannian geometry and holonomy groups}, volume 201 of {\em
  Pitman Research Notes in Mathematics Series}.
\newblock Longman Scientific \& Technical, Harlow; copublished in the United
  States with John Wiley \& Sons, Inc., New York, 1989.

\bibitem{wallach1972compact}
N.~R. Wallach.
\newblock Compact homogeneous {R}iemannian manifolds with strictly positive
  curvature.
\newblock {\em Ann. of Math. (2)}, 96:277--295, 1972.

\bibitem{wang1998Einstein}
J.~Wang and M.~Y. Wang.
\newblock Einstein metrics on ${S}^2$-bundles.
\newblock {\em Mathematische Annalen}, 310(3):497--526, 1998.

\bibitem{wang1991preserving}
M.~Y. Wang.
\newblock Preserving parallel spinors under metric deformations.
\newblock {\em Indiana Univ. Math. J.}, 40(3):815--844, 1991.

\bibitem{wink2017cohomogeneity}
M.~Wink.
\newblock Cohomogeneity one {R}icci solitons from {H}opf fibrations.
\newblock {\em arXiv preprint arXiv:1706.09712}, 2017.

\end{thebibliography}
\bibliographystyle{plain}
\end{document}